\documentclass[12pt,a4paper,twoside,reqno]{amsart}
\usepackage[a4paper,top=3cm,bottom=3cm,inner=2.5cm,outer=2.5cm]{geometry}

\usepackage[utf8]{inputenc}
\usepackage{amsmath,amssymb,amsthm,amsfonts}
\usepackage[english]{babel}
\usepackage{multirow,enumitem,array,longtable}
\usepackage{xparse}
\usepackage{pst-node,placeins,xspace,fix-cm}
\usepackage{tikz-cd}
\usepackage{accents}
\usepackage[nomessages]{fp}
\usepackage[colorlinks,linkcolor=blue,citecolor=violet,urlcolor=darkgray,final,hyperindex,linktoc=page,hyperfootnotes=true]{hyperref}
\newcolumntype{F}{>{$}c<{\hspace{-0.9ex}$}}
\newcolumntype{:}{>{$}m{0.8ex}<{$}}
\newcolumntype{R}{>{$}r<{$}}
\newcolumntype{C}{>{$}c<{$}}
\newcolumntype{L}{>{$}l<{$}}
\newcolumntype{N}{@{}>{$}l<{$}}
\newlength\horspace
\newcommand{\h}[1][1.0]{\hspace{#1\horspace}}
\newlength\verspace
\renewcommand{\v}[1][1.0]{\vspace{#1\verspace}\xspace}
\newlength\negverspace
\makeatletter
\newcommand\semiHuge{\@setfontsize\semiHuge{22.72}{27.38}}
\newcommand\semiLarge{\@setfontsize\semiLarge{12}{14}}
\makeatother
\tikzset{iso/.style={draw=none,every to/.append style={edge node={node [sloped, allow upside down, auto=false]{$\cong$}}}}}
\tikzset{adjunction/.style={draw=none,every to/.append style={edge node={node [sloped, allow upside down, auto=false]{$\dashv$}}}}}
\tikzset{simeq/.style={draw=none,every to/.append style={edge node={node [sloped, allow upside down, auto=false]{$\simeq$}}}}}
\tikzset{simeqS/.style={draw=none,every to/.append style={edge node={node [sloped, allow upside down, auto=false]{$\raisebox{0.8em}{$\simeq$}$}}}}}
\tikzset{aiso/.style={simeqS,preaction={draw,->}}}
\tikzset{dotdot/.style={dash pattern=on 0.25ex off 0.2ex, dash phase=0ex}}
\tikzset{RightA/.style={double distance=3.5pt,>={Implies},->},%
	triple/.style={-,preaction={draw,RightA}},%
	quadruple/.style={preaction={draw,RightA,shorten >=0pt},shorten >=1pt,-,double,double distance=0.2pt}}

\newtheorem{teor}{Theorem}[section]
\newtheorem{coroll}[teor]{Corollary}

\newtheorem{prop}[teor]{Proposition}

\theoremstyle{definition}
\newtheorem{defne}[teor]{Definition}

\newtheorem{rem}[teor]{Remark}
\newtheorem{exampl}[teor]{Example}

\newtheorem{cons}[teor]{Construction}
\newtheorem{notation}[teor]{Notation}

\def\nameit#1{\textrm{#1}~}
\def\defx{\nameit{Definition}}
\def\thex{\nameit{Theorem}}
\def\prox{\nameit{Proposition}}
\def\conx{\nameit{Construction}}
\def\remx{\nameit{Remark}}

\def\corx{\nameit{Corollary}}

\def\exax{\nameit{Example}}

\NewDocumentEnvironment{cd}{s O{7} O{7} b}{%
	\IfBooleanF{#1}{\begin{equation*}}\begin{tikzcd}[row sep=#2ex,column sep=#3ex,ampersand replacement=\&]
			#4
		\end{tikzcd}\IfBooleanF{#1}{\end{equation*}}\ignorespacesafterend}{}

\newenvironment{enum}{\begin{enumerate}[label=$($\hspace{0.12ex}\roman*\hspace{0.075ex}$)$]}{\end{enumerate}}
\newenvironment{enumT}{\begin{enumerate}[label=$($\hspace{-0.1ex}\roman*\hspace{0.13ex}$)$]}{\end{enumerate}}
\newenvironment{fun}{\[\begin{tabular}{F:RCL}}{\end{tabular}\]\ignorespacesafterend}
\newenvironment{eqD}[1]{\begin{equation}\label{#1}}{\end{equation}\ignorespacesafterend}
\newenvironment{eqD*}{\begin{equation*}}{\end{equation*}\ignorespacesafterend}

\def\:{\colon}
\providecommand\ordinarycolon{:}
\def\vcentcolon{\mathrel{\mathop\ordinarycolon}}
\newcommand{\deq}{\mathrel{\vcentcolon\mkern-1.2mu=}}
\def\phi{\varphi}
\def\eps{\varepsilon}
\def\O{\ensuremath{\Omega}\xspace}
\newcommand{\emp}{\emptyset}
\newcommand{\romanuppercase}[1]{\uppercase\expandafter{\romannumeral #1\relax}}

\def\set#1#2{\ensuremath{\left\{{#1}\left.\right|\,{#2}\right\}}}

\newcommand{\p}[1]{\big(\mkern1mu{#1}\mkern1mu\big)}
\newcommand{\pteor}[1]{\textup{(}{#1}\kern2pt\textup{)}}

\def\dfn#1{{\itshape #1}}
\def\predfn#1{{\itshape #1}}
\newcommand{\refs}[1]{\textup{(}\ref{#1}\textup{)}}

\newcommand{\scaleu}[2][1.2]{{\scalebox{#1}{$#2$}}}

\newcommand{\ov}[1]{\overline{#1}}
\def\t#1{\widetilde{#1}} 
\def\b{_{\bullet}}

\def\c{\circ}
\newcommand{\cont}{\subseteq}

\newcommand{\st}{^{\ast}}
\newcommand{\stb}{_{\ast}}

\newcommand{\slant}[2]{{\raisebox{.05em}{$#1$}\mkern-1mu\left/\raisebox{-.15em}{$#2$}\right.}}
\newcommand{\sliceslant}[2]{{\raisebox{.1em}{$#1$}\mkern-1mu\left/\raisebox{-.25em}{$#2$}\right.}}

\newcommand{\laxsliceslant}[2]{{\raisebox{.1em}{$#1$}\mkern-1mu\left/_{\mkern-4.1mu\operatorname{lax}}\right.\raisebox{-.25em}{$#2$}}}
\newcommand{\oplaxsliceslant}[2]{{\raisebox{.1em}{$#1$}\mkern-1mu\left/_{\mkern-4.1mu\operatorname{oplax}}\right.\raisebox{-.25em}{$#2$}}}

\makeatletter
\newcommand{\douwidehat}[2]{%
	\sbox0{$\m@th#1\widehat{\hphantom{#2}}$}%
	\sbox2{$\m@th#1x$}
	\sbox4{$\m@th#1#2$}
	\dimen0=\ht0
	\advance\dimen0 -.8\ht2
	\dimen2=\dp4
	\rlap{%
		\raisebox{\dimexpr\dimen0-\dimen2}{%
			\scalebox{1}[-1]{\box0}}}{#2}}
\makeatother

\DeclareFontFamily{OT1}{pzc}{}
\DeclareFontShape{OT1}{pzc}{m}{it}{<->s*[1.19]pzcmi7t}{}
\DeclareMathAlphabet{\mathpzc}{OT1}{pzc}{m}{it}

\DeclareFontFamily{U}{dutchcal}{\skewchar\font=45}
\DeclareFontShape{U}{dutchcal}{m}{n}{<->s*[1.05] dutchcal-r}{}
\DeclareMathAlphabet{\mathlcal}{U}{dutchcal}{m}{n}
\newcommand{\catfont}[1]{\ensuremath{\mathpzc{#1}}\xspace}

\newcommand{\A}{\catfont{A}}
\newcommand{\B}{\catfont{B}}
\newcommand{\C}{\catfont{C}}
\newcommand{\D}{\catfont{D}}
\newcommand{\E}{\catfont{E}}
\newcommand{\F}{\catfont{F}}
\newcommand{\K}{\catfont{K}}
\newcommand{\M}{\catfont{M}}

\newcommand{\V}{\catfont{V}}

\newcommand{\U}{\catfont{U}}
\newcommand{\W}{\catfont{W}}
\def\L{\catfont{L}}

\newcommand{\Y}{\ensuremath{\mathcal{Y}}\xspace}
\newcommand{\I}{\ensuremath{\mathcal{I}}\xspace}
\newcommand{\1}{\catfont{1}}
\newcommand{\2}{\catfont{2}}

\newcommand{\Set}{\catfont{Set}}

\newcommand{\Cat}{\catfont{Cat}}
\newcommand{\CAT}{\catfont{Cat}}
\newcommand{\CATlarge}{\catfont{CAT}}
\newcommand{\twoCAT}{2\h[2]\mbox{-}\CAT}

\NewDocumentCommand{\Fib}{t' t" t+ t? O{n} O{n} o}{
	\ensuremath{\ifx#5t{2\mbox{-}\Set\mbox{-}}\fi\catfont{\ifx#5d{D}\else{\ifx#5o{Op}\else{\ifx#5b{DOp}\else{\ifx#5t{Op}\else{\ifx#5c{Cl\h[-3]}\else{\ifx#5s{Sp}\fi}\fi}\fi}\fi}\fi}\fi{Fib}}\IfBooleanT{#3}{^{\h[3.7]\opn{s}\h[-3.2]}}\IfBooleanT{#4}{^{\h[3.7]\opn{P}\h[-3,7]}}{\IfBooleanTF{#1}{_{\h[0.4]\opn{cart}\ifx#6n{}\else{\h[-1.4],\h[0.4]{#6}}\fi}}{\IfBooleanTF{#2}{_{\h[0.4]\opn{clov}\ifx#6n{}\else{\h[-1.4],\h[0.4]{#6}}\fi}}{\ifx#6n{}\else{_{\h[0.4]{#6}}}\fi}}}\IfNoValueF{#7}{\h[-1]\left({#7}\right)}}
}

\NewDocumentCommand{\Sh}{o m}{
	\ensuremath{\catfont{Sh}\hspace{-0.15ex}\left({#2}\IfNoValueF{#1}{,{#1}}\right)}
}

\NewDocumentCommand{\St}{t+ o m}{
	\ensuremath{\catfont{St}\IfBooleanT{#1}{_{\opn{Ps}}}\hspace{-0.15ex}\left({#3}\IfNoValueF{#2}{,{#2}}\right)}
}

\NewDocumentCommand{\Alg}{t+ t' m}{
	\ensuremath{\IfBooleanT{#1}{\catfont{Ps}\mbox{-}}{#3}\mbox{-}\catfont{\IfBooleanT{#2}{Co}Alg}}
}

\newcommand{\slice}[2]{\sliceslant{#1}{#2}}

\newcommand{\laxslice}[2]{\laxsliceslant{#1}{#2}}
\newcommand{\oplaxslice}[2]{\oplaxsliceslant{#1}{#2}}

\newcommand{\Hom}[3][]{\operatorname{Hom}_{\mkern1mu #1}\mkern-1.5mu\left({#2},{#3}\right)}
\newcommand{\HomC}[3]{{#1}\left({#2},\h[1]{#3}\right)}
\newcommand{\m}[2]{\ensuremath{\left[#1,#2\right]}\xspace}

\newcommand{\Psm}[2]{\ensuremath{\opn{Ps}\left[#1,#2\right]}\xspace}

\newcommand{\moplaxn}[2]{\ensuremath{\left[#1,#2\right]_{\oplaxn}}\xspace}

\newcommand{\Sub}[1]{\operatorname{Sub}\hspace{-0.1ex}\left({#1}\right)}

\newcommand{\opn}[1]{\operatorname{#1}}
\newcommand{\y}[1]{\ensuremath{\operatorname{y}\hspace{-0.2ex}\left({#1}\right)}}

\DeclareMathOperator{\yy}{y}
\newcommand{\yyop}{\yy\op}
\newcommand{\id}[1]{\operatorname{id}_{#1}}
\newcommand{\Id}[1]{\operatorname{Id}_{#1}}

\newcommand{\op}{\ensuremath{^{\operatorname{op}}}}

\newcommand{\restr}[2]{{\left.\kern-\nulldelimiterspace {#1}\vphantom{\big|} \right|_{#2}}}
\newcommand{\ceil}[1]{\lceil #1 \rceil}
\newcommand{\dom}{\operatorname{dom}}

\newcommand{\pr}[1]{\operatorname{pr}_{#1}}

\newcommand{\Match}[2][]{\operatorname{Match}^{#1}\mkern-1.3mu\left(#2\right)}

\newcommand{\Union}[2]{\underset{#1}{\scaleu[1.27]{\bigcup}}\hspace{0.1ex}{#2}}

\DeclareMathOperator{\colim}{colim}

\DeclareMathOperator{\oplaxn}{oplax^{cart}}

\DeclareMathOperator{\pseudo}{pseudo}
\newcommand{\wlim}[2]{{\lim}^{#1}\h{#2}}
\newcommand{\wcolim}[2]{{\colim}^{#1}\h{#2}}

\newcommand{\oplaxncolim}[1]{\oplaxn\mbox{-}\h[1.5]\wcolim{\Delta 1}{#1}}

\newcommand{\G}[1]{\catfont{G}_{#1}}
\newcommand{\Gg}[1]{\catfont{\widehat{G}}_{#1}}

\newcommand{\dG}[2][]{\operatorname{G}_{#1}^{#2}}

\newcommand{\Int}[1]{\ensuremath{\int \hspace{-0.35ex} #1}}
\newcommand{\Intdiag}[1]{\ensuremath{\scaleu{\int} \hspace{-0.15ex} #1}}

\newcommand{\Groth}[1]{\Int{#1}}
\newcommand{\Grothdiag}[1]{\Intdiag{#1}}

\newcommand{\groth}[1]{\mathcal{G}\mkern-1.4mu\left(#1\right)}

\newcommand{\too}{\longrightarrow}
\newcommand{\mto}{\mapsto}
\newcommand{\mtoo}{\longmapsto}
\newcommand{\ito}{\hookrightarrow}

\makeatletter
\newcommand{\ar}[2][]{\xrightarrow[#1]{#2}}
\def\xlongrightarrowfill@{\arrowfill@\relbar\relbar\longrightarrow}
\newcommand{\arr}[2][]{%
	\ext@arrow 0099\xlongrightarrowfill@{#1}{#2}}
\newcommand{\aarr}[2][]{%
	\ext@arrow 0099\xlongrightarrowfill@{#1}{#2}} 
\newcommand{\arrr}[2]{%
	\begin{tikzcd}[column sep=#2ex,ampersand replacement=\&]
		\hspace{-0.75ex}\arrow[r,"{#1}"]\&\hspace{-0.75ex}
\end{tikzcd}}
\newcommand{\aR}[2][]{%
	\ext@arrow 0055{\Rightarrowfill@}{#1}{#2}}
\def\xLongrightarrowfill@{\arrowfill@\Relbar\Relbar\Longrightarrow}
\newcommand{\aRR}[2][]{%
	\ext@arrow 0099\xLongrightarrowfill@{#1}{#2}}
\def\aitofill@{\arrowfill@{\lhook\joinrel\relbar}\relbar\rightarrow}
\newcommand{\aito}[2][]{%
	\ext@arrow 3095\aitofill@{#1}{#2}}
\def\Longaitofill@{\arrowfill@{\lhook\joinrel\relbar\joinrel\relbar}\relbar\rightarrow}
\newcommand{\aitoo}[2][]{%
	\ext@arrow 0099\Longaitofill@{#1}{#2}}
\newcommand{\ffto}[1][]{\aito[\operatorname{f\h[0.1]f}]{#1}}

\newcommand{\al}[2][]{\xleftarrow[#1]{#2}}
\def\xlongleftarrowfill@{\arrowfill@\longleftarrow\relbar\relbar}
\newcommand{\all}[2][]{%
	\ext@arrow 0099\xlongleftarrowfill@{#1}{#2}}
\newcommand{\aL}[2][]{%
	\ext@arrow 0055{\Leftarrowfill@}{#1}{#2}}
\def\xLongleftarrowfill@{\arrowfill@\Longleftarrow\Relbar\Relbar}
\newcommand{\aLL}[2][]{%
	\ext@arrow 0099\xLongleftarrowfill@{#1}{#2}}
\def\xmapstofill@{\arrowfill@{\mapstochar\relbar}\relbar\rightarrow}
\newcommand{\am}[2][]{%
	\ext@arrow 0395\xmapstofill@{#1}{#2}}
\def\xlongmapstofill@{\arrowfill@\relbar\relbar\longmapsto}
\newcommand{\amm}[2][]{%
	\ext@arrow 0399\xlongmapstofill@{#1}{#2}}

\newcommand{\eqq}{\DOTSB\protect\Relbar\protect\joinrel\Relbar}
\def\xeqqfill@{\arrowfill@\Relbar\Relbar\eqq}
\newcommand{\aeqq}[2][]{%
	\ext@arrow 0099\xeqqfill@{#1}{#2}}
\def\xRrightarrowfill@{\arrowfill@\equiv\equiv\Rrightarrow}
\newcommand{\aM}[2][]{\ext@arrow 0359\xRrightarrowfill@{#1}{#2}}
\newcommand{\Llongrightarrow}{%
	\DOTSB\protect\equiv\protect\joinrel\Rrightarrow}
\def\xLlongrightarrowfill@{\arrowfill@\equiv\equiv\Llongrightarrow}
\newcommand{\aMM}[2][]{%
	\ext@arrow 0099\xLlongrightarrowfill@{#1}{#2}}
\makeatother

\newcommand{\aoplaxn}[1]{\aR[\oplaxn]{#1}}

\newcommand{\apseudo}[1]{\aR[\pseudo]{#1}}

\newcommand{\afull}[1]{\ar[\operatorname{full}]{#1}}

\newcommand{\iso}{\cong}

\newcommand{\aisoo}[1][]{\aarr[#1]{\scriptstyle\simeq}}
\newcommand{\aequi}{\ensuremath{\stackrel{\raisebox{-1ex}{\kern-.3ex$\scriptstyle\sim$}}{\rightarrow}}}
\newcommand{\aequii}{\ensuremath{\stackrel{\raisebox{-1ex}{\kern-.3ex$\scriptstyle\sim$}}{\longrightarrow}}}

\newcommand{\PB}[1]{\arrow[#1,phantom,"\scalebox{1.6}{\color{black}$\lrcorner$}",very near start]}
\newcommand{\Ar}[4][]{\arrow[#2,"{#3}"{#1},""{name=#4, anchor=center}]}
\newcommand{\Ars}[4][]{\arrow[#2,"{#3}"'{#1},""{name=#4, anchor=center}]}
\newcommand{\Arb}[6][]{\arrow[#2,"{#3}"{#1},from=#4,to=#5,shorten <= #6 em, shorten >= #6 em]}
\newcommand{\Arbs}[6][]{\arrow[#2,"{#3}"'{#1},from=#4,to=#5,shorten <= #6 em, shorten >= #6 em]}

\NewDocumentCommand{\fib}{O{n} O{2.3} mmm}{%
	\begin{cd}*[#2][5]
		{#3}\ifx#1n{\arrow[d,"{\,\scaleu{#4}}"]}\else{\ifx#1i{\arrow[d,hookrightarrow,"{\,\scaleu{#4}}"]}\else{\ifx#1e{\arrow[d,equal,"{\,\scaleu{#4}}"]}\else{\ifx#1R{\arrow[d,Rightarrow,"{\,\scaleu{#4}}"]}\fi}\fi}\fi}\fi\\
		{#5}\ifx#1o{\arrow[u,"{\,\scaleu{#4}}"']}\fi
	\end{cd}\xspace
}
\NewDocumentCommand{\fibdiag}{O{n} O{2.3} mmm}{%
	\begin{cd}*[#2][5]
		{#3}\ifx#1n{\arrow[d,"{\,{#4}}"]}\else{\ifx#1i{\arrow[d,hookrightarrow,"{\,{#4}}"]}\else{\ifx#1e{\arrow[d,equal,"{\,{#4}}"]}\else{\ifx#1R{\arrow[d,Rightarrow,"{\,{#4}}"]}\fi}\fi}\fi}\fi\\
		{#5}\ifx#1o{\arrow[u,"{\,{#4}}"']}\fi
	\end{cd}\xspace
}
	
\NewDocumentCommand{\sq}{s O{n} O{7} O{7} O{} O{2.7} O{2.2} O{0.5} O{n}}{%
	\def\foosq##1##2##3##4##5##6##7##8{%
		\IfBooleanTF{#1}{\begin{cd}*}{\begin{cd}}[#3][#4]
				{##1}\ifx#2p{\PB{rd}}\fi\arrow[r,"{##5}"]\ifx#9l{\arrow[d,equal,"{##6}"']}\else{\arrow[d,"{##6}"']}\fi\&{##2}\ifx#9r{\arrow[d,equal,"{##7}"]}\else{\arrow[d,"{##7}"]}\fi\ifx#2l{\arrow[ld,Rightarrow,shorten <=#6ex,shorten >=#7ex,"{#5}"{pos=#8}]}\fi\\
				{##3}\ifx#9d{\arrow[r,equal,"{##8}"']}\else{\arrow[r,"{##8}"']}\fi\ifx#2o{\arrow[ur,Rightarrow,shorten <=#6ex,shorten >=#7ex,"{#5}"{pos=#8}]}\fi\&{##4}
		\end{cd}}%
		\foosq}
\NewDocumentCommand{\sqs}{s O{n} O{7} O{7} O{} O{} O{} O{}}{%
	\def\foosqs##1##2##3##4##5##6##7##8{%
		\IfBooleanTF{#1}{\begin{cd}*}{\begin{cd}}[#3][#4]
				{##1}\ifx#2p{\PB{rd}}\fi\arrow[r,"{##5}"#5]\arrow[d,"{##6}"'#6]\&{##2}\arrow[d,"{##7}"#7]\\
				{##3}\arrow[r,"{##8}"'#8]\&{##4}
		\end{cd}}%
		\foosqs}
\NewDocumentCommand{\nat}{s O{n} O{7} O{7} O{2.7} O{2.2} O{0.5} O{n}}{%
	\def\foonat##1##2##3##4##5##6{%
		\IfBooleanTF{#1}{\sq*}{\sq}[#2][#3][#4][{##1}_{##4}][#5][#6][#7][#8]{{##2}\ifx#8l{}\else{({##5})}\fi}{{##3}\ifx#8r{}\else{({##5})}\fi}{{##2}\ifx#8l{}\else{({##6})}\fi}{{##3}\ifx#8r{}\else{({##6})}\fi}{{##1}_{##5}}{\ifx#8l{}\else{{##2}({##4})}\fi}{\ifx#8r{}\else{{##3}({##4})}\fi}{{##1}_{##6}}}%
	\foonat}
\NewDocumentCommand{\tr}{s O{4.5} O{6.5} O{0} O{0} O{n} O{0} O{} O{0}}{%
	\def\footr##1##2##3##4##5##6{%
		\IfBooleanTF{#1}{\begin{cd}*}{\begin{cd}}[#3][#2]
				{##1}\arrow[rr,"{##4}"]
				\Ars[inner sep =0.2ex]{dr}{##5}{A}\&[#4ex]\&[#5ex]{##2}\Ar[inner sep =0.2ex]{ld}{##6}{B}\\
				\&{##3}
				\ifx#6l{\Arb{Rightarrow,shift right=#7em}{#8}{A}{B}{#9}}\else{\ifx#6o{\Arbs{Rightarrow,shift right=#7em}{#8}{B}{A}{#9}}\else{\ifx#6i{\Arbs[inner sep=0.9ex]{iso,shift right=#7em}{#8}{A}{B}{#9}}\else{\ifx#6e{\Arb{equal,shift right=#7em}{#8}{A}{B}{#9}}\else{}\fi}\fi}\fi}\fi
		\end{cd}}%
		\footr}
\NewDocumentCommand{\trslice}{s t+ O{7} O{7}}{%
	\def\footrslice##1##2##3##4##5##6{%
		\IfBooleanTF{#1}{\begin{cd}*}{\begin{cd}}[#3][#4]
				{##1}\arrow[r,"{##4}"]\IfBooleanF{#2}{\arrow[d,"{##5}"']}\&{##2}\\
				{##3}\IfBooleanT{#2}{\arrow[u,"{##5}"]}\arrow[ru,"{##6}"']
		\end{cd}}%
		\footrslice}
\NewDocumentCommand{\tc}{s t+ O{7} O{30} O{} O{} O{} o}{
	\def\footc##1##2##3##4##5{%
		\FPmul\Mulresulttwo{#3}{#3}%
		\FPmul\Mulresult{0.0026}{\Mulresulttwo}%
		\IfBooleanTF{#1}{\begin{cd}*}{\begin{cd}}[#3][#3]
				{##1}\Ar[#5]{r,bend left=#4}{##3}{A}\Ars[#6]{r,bend right=#4}{##4}{B}\&{##2}
				\IfBooleanTF{#2}{\Arb[description,pos=0.49]}{\Arb}{Rightarrow #7}{\mkern1mu {##5}}{A}{B}{\IfNoValueTF{#8}{\Mulresult}{#8}}
		\end{cd}}%
		\footc}
\NewDocumentCommand{\tcwl}{s t+ O{7} O{30} O{} O{} O{} o O{-2}}{
	\def\footcwl##1##2##3##4##5##6##7{%
		\FPmul\Mulresulttwo{#3}{#3}%
		\FPmul\Mulresult{0.0026}{\Mulresulttwo}%
		\IfBooleanTF{#1}{\begin{cd}*}{\begin{cd}}[#3][#3]
				##6 \arrow[r,"{##7}"]\&[#9ex]{##1}\Ar[#5]{r,bend left=#4}{##3}{A}\Ars[#6]{r,bend right=#4}{##4}{B}\&{##2}\IfBooleanTF{#2}{\Arb[description,pos=0.49]}{\Arb}{Rightarrow #7}{\mkern1mu {##5}}{A}{B}{\IfNoValueTF{#8}{\Mulresult}{#8}}
		\end{cd}}%
		\footcwl}
\NewDocumentCommand{\tcwr}{s t+ O{7} O{30} O{} O{} O{} o O{-2}}{
	\def\footcwr##1##2##3##4##5##6##7{%
		\FPmul\Mulresulttwo{#3}{#3}%
		\FPmul\Mulresult{0.0026}{\Mulresulttwo}%
		\IfBooleanTF{#1}{\begin{cd}*}{\begin{cd}}[#3][#3]
				{##1}\Ar[#5]{r,bend left=#4}{##3}{A}\Ars[#6]{r,bend right=#4}{##4}{B}\&{##2}\arrow[r,"{##7}"]\&[#9ex]##6\IfBooleanTF{#2}{\Arb[description,pos=0.49]}{\Arb}{Rightarrow #7}{\mkern1mu {##5}}{A}{B}{\IfNoValueTF{#8}{\Mulresult}{#8}}
		\end{cd}}%
		\footcwr}
\NewDocumentCommand{\tcv}{s t' O{7} O{30} mmmmm}{
	\FPmul\Mulresulttwo{#3}{#3}%
	\FPmul\Mulresult{0.0026}{\Mulresulttwo}%
	\IfBooleanTF{#1}{\begin{cd}*}{\begin{cd}}[#3][#3]
			{#5}\IfBooleanTF{#2}{\Ars{d,leftarrow,bend right=#4}{#7}{A}\Ar{d,leftarrow,bend left=#4}{#8}{B}}{\Ars{d,bend right=#4}{#7}{A}\Ar{d,bend left=#4}{#8}{B}}\\{#6}
			\Arb{Rightarrow}{#9}{A}{B}{\Mulresult}
		\end{cd}}
\NewDocumentCommand{\twonats}{s O{2.2} O{8} O{7} O{1.05} O{3.45} O{2}}{%
	\def\footwonats##1##2##3##4##5##6##7##8##9{%
		\def\foofootwonats####1####2####3####4####5{%
			\IfBooleanTF{#1}{\begin{cd}*}{\begin{cd}}[#3][#4]
					##1 \Ar{r}{##9}{} \Ars{d,bend right=40}{##5}{A} \Ar{d,bend left=40}{##6}{B} \&
					##2 \Ars{d,bend left}{##8}{Q} \arrow[ld,Rightarrow,shift left=#7,"{####4}"{pos=0.48},shorten <=#5ex, shorten >=#6ex]\&[-2ex]
					##1 \Ar{r}{##9}{} \Ar{d,bend right}{##5}{R} \&
					##2 \Ars{d,bend right=40}{##7}{C} \Ar{d,bend left=40}{##8}{D} \arrow[ld,Rightarrow,shift right=#7,"{####5}"'{pos=0.52},shorten <=#6ex, shorten >=#5ex] \\
					##3 \Ars{r}{####1}{} \&
					##4 \&
					##3 \Ars{r}{####1}{} \& 
					##4
					\Arbs{Rightarrow}{\,{####2}}{B}{A}{0.3}
					\Arbs{Rightarrow}{\,{####3}}{D}{C}{0.3}
					\Arb{equal}{}{Q}{R}{#2}
			\end{cd}}%
			\foofootwonats}\footwonats}
\makeatletter
\newcommand{\DOpFiboi}[2][\@nil]{%
	\def\tmp{#1}%
	\ifx\tmp\@nnil{\ensuremath{\slant{\catfont{DOpFib}\left({#2}\right)}}{\iso}}%
	\else{\ensuremath{\slant{\catfont{DOpFib}_{\h[0.4]{#1}}\h[-1]\left({#2}\right)}{\iso}}}\fi}
\makeatother
\newcommand{\tcstwodim}[9][7]{%
	\FPmul\Mulresulttwo{#1}{#1}%
	\FPmul\Mulresult{0.0027}{\Mulresulttwo}%
	\begin{cd}[#1][#1]
		{#2}\Ar[#7]{r,Rightarrow,bend left}{#4}{A}\Ars[#8]{r,Rightarrow,bend right}{#5}{B}\&{#3}
		\Arb{triple #9}{\mkern1mu {#6}}{A}{B}{\Mulresult}
	\end{cd}}
\newcommand{\tcswltwodim}[9][7]{%
	\def\footcswl##1##2{%
		\FPmul\Mulresulttwo{#1}{#1}%
		\FPmul\Mulresult{0.0027}{\Mulresulttwo}%
		\begin{cd}[#1][#1]
			##1 \arrow[r,Rightarrow,"{##2}"]\&[-2ex]{#2}\Ar[#7]{r,Rightarrow,bend left}{#4}{A}\Ars[#8]{r,Rightarrow,bend right}{#5}{B}\&{#3}
			\Arb{triple #9}{\mkern1mu {#6}}{A}{B}{\Mulresult}
	\end{cd}}%
	\footcswl%
}

\def\pbsqdrp#1#2#3#4#5#6#7#8#9{%
	\def\foopbsqdrp##1##2##3##4{%
		\begin{cd}[6][7.5]
			#1 \PB{rd} \arrow[r,"{#7}"] \arrow[d,"{#9}"'] \& #2 \PB{rd} \arrow[r,"{#8}"] \arrow[d,"{##1}"'] \& #3 \arrow[d,"{##2}"] \\
			#4  \arrow[r,"{##3}"']  \& #5 \arrow[r,"{##4}"'] \& #6 
	\end{cd}}%
	\foopbsqdrp%
}
\def\pbsqdrpN#1#2#3#4#5#6#7#8#9{%
	\def\foopbsqdrpN##1##2##3##4{%
		\begin{cd}*[5.5][5.5]
			#1 \PB{rd} \arrow[r,"{#7}"] \arrow[d,"{#9}"'] \& #2 \PB{rd} \arrow[r,"{#8}"] \arrow[d,"{##1}"'] \& #3 \arrow[d,"{##2}"] \\
			#4  \arrow[r,"{##3}"']  \& #5 \arrow[r,"{##4}"'] \& #6 
	\end{cd}}%
	\foopbsqdrpN%
}
\newcommand{\modifopls}[9][2.3]{%
	\def\foomodifopls##1##2##3##4##5##6##7##8##9{%
		\def\foofoomodifopls####1####2####3####4####5####6####7####8####9{%
			\begin{cd}[####8][####9]
				#2 \Ars{d}{##1}{} \Ar[##7]{r,bend left}{#6}{} \&
				#3 \Ar{d}{##2}{Q} \&[-2ex]
				#2 \Ars{d}{##1}{R} \Ar[##7]{r,bend left}{#6}{A} \Ars[##8]{r,bend right}{#7}{B} \&
				#3 \Ar{d}{##2}{} \\
				#4 \Ar[##9]{r,bend left}{#8}{C} \Ars[####1]{r,bend right}{#9}{D} 
				\arrow[ru,Rightarrow, shift left = 0.8em,"{##5}"{####2}####6]
				\&
				#5 \&
				#4 \Ars[####1]{r,bend right}{#9}{} 
				\arrow[ru,Rightarrow, shift right = 0.77em,"{##6}"'{####3}####7]
				\& 
				#5
				\Arbs[####4]{Rightarrow}{\,{##3}}{B}{A}{0.3}
				\Arbs[####5]{Rightarrow}{\,{##4}}{D}{C}{0.3}
				\Arb{equal}{}{Q}{R}{#1}
		\end{cd}}%
		\foofoomodifopls}%
	\foomodifopls%
}
\newcommand{\trslicematchN}[8]{%
	\begin{cd}*[#7][#8]
		{#1}\arrow[d,"{#4}"']\arrow[rd,dotdot,"{#5}"]\\
		{#2}\arrow[r,dotdot,"{#6}"']\&{#3}
\end{cd}}
			
\begin{document}

\title[2-classifiers via dense generators and a universe in stacks]{2-classifiers via dense generators and Hofmann--Streicher universe in stacks}
\author[L. Mesiti]{Luca Mesiti}
\address{School of Mathematics, University of Leeds, Woodhouse, Leeds LS2~9JT, England, United Kingdom}
\email{mmlme@leeds.ac.uk}
\keywords{elementary topos, 2-category, classifier, stack, dense generator, fibration}
\subjclass[2020]{18B25, 18N10, 18F20, 18D30}

\begin{abstract}
	We expand the theory of 2-classifiers, that are a 2-categorical generalization of subobject classifiers introduced by Weber. The idea is to upgrade monomorphisms to discrete opfibrations. We prove that the conditions of 2-classifier can be checked just on a dense generator. The study of what is classified by a 2-classifier is similarly reduced to a study over the objects that form a dense generator. We then apply our results to the cases of prestacks and stacks, where we can thus look just at the representables. We produce a 2-classifier in prestacks that classifies all discrete opfibrations with small fibres. Finally, we restrict such 2-classifier to a 2-classifier in stacks. This is the main ingredient of a proof that Grothendieck 2-topoi are elementary 2-topoi. Our results also solve a problem posed by Hofmann and Streicher when attempting to lift Grothendieck universes to sheaves.
\end{abstract}

\maketitle

\setcounter{tocdepth}{1}
\tableofcontents
				
\section{Introduction}

The notion of 2-classifier has been proposed by Weber in~\cite{weber_yonfromtwotop}. It is a 2-categorical generalization of the concept of subobject classifier and thus the main ingredient of a 2-dimensional notion of elementary topos. In~\cite{weber_yonfromtwotop}, Weber proposes as well a definition of elementary 2-topos. Although 2-dimensional elementary topos theory is still at its beginning, we believe it has a great potential. Indeed, for example, elementary 2-topoi could pave the way towards a 2-categorical logic and offer the right tools to study it. We believe they could also be fruitful for theories of bundles in geometry.

In this paper, we contribute to expand 2-dimensional elementary topos theory. We substantially reduce the work needed to prove that something is a 2-classifier, and we present the main part of a proof that Grothendieck 2-topoi are elementary 2-topoi. The reason why we focus on 2-classifiers is that the rest of the definition of elementary 2-topos proposed by Weber is yet to be ascertained. We hope that this paper will contribute to reach a universally accepted notion of elementary 2-topos. Weber's idea for 2-categorical classifiers is that, moving to dimension 2, one can and wants to classify morphisms with higher dimensional fibres. So monomorphisms (or subobjects) are upgraded to discrete opfibrations in a 2-category, which have been introduced by Street in~\cite{street_fibandyonlemma}. Interestingly, a 2-classifier can also be thought of as a Grothendieck construction inside a 2-category, thanks to Weber's~\cite{weber_yonfromtwotop}, see also~\ref{subsection2class}. Indeed the archetypal example of 2-classifier is given by the Grothendieck construction (or category of elements), that exhibits the 2-category $\Cat$ of small categories as the archetypal elementary 2-topos. We introduce the notion of \dfn{good 2-classifier} in \defx\ref{defgood2clas}, which captures well-behaved 2-classifiers. The idea is to keep as classifier a morphism with domain the terminal object, and upgrade the classification process from one regulated by pullbacks to one regulated by comma objects. Good 2-classifiers are closer to the point of view of logic, as they can still be interpreted as the inclusion of a verum inside an object of generalized truth values. Moreover, a classification process regulated by comma objects is sometimes more natural and easier to handle. We also ask good 2-classifiers to classify all discrete opfibrations that satisfy a fixed pullback-stable property $\opn{P}$. In our examples, such a $\opn{P}$ will be the property of having small fibres. Of course, the construction of the category of elements, hosted by $\Cat$, is a good 2-classifier, classifying all discrete opfibrations (in $\Cat$) with small fibres. A problem with 2-classifiers is that it is quite hard and lengthy to prove that something is a 2-classifier.

We prove that both the conditions of 2-classifier and what gets classified by a 2-classifier can be checked just over the objects that form a dense generator. So that the whole study of a would-be 2-classifier is substantially reduced. We also give a concrete recipe to build the characteristic morphisms (i.e.\ the morphisms into the universe that encode what gets classified). This is organized in the three Theorems~\ref{faith},~\ref{fullness} and~\ref{esssurj}; see also Corollaries~\ref{corollsharperesssurj} and~\ref{corollgoodtwoclas}. Dense generators capture the idea of a family of objects that generate all the other ones via colimits in a nice way. The preeminent example is given by representables in categories of presheaves. To have a hint of the benefits offered by our theorems of reduction to dense generators, we can look at the case of $\Cat$. We have that the singleton category alone forms a dense generator of $\Cat$. All the major properties of the Grothendieck construction are hence deduced from the trivial observation that everything works well over the singleton category (\exax\ref{exampleredincat}). The proof of our theorems of reduction to dense generators uses our calculus of colimits in 2-dimensional slices, developed in~\cite{mesiti_colimitsintwodimslices}. Such calculus generalizes to dimension 2 the well-known fact that a colimit in a 1-dimensional slice is precisely the map from the colimit of the domains which is induced by the universal property. It is based on the reduction of weighted 2-colimits to cartesian-marked oplax conical ones, developed by Street in~\cite{street_limitsindexedbycatvalued} and recalled in~\ref{subsectioncartmarkedoplaxcolim}, and on $\F$-category theory (also called enhanced 2-category theory), for which we take as main reference Lack and Shulman's~\cite{lackshulman_enhancedtwocatlimlaxmor}.

We then apply our theorems of reduction of the study of 2-classifiers to dense generators to the case of 2-presheaves, i.e.\ prestacks. Our theorems allow us to just consider the classification over representables. Yoneda lemma determines up to equivalence the construction of a good 2-classifier in prestacks that classifies all discrete opfibrations with small fibres. We explain how this involves discrete opfibrations over representables, which offer a 2-categorical notion of sieve. Indeed, the philosophy is to upgrade monomorphisms to discrete opfibrations. And sieves, which can be characterized as subfunctors of representables, are hence upgraded to discrete opfibrations over representables. Exactly as sieves are a key element for the subobject classifier in presheaves, the 2-dimensional generalization of sieves described above is a key element for the 2-classifier in prestacks. The only problem is that taking discrete opfibrations over representables only gives a pseudofunctor $\O$ which is not a 2-functor and that a priori only lands in large categories. Thanks to our joint work with Caviglia~\cite{cavigliamesiti_indexedgrothconstr}, we can apply an indexed version of the Grothendieck construction to produce a nice concrete strictification of such pseudofunctor. Although it was already known before~\cite{cavigliamesiti_indexedgrothconstr} that any pseudofunctor can be strictified, by the theory developed by Power in~\cite{power_generalcoherenceresult} and later by Lack in~\cite{lack_codescentobjcoherence}, the work of ~\cite{cavigliamesiti_indexedgrothconstr} can be applied to produce an explicit and easy to handle strictification $\t{\O}$ of $\O$, which in addition lands in $\Cat$. Moreover, we will show in Section~\ref{sectiontwoclassinstacks} that such strictification can also be restricted in a natural way to a good 2-classifier in stacks. Explicitly, the 2-functor $\t{\O}$ that we obtain takes presheaves on slices (see \prox\ref{propstrictification}). In \thex\ref{teor2classinprestacks}, we prove that $\t{\O}$ gives a good 2-classifier in prestacks that classifies all discrete opfibrations with small fibres. A partial result on this direction is already in Weber's~\cite{weber_yonfromtwotop}. Our result is in line with Hofmann and Streicher's~\cite{hofmannstreicher_liftinggrothuniverses}, that uses a similar idea to lift Grothendieck universes to presheaves, in order to interpret Martin-L\"{o}f type theory in a presheaf topos. It is also in line with the recent Awodey's~\cite{awodey_onhofmannstreicheruniverses}, that captures the construction of the Hofmann and Streicher's universe in presheaves in a conceptual way. Our proof goes through the bicategorical classification process given by the pseudofunctor $\O$. We show that, over representables, such classification is exactly the Yoneda lemma. We then use this to prove that the strictification $\t{\O}$ is a good 2-classifier in prestacks. Although some points would be smoother in a bicategorical context, we believe that it is important to show how strict the theory can be. In the case of prestacks, the strict classification process, which involves an indexed Grothendieck construction, actually seems more interesting than the bicategorical one, which reduces to the Yoneda lemma. In future work, we will adapt the results of this paper to the bicategorical context, using a suitable bicategorical notion of classifier.

Finally, in \thex\ref{teor2classinstacks}, we restrict our good 2-classifier in prestacks to a good 2-classifier in stacks, that classifies again all discrete opfibrations with small fibres. We achieve this by proving a general result of restriction of good 2-classifiers to nice sub-2-categories (\thex\ref{teorfactorization}), involving factorization arguments and our theorems of reduction to dense generators. Stacks are a bicategorical generalization of sheaves and they were introduced by Giraud in~\cite{giraud_methodedeladescente}. Like sheaves, they are able to glue together families of objects that are compatible under descent. But such descent compatibilities are only asked up to isomorphism. And the produced global data then equally recovers the starting local data up to isomorphism. Stacks are the right notion to use to generalize Grothendieck topoi to dimension 2. Our result is thus the main part of a proof that Grothendieck 2-topoi are elementary 2-topoi. As explained in \remx\ref{remourstacks}, we consider strictly functorial stacks with respect to a subcanonical Grothendieck topology. So that they form a full sub-2-category of the 2-category of 2-presheaves. While our good 2-classifier in prestacks involves a 2-dimensional notion of sieves, our good 2-classifier in stacks involves a 2-dimensional notion of closed sieves. The idea is to select, out of all the presheaves on slices considered in the definition of $\t{\O}$, the sheaves with respect to the Grothendieck topology induced on the slices. This restriction of $\t{\O}$ is tight enough to give a stack $\O_J$, but at the same time loose enough to still host the classification process of prestacks. We prove that $\O_J$ is a good 2-classifier in stacks that classifies all discrete opfibrations with small fibres. Our result solves a problem posed by Hofmann and Streicher in~\cite{hofmannstreicher_liftinggrothuniverses}. Indeed, in a different context, they considered the same natural idea to restrict their analogue of $\t{\O}$ by taking sheaves on slices. However, this did not work for them, as it does not give a sheaf. Only with a different approach Streicher managed in~\cite{streicher_universesintoposes} to construct a universe in sheaves. Our results show that the natural restriction of $\t{\O}$ to take sheaves on slices yields nonetheless a stack and a good 2-classifier in stacks. The idea is that, in order to increase the dimension of the fibres of the morphisms to classify, one should also increase the dimension of the ambient. And thus stacks behave better than sheaves for the classification of small families.

\subsection*{Outline of the paper}

In Section~\ref{sectionpreliminaries}, we recall 2-classifiers (\ref{subsection2class}), dense generators (\ref{subsectiondensegen}), stacks (\ref{subsectionstacks}) and cartesian-marked oplax colimits (\ref{subsectioncartmarkedoplaxcolim}). We also introduce good 2-classifiers (\defx\ref{defgood2clas}).

In Section~\ref{sectionreductiontogenerators}, we present a reduction of the study of a 2-classifier to dense generators (Theorems~\ref{faith},~\ref{fullness} and~\ref{esssurj}; see also Corollaries~\ref{corollsharperesssurj} and~\ref{corollgoodtwoclas}). We then prove a general result of restriction of good 2-classifiers to nice sub-2-categories (\thex\ref{teorfactorization}).

In Section~\ref{sectiontwoclassintwopresheaves}, we apply our theorems of reduction of the study of a 2-classifier to dense generators to the case of prestacks. We thus produce a good 2-classifier in prestacks that classifies all discrete opfibrations with small fibres (\thex\ref{teor2classinprestacks}). We also show a concrete recipe for the characteristic morphisms (\remx\ref{remconcreterecipeforclasmorph}).

In Section~\ref{sectiontwoclassinstacks}, we apply \thex\ref{teorfactorization} to restrict our good 2-classifier in prestacks to a good 2-classifier in stacks, classifying again all discrete opfibrations with small fibres (\thex\ref{teor2classinstacks}).

\subsection*{Notations}

Throughout this paper, we fix Grothendieck universes $\U$, $\V$ and $\W$ such that $\U\in \V\in \W$. We denote as $\Set$ the category of $\U$-small sets, as $\Cat$ the 2-category of $\V$-small categories (i.e.\ categories such that both the collections of their objects and of their morphisms are $\V$-small) and as $\CATlarge$ the 2-category of $\W$-small categories. So that $\Set\in \Cat$ and the underlying category $\Cat_0$ of $\Cat$ is in $\CATlarge$. \predfn{Small category} will mean $\V$-small category. \predfn{Small fibres}, for a discrete opfibration in $\Cat$, will mean $\U$-small fibres. \predfn{2-category} will mean a $\W$-small $\Cat$-enriched category. \predfn{Small 2-category} will mean $\V$-small 2-category.

We fix an arbitrary $2$-category \L with pullbacks along discrete opfibrations (see \defx\ref{R}), comma objects and terminal object. We also fix a choice of such pullbacks in \L such that the change of base of an identity is always an identity. Following the proofs, it will be clear that some results of this paper involving 2-classifiers \`{a} la Weber work also without assuming comma objects, and that some results involving \predfn{good 2-classifiers} (see \defx\ref{defgood2clas}) work also without assuming pullbacks along discrete opfibrations.

We will use the following notations.
\begin{flushleft}
	\begin{longtable}{p{0.215\linewidth} p{0.75\linewidth}}
		$\m{\C\op}{\Cat}$ & the (strict) functor 2-category from $\C\op$ to $\Cat$\\
		$\moplaxn{\C\op}{\Cat}$ & the 2-category of 2-functors, cartesian-marked oplax natural transformations and modifications\\
		$\Psm{\C\op}{\Cat}$ & the 2-category of pseudofunctors, pseudo-natural transformations and modifications\\
		$\yy\:\C\to\m{\C\op}{\Cat}$ & the 2-categorical Yoneda embedding\\
		$\St[J]{\C}$ & the full sub-2-category of $\m{\C\op}{\Cat}$ on stacks with respect to the Grothendieck topology $J$\\
		$\groth{F}\:\Groth{F}\to \C$ & the 2-category of elements of a 2-functor $F\:\A\op\to \Cat$ (with $\A$ a 2-category), or also the classical Grothendieck construction\\
		$(\phi)_B$ & the fibre of an opfibration $\phi$ in $\Cat$ over an object $B$ of the base\\
		$\Fib[b][\L][F]$ & the category of discrete opfibrations in a 2-category $\L$ over $F$\\
		$\Fib?[b][n][F]$ & the full subcategory of $\Fib[b][\L][F]$ on the discrete opfibrations that satisfy a fixed pullback-stable property $\opn{P}$; we will denote as $\opn{s}$ the property of having small fibres\\
		$\slice{\C}{C}$ & the slice of a category $\C$ over $C\in \C$\\
		$\laxslice{\L}{M}$ & the lax slice of a 2-category $\L$ over $M\in \L$\\
		$\dom$ & the domain (2-)functor from a (lax) slice\\
		$1$ & the terminal object of a (2-)category; variant $\1$ for the singleton category\\
		$\Delta 1$ & the constant at 1 presheaf\\
		$F\apseudo{} G$ & a pseudo-natural transformation\\
		$F\aoplaxn{} G$ & a cartesian-marked oplax natural transformation\\
		$\wlim{W}{F}$ & the enriched limit of $F$ weighted by $W$; variant $\wcolim{W}{F}$ for colimits\\
		$\oplaxncolim{K}$ & the cartesian-marked oplax conical colimit of the 2-diagram $K$\\
		$f\c g$ & the composite of 1-cells or the vertical composition of 2-cells\\
		$\alpha\ast \beta$ & the horizontal composition of 2-cells; variants $\alpha\ast f$ and $f\ast \alpha$ for the whiskerings of a 2-cell $\alpha$ with a 1-cell $f$\\
		$\id{A}$ & the identity 1-cell on $A$; variants $\Id{C}$ for the identity (2-)functor on $\C$ and $\id{f}$ for the identity 2-cell on a 1-cell $f$\\
		$A\ffto B$ & a fully faithful morphism in a 2-category or a fully faithful 2-functor\\
		$\M\cont \L$ & a full sub-2-category, i.e.\ an injective on objects and fully faithful 2-functor\\
		$-$ & placeholder
	\end{longtable}
\end{flushleft}

\section{Preliminaries}\label{sectionpreliminaries}

In this section, we recall some important concepts and results that we will use throughout the paper. These include 2-classifiers, dense generators and stacks. As we will need them to prove our theorems of reduction to dense generators of the study of 2-classifiers, we also recall cartesian-marked oplax conical colimits. Moreover, we introduce the notion of \predfn{good 2-classifier} (\defx\ref{defgood2clas}).

In~\ref{subsection2class}, we recall the notion of 2-classifier. Weber's idea (in~\cite{weber_yonfromtwotop}) is that a 2-classifier should be a discrete opfibration classifier. The definition of opfibration in a $2$-category is due to Street~\cite{street_fibandyonlemma}, in terms of pseudo-algebras for a $2$-monad. It is known that opfibrations in a $2$-category can be equivalently defined by representability, as in Weber's~\cite[Section~2.2]{weber_yonfromtwotop}. In \defx\ref{defgood2clas}, we introduce the notion of \predfn{good 2-classifier}.

In~\ref{subsectiondensegen}, we briefly recall from Kelly~\cite[Chapter~5]{kelly_basicconceptsofenriched} (2-categorical) dense generators and the preeminent example of representables in 2-presheaves. The idea is that every object of a 2-category can be written as a nice colimit of a small family of objects.

In~\ref{subsectionstacks}, we recall the concept of stack, which is a bicategorical generalization of sheaves. Stacks have been introduced by Giraud in~\cite{giraud_methodedeladescente}. Like sheaves, they are able to glue together families of objects that are compatible under descent. But such descent compatibilities are only asked up to isomorphism. The produced global data then equally recovers the starting local data up to isomorphism.

Finally, in~\ref{subsectioncartmarkedoplaxcolim}, we recall that the theory of weighted 2-colimits is equivalent to the theory of cartesian-marked oplax conical colimits. We will use this to prove our theorems of reduction to dense generators. Indeed cartesian-marked oplax conical colimits are essentially conical and much easier to handle in our proofs. They are a particular case of a general notion of (co)limit introduced by Gray in~\cite[\defx{I,7.9.1}]{gray_formalcattheory}. Street proved in~\cite[\thex{14} and \thex{15}]{street_limitsindexedbycatvalued} that these colimits are weighted colimits and that any weighted colimit can be reduced to one of this form. We gave new, more elementary proofs of this in~\cite[\thex{2.18} and \thex{2.19}]{mesiti_twosetenrgrothconstrlaxnlimits}.

\subsection{2-classifiers}\label{subsection2class}

Weber's idea is to define 2-classifiers by looking at the following well-known equivalent definition of subobject classifier.

\begin{defne}\label{defsubclas}
	Let $\E$ be a category. A \dfn{subobject classifier} is a monomorphism $i\:I\aito{} \O$ in $\E$ such that for every $F\in \E$ the function
	$$G_{i,F}\:\Hom[\E]{F}{\O}\to \Sub{F}$$
	given by pulling back $i$ is a bijection, where $\Sub{F}$ is the set of subobjects of $F$. When this holds, $I$ is forced to be the terminal object of $\E$.
\end{defne}

Towards a notion of 2-classifier, Weber proposed in~\cite{weber_yonfromtwotop} to upgrade the concept of monomorphism to the one of discrete opfibration. The idea is that, moving to dimension 2, i.e.\ increasing by 1 the dimension of the ambient, we want to increase by 1 also the dimension of the fibres of the morphisms to classify. While injective functions have as fibres either the singleton or the empty set, discrete opfibrations have as fibres general sets. Exactly as the notion of injective function extends to the one of monomorphism in any category, the notion of discrete opfibration extends to the one inside any 2-category. This idea is closely connected with that of homotopy level in Voevodsky's univalent foundations, see~\cite[Chapter~7]{univalentfoundations_homotopytypetheory} and Voevodsky's~\cite{voevodsky_univalentfoundations}.

Exactly as $\Set$ is the archetypal elementary topos, $\Cat$ needs to be the archetypal elementary 2-topos. And $\Cat$ hosts indeed a nice classification of discrete opfibrations, given by the category of elements (or Grothendieck construction).

\begin{defne}[Weber~{\cite[Section~2.2]{weber_yonfromtwotop}}, Street~\cite{street_fibandyonlemma}]\label{R}
	A morphism $s:{G}\to{F}$ in $\L$ is a \dfn{discrete opfibration in \L $($over $F$$)$} if for every $X\in \L$ the functor
	$$\HomC{\L}{X}{G}\arrr{s\c -}{5}\HomC{\L}{X}{F}$$
	induced between the hom-categories is a discrete opfibration in \Cat.
	
	We denote as \Fib[b][\L][F] or just as \Fib[b][n][F] the full subcategory of the strict slice $\slice{\L}{F}$ on the discrete opfibrations over $F$. That is, a morphism between discrete opfibrations $s\:G\to F$ and $s'\:G'\to F$ is a morphism $G\to G'$ that makes the triangle with $s$ and $s'$ commute. We denote as $\Fib?[b][n][F]$ the full subcategory of \Fib[b][n][F] on the discrete opfibrations that satisfy a fixed pullback-stable property $\opn{P}$.
\end{defne}

\begin{rem}
	By definition, a discrete opfibration $s\:G\to F$ in $\L$ is required to lift every $2$-cell $\theta\:s\c a\aR{}b$ to a unique $2$-cell $\ov{\theta}^a\:a\aR{}\theta\stb a$. We can draw the following diagram to say that $s\c \theta\stb a=b$ and $s\star\ov{\theta}^a=\theta$.
	\begin{cd}[4.3][4.3]
		X \arrow[dd,equal]\arrow[rr,bend left=25,"{a}",""'{name=A}]\arrow[rr,bend right=25,"{\theta\stb a}"',""{name=B}]\& \& G \arrow[dd,"{s}"]\\[1ex]
		\& |[alias=W]| G \arrow[rd,bend left=20,"{s}"]\\[-5ex]
		X \arrow[ru,bend left=18,"{a}"]\arrow[rr,bend right=18,"{b}"',""{name=Q}] \&\hphantom{.}\& F
		\arrow[from=A,to=B,Rightarrow,"{\h[2]\ov{\theta}^a}"]
		\arrow[from=W,to=Q,Rightarrow,"{\h[2]\theta}"]
	\end{cd}
\end{rem}

\begin{rem}\label{rempullbackopfib}
	Discrete opfibrations in $\L$ are stable under pullbacks. Indeed $\HomC{\L}{X}{-}$ preserves pullbacks (as it preserves all limits, because it is a representable) and discrete opfibrations in $\Cat$ are stable under pullbacks.
\end{rem}

\begin{rem}\label{remrecoverusualdiscopf}
	Applying \defx\ref{R} to $\L=\CAT$, we obtain a notion that is equivalent to the usual one of discrete opfibration. This is essentially because for $\L=\CAT$ it suffices to require the above liftings for $X=\1$. We are then able to lift entire natural transformations $\theta$ componentwise.
\end{rem}

The following proposition from our joint work with Caviglia~\cite{cavigliamesiti_indexedgrothconstr} extends the idea of \remx\ref{remrecoverusualdiscopf} to prestacks. Let $\C$ be a small category and consider the functor 2-category $\m{\C\op}{\Cat}$ (i.e.\ the 2-category of prestacks).

\begin{prop}[{\cite[\prox{3.5}]{cavigliamesiti_indexedgrothconstr}}]\label{propdiscopfibinpresheaves}
	A morphism $s:G\to F$ in $\m{\C\op}{\Cat}$ \pteor{that is, a $2$-natural transformation $s$} is a discrete opfibration in $\m{\C\op}{\Cat}$ if and only if for every $C\in \C$ the component $s_C$ of $s$ on $C$ is a discrete opfibration in $\Cat$.
\end{prop}

Recall from \cite{cavigliamesiti_indexedgrothconstr} also that \prox\ref{propdiscopfibinpresheaves} allows us to define \predfn{having small fibres} for a discrete opfibration in $\m{\C\op}{\Cat}$.

\begin{defne}[{\cite[\defx{3.7}]{cavigliamesiti_indexedgrothconstr}}]\label{defhavingsmallfibres}
	A discrete opfibration $s\:G\to F$ in $\m{\C\op}{\Cat}$ \dfn{has small fibres} if for every $C\in \C$ the component $s_C$ of $s$ on $C$ has small fibres.
\end{defne}

\begin{rem}[{\cite[\remx{3.8}]{cavigliamesiti_indexedgrothconstr}}]\label{remhavingsmallfibresispbstable}
	The property of having small fibres for a discrete opfibration in $\m{\C\op}{\Cat}$ is stable under pullbacks. Indeed taking components on $C\in \C$ preserves 2-limits in 2-presheaves and discrete opfibrations in $\Cat$ with small fibres are stable under pullbacks.
	
	We denote as $\Fib+[b][\m{\C\op}{\Cat}][F]$ the full subcategory of $\Fib[b][\m{\C\op}{\Cat}][F]$ on the discrete opfibrations with small fibres (that is, denoting the property of having small fibres as $\opn{s}$).
\end{rem}

\begin{prop}\label{Gpf}
	Let $p\:{E}\to{B}$ be a discrete opfibration in \L. For every $F\in \L$, pulling back $p$ extends to a functor
	$$\G{p,F}\:\HomC{\L}{F}{B}\to \Fib[b][n][F].$$
\end{prop}
\begin{proof}
	Given a morphism $z\:F\to B$ in \L, consider the chosen pullback in \L
	\sq[p][6][6]{G^z}{E}{F}{B}{\t{z}}{s^z}{p}{z}
	We define $\G{p,F}(z)$ to be $s^z$, which is a discrete opfibration in \L by \remx\ref{rempullbackopfib}.
	
	Given a $2$-cell $\alpha\:z\aR{}z'\:F\to B$, we induce $\G{p,F}(\alpha)$ by lifting the 2-cell
	\tcwl{F}{B}{z}{z'}{\alpha}{G^z}{s^z}
	to $\t{z}$ along the discrete opfibration $p$. We obtain a unique $2$-cell
	\tc{G^z}{E}{\t{z}}{v}{\t{\alpha}}
	such that $p\c v = z'\c s^z$ and $p \t{\alpha}=\alpha s^z$. We define $\G{p,F}(\alpha)$ to be the morphism from $G^z$ to $G^{z'}$ induced by $s^z$ and $v$ via the universal property of the pullback $G^{z'}$. We may represent this with the following diagram:
	\begin{eqD}{1}
		\begin{cd}*[4][6.55]
			G^z \arrow[rd,dashed,"{\G{p,F}(\alpha)}"'{pos=0.5,description,inner sep=-0.5ex}]\arrow[rrd,bend left,"{\t{z}}",""'{name=A,pos=0.55}] \arrow[dd,"{s^z}"'{pos=0.35}]\\[-4ex]
			\&[-2ex] G^{z'} \arrow[r,"{\t{z'}}"] \arrow[Rightarrow,from=A,shift left=0.6ex,"{\t{\alpha}}"{pos=0.3}]\& E\arrow[dd,"{p}"]\\
			F\arrow[rd,equal] \arrow[rrd,bend left,"{z}"{pos=0.53},""'{name=B,pos=0.55}]\\[-4ex]
			\& F \arrow[uu,leftarrow,crossing over,"{s^{z'}}"{pos=0.68}]\arrow[r,"{z'}"'] \arrow[Rightarrow,from=B,shift left=0.6ex,"{\alpha}"{pos=0.3}] \& B
		\end{cd}
	\end{eqD}
	It is straightforward to show that $\G{p,F}$ is a functor, using the universal property of the pullback and the uniqueness of the liftings through the discrete opfibration $p$.
\end{proof}

\begin{notation}
	Given a morphism $z\:F\to B$ in \L, we will denote as $\dG[p,F]{z}$ or just as $\dG{z}$ the domain of the discrete opfibration $\G{p,F}(z)$. We will also often draw the action of $\G{p,F}$ on morphisms as in the diagram of equation~\refs{1}. Sometimes, we will denote the functor $\G{p,F}$ as $\G{p}$.
\end{notation}

\begin{prop}[{\cite[\prox{3.9}]{cavigliamesiti_indexedgrothconstr}}]\label{DOpFib-}
	The assignment $F\in \L\mto \Fib[b][n][F]\in \CATlarge$ extends to a pseudofunctor 
	$$\Fib[b][n][-]\:\L\op\to \CATlarge.$$
	Moreover this restricts to a pseudofunctor $\Fib?[b][n][-]$ that sends $F\mto\Fib?[b][n][F]$.
\end{prop}
\begin{proof}(Definition of the assignment). On the underlying category of $\L\op$, we define $\Fib[b][n][-]$ as the restriction of the pseudofunctor given by the pullback (thanks to \remx\ref{rempullbackopfib}). So, given a morphism $H\ar{y} F$ in $\L$, we define
	$$\Fib[b][n][y]\deq y\st\:\Fib[b][n][F]\to\Fib[b][n][H]$$
	
	Given a $2$-cell $\alpha\:y\aR{}y'\:H\to F$ in $\L$, we define $\Fib[b][n][\alpha]=\alpha\st\:y\st\to {(y')}\st$ as the natural transformation whose component on a discrete opfibration $s\:G\to F$ in $\L$ is $\G{s,H}(\alpha)$.
\end{proof}

\begin{prop}\label{pseudonatG}
	Let $p\:{E}\to{B}$ be a discrete opfibration in \L. The functors $\G{p,F}$ are pseudo-natural in $F\in \L$.
\end{prop}
\begin{proof}
	The proof is straightforward, using the universal property of the pullback and the uniqueness of the liftings through a discrete opfibration.
\end{proof}

\begin{defne}[Weber~{\cite[\defx{4.1}]{weber_yonfromtwotop}}]\label{def2clas}
	A \dfn{$2$-classifier in $\L$} is a discrete opfibration $p\:{E}\to{B}$ in \L such that for every $F\in \L$ the functor
	$$\G{p,F}\:\HomC{\L}{F}{B}\to \Fib[b][n][F]$$
	is fully faithful.
	
	In that case, we say that a discrete opfibration $s\:{G}\to{F}$ in \L is \dfn{classified by $p$} (or that \dfn{$p$ classifies $s$}) if $s$ is in the essential image of $\G{p,F}$, and we call \dfn{characteristic morphism of $s$} a morphism $z\:F\to B$ such that $\G{p,F}(z)\iso s$.
\end{defne}

\begin{rem}
	While \defx\ref{defsubclas} asks for a universal monomorphism, \defx\ref{def2clas} asks for a universal discrete opfibration. The classification process is kept to be regulated by pullbacks. The condition to have a bijection is upgraded in dimension 2 to ask $\G{p,F}$ to be an equivalence of categories with its essential image. Notice that \defx\ref{def2clas} allows for a classification of a smaller class of discrete opfibrations. In dimension $1$, this idea brought for example to the concept of quasitopos.
	
	However, \defx\ref{def2clas} loses the interpretation of the subobject classifier as picking a verum inside an object of generalized truth values. Indeed the domain of a 2-classifier is not forced at all to be the terminal object. In order to keep such point of view, which is useful for categorical logic, we propose to upgrade the 1-dimensional classification process, which is regulated by pullbacks, to one regulated by comma objects. This is slightly less general than \defx\ref{def2clas}, but with better properties (see the following two remarks).
\end{rem}

\begin{defne}\label{defgood2clas}
	Let $\opn{P}$ be a fixed pullback-stable property $\opn{P}$ for discrete opfibrations in $\L$. A \dfn{good $2$-classifier in $\L$} (with respect to $\opn{P}$) is a morphism $\omega\:1\to \O$ in $\L$ such that for every $F\in \L$ the functor
	$$\Gg{\omega,F}\:\HomC{\L}{F}{\O}\to \Fib[b][n][F]$$
	given by taking comma objects from $\omega$ is fully faithful and forms an equivalence of categories when restricting the codomain to $\Fib?[b][n][F]$. (In particular, we are asking that $\Gg{\omega,F}$ lands in $\Fib?[b][n][F]$).
\end{defne}

In the following two remarks, we show that $\Gg{\omega,F}$ is indeed a functor which lands in $\Fib[b][n][F]$ and that good 2-classifiers are 2-classifiers enjoying better properties.

\begin{rem}\label{remreplacement}
	Taking comma objects from $\omega$ is equivalent to pulling back the lax limit $\tau$ of the arrow $\omega$, which serves as a replacement.
	\begin{eqD*}
		\sq*[l][6][6][\h[-3]\v[2]\opn{comma}]{G}{1}{F}{\O}{}{\Gg{\omega,F}(z)}{\omega}{z} \quad = \quad 
		\begin{cd}*[5.8][5.8]
			G\PB{rd} \arrow[d,"{\G{\tau,F}(z)}"'] \arrow[r] \& \O\b \arrow[d,"{\tau}"'] \arrow[r]\& 1 \arrow[d,"{\omega}"]\arrow[ld,Rightarrow,"{\text{comma}}"{inner sep=0.3ex,pos=0.6},shorten <=2.7ex, shorten >=2.2ex] \\
			F\arrow[r,"z"'] \& \O \arrow[r,equal]\& \O
		\end{cd}
	\end{eqD*}
	Moreover, by Weber's~\cite[\thex{2.11}]{weber_yonfromtwotop}, the span with vertex $\O\b$ formed by $\tau$ and the map to $1$ is a bisided discrete fibration. And we get that $\tau$ is a discrete opfibration. (In $\L=\Cat$, it is also known that such a $\tau$ is the free opfibration on the functor $\omega$.) Explicitly, the lifting of $\theta\:\tau\c a\to b$ to $a$ is calculated by applying the universal property of the comma object. Indeed $\theta\stb a$ is induced by 
	\begin{cd}[5.8][5.8]
		X \arrow[rd,bend right=20,"{b}"',""{name=A}]\arrow[r,"{a}"]\& \O\b\arrow[d,"{\tau}"]\arrow[r,"{}"] \arrow[to=A,Rightarrow,"{\theta}"]\& 1\arrow[ld,Rightarrow,"{\opn{comma}}",shorten <=2.7ex,shorten >=2.2ex] \arrow[d,"{\omega}"]\\
		\& \O \arrow[r,equal] \& \O
	\end{cd}
	and $\ov{\theta}^a\:a\aR{}\theta\stb a$ is then induced by the pair of 2-cells formed by the identity (between $X$ and $1$) and $\theta$ (between $X$ and $\O$).
	
	By \remx\ref{rempullbackopfib}, it follows that $\Gg{\omega,F}$ lands in $\Fib[b][n][F]$. Moreover, $\Gg{\omega,F}$ lands in $\Fib?[b][n][F]$ if and only if $\tau$ satisfies $\opn{P}$. It is easy to see that, up to choosing appropriate representatives of the comma objects, $\Gg{\omega,F}=\G{\tau,F}$. So that if $\omega$ is a good 2-classifier, $\tau$ is a 2-classifier.
\end{rem}

\begin{rem}\label{remgood2clas}
	Good 2-classifiers enjoy better properties than 2-classifiers. They are closer to the point of view of categorical logic, as they can still be thought of as the inclusion of a verum inside an object of generalized truth values. Moreover, a classification process regulated by taking comma objects from a morphism that has the terminal object as domain is sometimes easier to handle. As an example, the assignment of $\Gg{\omega,F}$ on morphisms is just induced by the universal property of the comma object, while for $\G{p,F}$, in \prox\ref{Gpf}, we had to consider liftings along a discrete opfibration. Moreover such classification process regulated by comma objects offers another justification for the idea of upgrading subobjects to discrete opfibrations. Indeed discrete opfibrations are what gets produced by taking comma objects from a morphism that has the terminal object as domain.
	
	In our examples, $\opn{P}$ will be the property of having small fibres. In some sense, our good 2-classifiers will classify ``all possible morphisms", as the $\Gg{\omega,F}$'s are required to be equivalences of categories. 
\end{rem}

\begin{exampl}\label{excat}
	$\CAT$ is the archetypal example of $2$-category endowed with a (good) $2$-classifier. Consider indeed $\omega=1\:\1\to \Set$. For every $\B\in \Cat$, the functor
	$$\Gg{\omega,\B}\:\HomC{\Cat}{\B}{\Set}\to \Fib[b][n][\B]$$
	is precisely the Grothendieck construction (or category of elements). It is well-known that this forms an equivalence of categories when restricting the codomain to be the full subcategory $\Fib+[b][n][\B]$ of $\Fib[b][n][\B]$ on the discrete opfibrations with small fibres. So that $1\:\1\to \Set$ is a good 2-classifier in $\Cat$.
	\begin{cd}[6.75][6.75]
		\E \arrow[r,""] \arrow[d,"{\forall p\hphantom{.}}"'{pos=0.25},"\text{disc opfib}"'{pos=0.55},"\text{small fibres}"'{pos=0.75}] \& {\1}\arrow[ld,Rightarrow,"{\text{comma}}"{inner sep=0.3ex,pos=0.6},shorten <=2.7ex, shorten >=2.2ex] \arrow[d,"{\1}"] \\
		\B \arrow[r,dashed,"{\exists \chi_p}"',"{\text{taking fibres}}"'{inner sep = 2.4ex}] \& {\Set}
	\end{cd}\v[-5.5]
	
	\noindent Notice that the lax limit of the arrow $\omega$ is given by the forgetful functor $\tau\:\Set\b\to \Set$ from pointed sets to sets.
\end{exampl}

\begin{rem}
	In light of this archetypal example, we can think of a 2-classifier as a Grothendieck construction inside a 2-category.
\end{rem}

\begin{notation}
	We will often write as $\tau\:\O\b\to \O$ any 2-classifier or would-be 2-classifier, having in mind the archetypal example of $\Cat$.
\end{notation}

\begin{rem}\label{remisoclasses}
	Upgrading monomorphisms to discrete opfibrations, one could try to upgrade $\Sub{F}$ to a category of isomorphism classes of discrete opfibrations over $F$. It is possible to form such a category and almost the entire reduction to dense generators of the study of 2-classifiers would equally work (if accordingly adjusted). However, there is one point, in \thex\ref{fullness}, that does not seem to work well with this choice. We will give more details in \remx\ref{remfullness}. We believe it is more natural and fruitful to work without isomorphism classes.
\end{rem}

\subsection{Dense generators}\label{subsectiondensegen}

In Section~\ref{sectionreductiontogenerators}, we will reduce to dense generators the study of 2-classifiers. Here we briefly recall what a (2-dimensional) dense generator is. The main reference we take for this is Kelly's~\cite[Chapter~5]{kelly_basicconceptsofenriched}.

The basic idea behind the concept of generator of a 2-category $\L$ is that of a family of objects that builds all the objects of $\L$ via weighted 2-colimits.

\begin{defne}\label{defnaivegen}
	A fully faithful 2-functor $I\:\K\to \L$ is a \dfn{naive generator} if every $F\in \L$ is a weighted 2-colimit in $\L$ of a diagram which factors through $\K$.
\end{defne}

\begin{defne}[{\cite[Section~5.1]{kelly_basicconceptsofenriched}}]\label{defdense}
	A 2-functor $I\:\Y\to \L$ with $\Y$ small is \dfn{a dense generator} (or also just \dfn{dense}) if the restricted Yoneda embedding
	\begin{fun}
		\t{I} & \: & \L & \too & \m{\Y\op}{\CAT} \\[0.7ex]
		&& F &\mto& \HomC{\L}{I(-)}{F}
	\end{fun}
	is (2-)fully faithful.
\end{defne}

\begin{rem}
	Of course the Yoneda embedding is fully faithful. And we may interpret this by saying that considering all morphisms from any object into $F$ we get the whole information of $F$. The idea of a dense generator is that morphisms from a smaller family of objects are enough.
\end{rem}

\begin{defne}
	Let $I\:\Y\to \L$ be a 2-functor. A weighted 2-colimit in $\L$ is called \dfn{$I$-absolute} if it is preserved by $\t{I}\:\L\to \m{\Y\op}{\CAT}$.
\end{defne}

When $I\:\Y\to \L$ is fully faithful, we can characterize density of $I$ in terms of weighted 2-colimits in $\L$.

\begin{teor}[{\cite[\thex{5.19}]{kelly_basicconceptsofenriched}}]\label{densityintermsofcolim}
	Let $I\:\Y\to \L$ be a fully faithful 2-functor. The following are equivalent:
	\begin{enumT}
		\item $I$ is a dense generator;
		\item every $F\in \L$ is an $I$-absolute weighted 2-colimit in $\L$ of a diagram which factors through $\Y$.
	\end{enumT}
\end{teor}

\begin{rem}
	We can thus interpret density of a fully faithful $I\:\Y\to \L$ as the request that all objects of $\L$ are nice weighted 2-colimits of objects of $\Y$. So this is stronger than being a naive generator.
\end{rem}

\begin{exampl}\label{exarepresinprestacks}
	Let $\C$ be a small 2-category. The Yoneda embedding $\yy\:\C\to \m{\C\op}{\Cat}$ is a dense generator. That is, representables form a dense generator of the 2-category of 2-presheaves. Indeed it is well-known that every 2-presheaf is a weighted 2-colimit of representables, weighted by itself. And $\yy$-absoluteness is automatic as $\t{\yy}$ is essentially the identity.
	
	In particular the singleton category $\1$ is a dense generator of $\Cat$.
\end{exampl}

\subsection{Stacks}\label{subsectionstacks}

We consider $\Cat$-valued stacks that, for simplicity, have a $1$-category as domain. The stacks we consider have the usual gluing condition that gives an equivalence of categories between the image on an object $C$ and each category of descent data on $C$. We recall below the explicit gluing conditions.

We will use the language of sieves, rather than the one of covering families. This simplifies the form of the conditions of stack and will make it easier for us to prove that our 2-classifier in stacks is indeed a stack. Moreover, sieves are also what forms the subobject classifier of sheaves, in dimension 1. The less standard equivalent definition of stack that we write below can be obtained unravelling Street's~\cite[Section~2]{street_charactbicategoriesstacks} abstract definition of stack, in the (more usual) case of a $1$-dimensional Grothendieck topology.

\begin{defne}\label{defsieve}
	Let $\C$ be a category. A \dfn{sieve} $S$ on $C\in \C$ is a collection of morphisms with codomain $C$ that is closed under precomposition with any morphism of $\C$.
	
	Equivalently, a sieve $S$ on $C$ is a subfunctor of the representable $\y{C}$.
	
	The \dfn{maximal sieve} is the collection of all morphisms with codomain $C$, or equivalently the identity on $\y{C}$.
\end{defne}

\begin{notation}
	Given a pseudofunctor $F\:\C\op \to \Cat$ and a morphism $g\:D'\to D$ in $\C$, we denote as $g\st$ the functor $F(g)$.
\end{notation}

The following definition upgrades the concept of matching family to dimension $2$. The compatibility under descent of the local data is up to isomorphism.

\begin{defne}\label{defdescentdatum}
	Let $F\: \C\op \to \Cat$ be a pseudofunctor and let $S$ be a sieve on $C\in \C$. A \dfn{descent datum for $F$ with respect to $S$} is an assignment
	$$(D\ar{f}C)\in S \quad\am{m}\quad M_f\in F(D)$$
	together with, for all composable morphisms
	$D'\ar{g}D\ar{f}C$
	with $f\in S$, an isomorphism
	$${\phi^{f,g}} \: g\st M_f \aisoo M_{g\c f}$$ such that, for all composable morphisms
	$D''\ar{h} D' \ar{g}  D\ar{f} C$
	with $f\in S$, the following cocycle condition holds:
	\begin{cd}[5][5]
		h\st g\st M_f \arrow[d,"", iso] \arrow[r,"h\st \phi^{f,g}"] \&   h\st M_{f\c g} \arrow[d,"\phi^{f\c g,h}"] \\
		(g\c h)\st M_f \arrow[r,"\phi^{f,g\c h}"']\& {M_{f\c g\c h}.}
	\end{cd}
\end{defne}

The following definition upgrades the concept of amalgamation for a matching family to dimension 2. The global data produced only recovers the starting local data up to isomorphism.

\begin{defne}\label{defeffectivedescentdatum}
	In the notation of \defx\ref{defdescentdatum}, a descent datum $m$ for $F$ with respect to $S$ is \dfn{effective} if there exists an object $M\in F(C)$ together with, for every morphism $f\:D\to C$ in $S$, an isomorphism
	$$\psi^f\:f\st M \aisoo M_f$$
	such that, for all composable morphisms
	$D'\ar{g} D \ar{f} C$
	with $f\in S$
	\begin{cd}[5][5]
		g\st f\st M \arrow[d,"",iso] \arrow[r,"g\st \psi^{f}"] \&   {g\st M_f} \arrow[d,"\phi^{f,g}"] \\
		(f\c g)\st M\arrow[r,"\psi^{f\c g}"']\& {M_{f\c g}.}
	\end{cd}
\end{defne}

\begin{rem}
	Notice that the square in \defx\ref{defeffectivedescentdatum} is very similar to the one of \defx\ref{defdescentdatum}. An object $M$ that makes a descent datum $m$ effective plays the role of an $M_{\id{}}$, although the identity belongs to the sieve if and only if the sieve is maximal.
\end{rem}

\begin{defne}\label{defstack}
	Let $\C$ be a category equipped with a Grothendieck topology $J$. A pseudofunctor $F\: \C\op \to \Cat$ is a \dfn{stack \pteor{with respect to $J$}} if it satisfies the following three conditions for every $C\in \C$ and covering sieve in $J$ on $C$:
	\begin{enum}
		\item (gluing of objects) every descent datum for $F$ with respect to $S$ is effective;
		\vspace{2.5mm}
		\item ({gluing of morphisms}) for all $X,Y\in F(C)$ and for every assignment to each $f\: D\to C$ in $S$ of a morphism $h_f\: f\st X \to f\st Y$ in $F(D)$ such that $g\st(h_f)=h_{f\c g}$ for all composable morphisms $D'\ar{g}D\ar{f}C$, there exists a morphism $h\: X \to Y$ such that $f\st h =h_f$;
		\vspace{2.5mm}
		\item ({uniqueness of gluings of morphisms}) for all $X,Y\in F(C)$ and morphisms $h,k:X\to Y$ such that $f\st h=f\st k$ for every $f\: D\to C$ in $S$, it holds that $h=k$.
	\end{enum}
\end{defne}

\begin{rem}
	Conditions $(ii)$ and $(iii)$ of \defx\ref{defstack} may be interpreted as saying that $F$ is a sheaf on morphisms.
\end{rem}

\begin{teor}[Street~{\cite[Section~2]{street_charactbicategoriesstacks}}]\label{stacksarebireflective}
	Stacks form a bireflective sub-2-category of the $2$-category $\Psm{\C\op}{\Cat}$ of pseudofunctors, pseudo-natural transformations and modifications.
\end{teor}

\begin{rem}\label{remourstacks}
	As the notion of 2-classifier is rather strict, we will consider strictly functorial stacks, so that they form a full sub-2-category of the functor 2-category $\m{\C\op}{\Cat}$. We will also take a subcanonical Grothendieck topology, so that all representables are sheaves (and hence stacks). We keep however the usual gluing conditions written above, that give an equivalence of categories between $F(C)$ and each category of descent data on $C$ with respect to a covering sieve $S$. 
	
	In future work, we will produce a suitable classifier for the usual pseudofunctorial stacks.
\end{rem}

\subsection{Cartesian-marked oplax conical colimits}\label{subsectioncartmarkedoplaxcolim}

In Section~\ref{sectionreductiontogenerators}, to prove our theorems of reduction to dense generators, we will need a calculus of colimits in 2-dimensional slices. We explored such a calculus in~\cite{mesiti_colimitsintwodimslices}. A key ingredient is the reduction of weighted 2-colimits to cartesian-marked oplax conical ones, that we now recall from Street's~\cite[\thex{15}]{street_limitsindexedbycatvalued}. See also our~\cite[\thex{2.19}]{mesiti_twosetenrgrothconstrlaxnlimits} for new, more elementary proofs. Such reduction is regulated by the 2-category of elements construction, which is a natural extension of the Grothendieck construction that admits $\Cat$-valued presheaves on 2-categories (see~\cite{street_limitsindexedbycatvalued,mesiti_twosetenrgrothconstrlaxnlimits}).

Cartesian-marked oplax conical colimits are a particular case of a general notion of (co)limit introduced by Gray in~\cite[\defx{I,7.9.1}]{gray_formalcattheory}. More recently, Descotte, Dubuc and Szyld brought attention to (a pseudo version of) this concept in~\cite{descottedubucszyld_sigmalimandflatpseudofun}, with the name \dfn{sigma limits}.

\begin{defne}
	Let $W\:\A\op\to \CAT$ be a $2$-functor with $\A$ small, and consider $2$-functors $M,N\:{\left(\Groth{W}\right)}\op\to \D$, where $\groth{W}\:\Groth{W}\to \A$ is the 2-category of elements of $W$. A \dfn{cartesian-marked oplax natural transformation $\alpha$ from $M$ to $N$}, denoted $\alpha\:M\aoplaxn{}N$,\v is an oplax natural transformation $\alpha$ from $M$ to $N$ such that the structure $2$-cell on every morphism $\p{f,\id{}}\:(A,W(f)(X))\al{} (B,X)$ in ${\left(\Groth{W}\right)}\op$ is the identity.
\end{defne}

\begin{defne}\label{defoplaxnormalconical}
	Let $W\:\A\op\to \CAT$ be a $2$-functor with $\A$ small, and let $F\:\Groth{W}\to \C$ be a $2$-functor. The \dfn{cartesian-marked oplax conical $2$-colimit of $F$}, which we will denote as $\oplaxncolim{F}$, is (if it exists) an object $C\in \C$ together with a $2$-natural isomorphism of categories
	$$\HomC{\C}{C}{U}\iso \HomC{\moplaxn{{\left(\Grothdiag{W}\right)}\op}{\CAT}}{\Delta 1}{\HomC{\C}{F(-)}{U}},$$
	where $\moplaxn{{\left(\Grothdiag{W}\right)}\op}{\CAT}$ is the 2-category of 2-functors from ${\left(\Grothdiag{W}\right)}\op$ to $\CAT$, cartesian-marked oplax natural transformations and modifications.
	
	\noindent When $\oplaxncolim{F}$ exists, the identity on $C$ provides a cartesian-marked oplax natural transformation $\mu\:\Delta 1 \aoplaxn{}\HomC{\C}{F(-)}{C}$
	called the \dfn{universal cartesian-marked oplax cocone}.
\end{defne}

\begin{teor}[Street~{\cite[\thex{15}]{street_limitsindexedbycatvalued}}, new proof in our~{\cite[\thex{2.19}]{mesiti_twosetenrgrothconstrlaxnlimits}}]\label{redoplaxnormalconical}
	Every weighted $2$-colimit can be reduced to a cartesian-marked oplax conical one. Given $2$-functors $F\:\A\to \C$ and $W\:\A\op\to\CAT$ with $\A$ small,
	$$\wcolim{W}{F}\iso \oplaxncolim{\left(F\c \groth{W}\right)}$$
	where $\groth{W}\:\Groth{W}\to\A$ is the $2$-category of elements of $W$.
\end{teor}

\begin{prop}[{\cite[\remx{2.20}]{mesiti_twosetenrgrothconstrlaxnlimits}}]\label{propstillpreserved}
	A weighted $2$-colimit is preserved or reflected precisely when its associated cartesian-marked oplax conical colimit is so.
\end{prop}

\begin{exampl}\label{univoplaxnormalcoconepresheaves}
	Every $2$-presheaf $W\:\A\op\to\CAT$ with $\A$ small can be expressed as
	$$W\iso \wcolim{W}{\yy}\iso \oplaxncolim{\left(\yy\c\h{\groth{W}}\right)}.$$
	The universal cartesian-marked oplax cocone is given by
	$$\forall \fib[o][2.4]{(B,X')}{(f,\nu)}{(A,X)} \text{ in } {\Grothdiag{W}} \quad\quad\quad{
		\begin{cd}*[3.8][1.8]
			\y{A} \arrow[rr,"{\ceil{X}}",""'{name=Q}] \arrow[d,"{\y{f}}"'] \&\& W\\
			\y{B} \arrow[rru,bend right,"{\ceil{X'}}"']\& \phantom{.} \arrow[Rightarrow,from=Q,"{\ceil{\nu}}"{pos=0.36},shift right=1.5ex,shorten <=1ex,shorten >= 2.7ex]
	\end{cd}}$$
	
	In particular, taking $\A=\1$, $W$ is a small category $\D$ and $\groth{W}$ is ${\D}\to{\1}$. We obtain that $\1$ is ``cartesian-marked oplax conical dense", building $\D$ with universal cocone
	$$\forall \fib[o]{D}{f}{C} \text{ in } \D \quad\quad\quad {\begin{cd}*[4][2]
			\1 \arrow[rr,"C",""'{name=Q}] \arrow[d,equal] \&\& \D\\
			\1 \arrow[rru,bend right,"D"']\& \phantom{.} \arrow[Rightarrow,from=Q,"{f}"{pos=0.355},shift right=1.5ex,shorten <=0.8ex,shorten >= 2.45ex]
	\end{cd}}$$
\end{exampl}

\section{Reduction of 2-classifiers to dense generators}\label{sectionreductiontogenerators}

In this section, we present a novel reduction of the study of a 2-classifier to dense generators. This is organized in the three Theorems~\ref{faith},~\ref{fullness}~and~\ref{esssurj}.
 More precisely, we prove that all the conditions of 2-classifier (see \defx\ref{def2clas}) can just be checked on those objects $F$ that form a dense generator (see \defx\ref{defdense}). The study of what is classified by a 2-classifier is similarly reduced to a study over the objects that form a dense generator. We also give a concrete recipe to build the characteristic morphisms.

 This result offers great benefits. For example, applied to $\Cat$, it reduces all the major properties of the Grothendieck construction to the trivial observation that everything works well over the singleton category (see \exax\ref{exampleredincat}).

We will apply our theorems of reduction to dense generators of this section to the cases of 2-presheaves (i.e.\ prestacks) and stacks, in Sections~\ref{sectiontwoclassintwopresheaves} and~\ref{sectiontwoclassinstacks}. This will allow us to find a good 2-classifier in 2-presheaves that classifies all discrete opfibrations with small fibres and to restrict it to a good 2-classifier in stacks.

In \thex\ref{teorfactorization}, we prove a general result of restriction of good 2-classifiers to nice sub-2-categories, involving factorization arguments and our theorems of reduction to dense generators. This is what we will use to produce our good 2-classifier in stacks.

Throughout this section, we fix a discrete opfibration $\tau\:{\O\b}\to{\O}$ in \L. Recall from \defx\ref{def2clas} that $\tau$ is a 2-classifier if for every $F\in \L$ the functor
$$\G{\tau,F}\:\HomC{\L}{F}{\O}\to \Fib[b][n][F]$$
is fully faithful. And that, provided that this is the case, the essential image of such functors precisely represents which discrete opfibrations are classified.

The following proposition will often be useful.

\begin{prop}\label{corollpseudonatG}
	Let $p\:{E}\to{B}$ be a discrete opfibration in \L. For every pair of composable morphisms $H\ar{y} F\ar{z} B$ in $\L$, the pseudo-naturality of $\G{p,-}$ \pteor{see \prox\ref{pseudonatG}} gives isomorphisms
	\begin{equation}\label{eqiso}
		\G{p,H}(z\c y)\iso y\st \G{p,F}(z) = \G{\G{p,F}(z),H}(y),
	\end{equation}
	where $y\st$ is the functor $\Fib[b][n][y]$ defined in the proof of \prox\ref{DOpFib-}.
	
	Moreover, given a diagram\v[-2]
	\begin{cd}
		H\arrow[r,bend left,"{y}",""'{name=A}]\arrow[r,bend right,"{y'}"',""{name=B}]\&F \arrow[r,bend left,"{z}",""'{name=C}]\arrow[r,bend right,"{z'}"',""{name=D}]\&B
		\arrow[from=A,to=B,Rightarrow,"{\alpha}"]
		\arrow[from=C,to=D,Rightarrow,"{\beta}"]
	\end{cd}
	in \L, the isomorphisms above form the following two commutative squares:
	\begin{eqD*}
		\begin{cd}*[5.2][5.2]
			\G{p,H}(z\c y) \arrow[d,"{\G{p,H}(z\h\alpha)}"'] \arrow[r,aiso] \& \G{\G{p,F}(z),H}(y) \arrow[d,"{\G{\G{p,F}(z),H}(\alpha)}"] \\
			\G{p,H}(z\c y') \arrow[r,aiso]\& \G{\G{p,F}(z),H}(y')
		\end{cd}\quad
		\begin{cd}*[5.2][5.2]
			\G{p,H}(z\c y) \arrow[d,"{\G{p,H}(\beta\h y)}"'] \arrow[r,aiso] \& \G{\G{p,F}(z),H}(y) \arrow[d,"{y\st \G{p,F}(\beta)}"] \\
			\G{p,H}(z'\c y) \arrow[r,aiso]\& \G{\G{p,F}(z'),H}(y)
		\end{cd}
	\end{eqD*}
\end{prop}
\begin{proof}
	The proof is straightforward. The first square is given by the $2$-dimensional part of the pseudo-naturality of $\G{p,-}$ applied to the $2$-cell $\alpha$, whereas the second square is given by the naturality in $z$ of the isomorphisms of equation~\refs{eqiso}.
\end{proof}

	We now present the first of our three theorems of reduction to dense generators. This reduces the study of the faithfulness of the functors $\G{\tau,F}$. Such first theorem is much easier than the other two and actually works with any naive generator.

	We also notice that injectivity on objects of the functors $\G{\tau,F}$ can be reduced in a similar way, although this is less interesting for us.

\begin{teor}\label{faith}
	Let \Y be a full subcategory of \L. If for every $Y\in \Y$
	$$\G{\tau,Y}\:\HomC{\L}{Y}{\O}\to \Fib[b][n][Y]$$
	is faithful, then for every $F$ in the closure of $\Y$ in $\L$ under weighted $2$-colimits, also
	$$\G{\tau,F}\:\HomC{\L}{F}{\O}\to \Fib[b][n][F]$$
	is faithful.
	
	In particular, if $\Y$ is a naive generator of $\L$ and $\G{\tau,Y}$ is faithful for every $Y\in \Y$, then $\G{\tau,F}$ is faithful for every $F\in \L$.
\end{teor}
\begin{proof}
	Let $K\:\I \to \L$ be a $2$-diagram which factors through $\Y$ and has a weighted $2$-colimit $F$ in \L, with weight $W:\I\op \to \CAT$. Call $\Lambda\: W\aR{} \HomC{\L}{K(-)}{F}$ the universal cocylinder of such colimit.
	
	In order to prove that $\G{\tau,F}$ is faithful, take two arbitrary $2$-cells in \L
	\begin{eqD*}
		\tc*{F}{\O}{z}{z'}{\alpha} \quad\text{ and }\quad \tc*{F}{\O}{z}{z'}{\alpha'}
	\end{eqD*}
	such that $\G{\tau,F}(\alpha)=\G{\tau,F}(\alpha')\:\dG{z}\to \dG{z'}$. We prove that $\alpha=\alpha'$.
	
	As $F=\wcolim{W}{K}$ with universal cocylinder $\Lambda$, it suffices to show that the two modifications
	\tcswltwodim[6]{\HomC{\L}{K(-)}{F}}{\HomC{\L}{K(-)}{\O}}{z\c-}{z'\c-}{\alpha\ast-}{}{}{}{W}{\Lambda}
	and $(\alpha'\ast-)\Lambda$ are equal, as we then conclude by the ($2$-dimensional) universal property of the weighted $2$-colimit $F$.
	
	It then suffices to prove that, given arbitrary $i\in \I$ and $X\in W(i)$,
	\begin{eqD*}
		\left(\tcwl*{F}{\O}{z}{z'}{\alpha}{K(i)}{\Lambda_i(X)}\right) \quad = \quad \left(  \tcwl*{F}{\O}{z}{z'}{\alpha'}{K(i)}{\Lambda_i(X)}\right)
	\end{eqD*}
	But as $K(i)\in \Y$ and $\G{\tau,K(i)}$ is faithful by assumption, it then suffices to show that
	\begin{equation}\label{eqDi}
		\G{\tau,K(i)}\left(\alpha\,\Lambda_i(X)\right)=\G{\tau,K(i)}\left(\alpha'\,\Lambda_i(X)\right).
	\end{equation}
	By \prox\ref{corollpseudonatG}, we can write $\G{\tau,K(i)}\left(\alpha\,\Lambda_i(X)\right)$ as the composite
	$$\G{\tau,K(i)}(z\c \Lambda_i(X))\iso \Lambda_i(X)\st\left(\G{\tau,F}(z)\right)\ar{\Lambda_i(X)\st\left(\G{\tau,F}(\alpha)\right)} \Lambda_i(X)\st\left(\G{\tau,F}(z')\right)\iso\G{\tau,K(i)}(z'\c \Lambda_i(X))$$
	and analogously for $\alpha'$. Since such composites for $\alpha$ and for $\alpha'$ are equal, we conclude that equation~\refs{eqDi} holds.
\end{proof}

\begin{rem}
	In order to prove our second and third theorems of reduction to dense generators, we apply our calculus of colimits in 2-dimensional slices, that we explored in~\cite{mesiti_colimitsintwodimslices}. Such calculus generalizes to dimension 2 the well-known fact that a colimit in a 1-dimensional slice is precisely the map from the colimit of the domains of the diagram which is induced by the universal property. It is based on the reduction of weighted 2-colimits to cartesian-marked oplax conical ones, recalled in~\ref{subsectioncartmarkedoplaxcolim}, and on $\F$-category theory (also called enhanced 2-category theory), for which we take as main reference Lack and Shulman's~\cite{lackshulman_enhancedtwocatlimlaxmor}. 
	
	We first need to compare, given $p$ a discrete opfibration in $\L$, the functor $\G{p,F}$ with the 2-functor $p\st$ of change of base between lax slices introduced in~\cite{mesiti_colimitsintwodimslices}.
\end{rem}

\begin{prop}[\cite{mesiti_colimitsintwodimslices}]\label{proptaustextendstoatwofunctor}
	Let $p\:E\to B$ be a discrete opfibration in $\L$. Then pulling back along $p$ extends to a $2$-functor
	$$p\st\:\laxslice{\L}{B}\to \laxslice{\L}{E}$$
	
	Moreover, considering the canonical $\F$-category structure on the lax slice, with the tight part given by the strict slice, $\tau\st$ is an $\F$-functor.
\end{prop}
\begin{proof}(Definition of the assignment). Given $z\:F\to B$ in $\L$, we define $p\st z$ as the upper morphism of the chosen pullback square in $\L$ on the left below. 
	
	Given then a morphism $(\widehat{\alpha},\alpha)\:z\to z'$ in $\laxslice{\L}{B}$ as in the middle below, we can lift the $2$-cell in $\L$ on the right below
	\begin{eqD*}
		\sq*[p][6][7]{\dG{z}}{E}{F}{B}{p\st z}{\G{p,F}(z)}{p}{z}
		\qquad\quad \tr*[3.8][5.2][0][0][l][-0.3][\alpha][0.85]{F}{F'}{B}{\widehat{\alpha}}{z}{z'}
		\qquad\quad
		\begin{cd}*[3.1][3.7]
			\dG{z} \arrow[rr,"{p\st z}"]\arrow[d,"{\G{p,F}(z)}"'] \&\& E \arrow[d,"{p}"] \\[-0.4ex]
			F \arrow[rd,"{\widehat{\alpha}}"'{inner sep=0.3ex}]\arrow[rr,"{z}",""'{name=A}] \&\& B \\
			\& F' \arrow[from=A,Rightarrow,"{\alpha}"{pos=0.4},shorten <=0.5ex,shorten >=0.9ex]\arrow[ru,"{z'}"'{inner sep =0.5ex}]
		\end{cd}
	\end{eqD*}
	along the discrete opfibration $p$, producing the unique $2$-cell $p\st{\alpha}\:p\st z\aR{}v\:\dG{z}\to E$ such that $p\c v= z'\c \widehat{\alpha}\c \G{p,F}(z)$ and $p\ast p\st{\alpha}=\alpha\ast \G{p,F}(z)$. By the universal property of the pullback $\dG{z'}$ of $p$ and $z'$, we factorize $v$ through $p\st z'$, obtaining a morphism $\widehat{p\st \alpha}\:\dG{z}\to \dG{z'}$. We define $p\st (\widehat{\alpha},\alpha)$ to be the upper triangle in the following commutative solid:
	\begin{eqD}{pstar}
	\begin{cd}*[4][6.5]
		\dG{z} \arrow[rd,dashed,"{\widehat{p\st \alpha}}"'{pos=0.5,description,inner sep=-0.5ex}]\arrow[rrd,bend left,"{p\st z}",""'{name=A,pos=0.55}] \arrow[dd,"{\G{p,F}(z)}"'{pos=0.35}]\\[-4ex]
		\&[-2ex] G^{z'} \arrow[r,"{p\st {z'}}"] \arrow[Rightarrow,from=A,shift left=0.6ex,"{p\st {\alpha}}"'{pos=0.75}]\& E\arrow[dd,"{p}"]\\
		F\arrow[rd,"{\widehat{\alpha}}"'] \arrow[rrd,bend left,"{z}"{pos=0.53},""'{name=B,pos=0.55}]\\[-4ex]
		\& F' \arrow[uu,leftarrow,crossing over,"{\G{p,F'}(z')}"{pos=0.68}]\arrow[r,"{z'}"'] \arrow[Rightarrow,from=B,shift left=0.6ex,"{\alpha}"{pos=0.3}] \& B
	\end{cd}
	\end{eqD}
	
	Given a $2$-cell $\delta\:(\widehat{\alpha},\alpha)\to (\widehat{\beta},\beta)\:z\to z'$ in $\laxslice{\L}{B}$, we define $p\st \delta$ to be the unique lifting of the $2$-cell $\delta\ast \G{p,F}(z)$ along the Grothendieck opfibration $\G{p,F'}(z')$.
\end{proof}

\begin{rem}\label{remcomparedompstdomgp}
	In order to apply the results of~\cite{mesiti_colimitsintwodimslices}, we need to compare $\dom\c\h p\st$ with $\dom\c\G{p,F}$, given $p\:E\to B$ a discrete opfibration in $\L$.
	
	For every $z\:F\to B$ in $\L$,
	$$\dom(p\st z)=\dom(\G{p,F}(z)).$$
	Given $(\id{},\alpha)\:z\to z'$ in $\laxslice{\L}{B}$, which is just $\alpha\:z\aR{}z'\:F\to B$, we have that
	$$\dom(p\st (\id{},\alpha))=\dom(\G{p,F}(\alpha))$$
	by comparing the diagrams of equations \refs{1} and~\refs{pstar}.
	
	Given a general $(\widehat{\alpha},\alpha)\:z\to z'$ in $\laxslice{\L}{B}$, we can still express $\dom(p\st (\widehat{\alpha},\alpha))=\widehat{p\st \alpha}$ in terms of $\dom(\G{p,F}(\alpha))$. Indeed consider the total pullback $R$ of $p$ with the composite $z'\c\widehat{\alpha}$ and the composite pullback $S$ as below
	\pbsqdrp{S}{\dG{z'}}{E}{F}{F'}{B}{\G{p}(z')\st \widehat{\alpha}}{p\st z'}{\G{\G{p}(z')}(\widehat{\alpha})}{\G{p}(z')}{p}{\widehat{\alpha}}{z'}
	Call $i$ the induced isomorphism between $R$ and $S$. Comparing the diagrams of equations \refs{1} and~\refs{pstar} we obtain that
	$$\dom(p\st (\widehat{\alpha},\alpha))=\G{p,F}(z')\st \widehat{\alpha}\c i\c\dom(\G{p,F}(\alpha))$$
\end{rem}

The following construction will be useful to prove our second and third theorems of reduction to dense generators. It is based on our calculus of colimits in 2-dimensional slices, explored in~\cite{mesiti_colimitsintwodimslices}.

\begin{cons}\label{lemma}
	Let $I\:\Y\to \L$ be a fully faithful dense generator of $\L$, and let $F\in \L$. By \thex\ref{densityintermsofcolim}, there exist a $2$-diagram $J\:\I\to \L$ which factors through $\Y$ and a weight $W\:\I\op\to \CAT$ such that $$F=\wcolim{W}{J}$$ in \L and this colimit is $I$-absolute. By \thex\ref{redoplaxnormalconical}, the $2$-diagram $K\deq J\c \groth{W}\:\Groth{W}\to \L$ factors through $\Y$ and is such that
	$$F=\oplaxncolim{K}.$$
	Moreover this colimit is still $I$-absolute by \prox\ref{propstillpreserved}. Call
	$$\Lambda\: \Delta 1 \aoplaxn{} \HomC{\L}{K(-)}{F}$$
	the universal cartesian-marked oplax cocone that presents such colimit. 
	
	Consider now a discrete opfibration $p\:E\to B$ in $\L$, a morphism $z\:F\to B$ and the chosen pullback in $\L$
	\sq[p][6.2][6.2]{\dG{z}}{E}{F}{B}{p\st z}{\G{p,F}(z)}{p}{z}
	We want to exhibit $\dG{z}$ as a cartesian-marked oplax conical colimit of a diagram constructed from $K$ and $\Lambda$.
	
	By~\cite{mesiti_colimitsintwodimslices}, we can construct from $K$ and $\Lambda$ a $2$-diagram $K'_z=K'\:\Groth{W}\to \laxslice{\L}{B}$ and a universal cartesian-marked oplax cocone $\Lambda'_z=\Lambda'$ which exhibits
	$$\fib{F}{z}{B}=\fib{\oplaxncolim{K}}{z}{B}=\oplaxncolim{K'}$$
	in the lax slice $\laxslice{\L}{B}$. Explicitly, $K'$ is the 2-diagram that corresponds to the cartesian-marked oplax cocone
	$$\lambda^z\:\Delta 1 \aoplaxn{} \HomC{\L}{K(-)}{B}\:{\left(\Grothdiag{W}\right)}\op\to \CAT$$
	associated to $z$. That is,
	\begin{fun}
		K' & \: & \Grothdiag{W}\hphantom{...} & \too & \hphantom{...} \laxslice{\E}{B} \\[1.3ex]
		&& \fibdiag{(C,X)}{(f,\nu)}{(D,X')} & \mto & \tr*[2][3.2][0][0][l][-0.235][\lambda^z_{f,\nu}][1]{K(C,X)}{K(D,X')}{B}{K(f,\nu)}{\lambda^z_{(C,X)}}{\lambda^z_{(D,X')}}\\[1.3ex]
		&& \delta \hphantom{.....}& \mto & \hphantom{.....} F(\delta)
	\end{fun}
	Considering the canonical $\F$-category structure on $\Groth{W}$, with the tight part given by the morphisms of type $(f,\id{})$, we have that $K'$ is an $\F$-diagram.
	
	The universal cartesian-marked oplax cocone $\Lambda'$ has component on $(C,X)$ given by the identity filled triangle (which is thus a tight morphism in $\laxslice{\L}{B}$)
	\tr[2.3][3.5][-1]{K(C,X)}{F}{B}{\Lambda_{(C,X)}}{\lambda^z_{(C,X)}}{z}
	Then $\Lambda'_{f,\nu}=\Lambda_{f,\nu}$ for every morphism $(f,\nu)$ in $\Groth{W}$.
	
	By~\cite{mesiti_colimitsintwodimslices}, the $2$-functor
	$${p}\st\: \laxslice{\L}{B}\to \laxslice{\L}{E}$$
	preserves the colimit $z=\oplaxncolim{K'}$, since $K'$ is an $\F$-diagram, $\Lambda'$ has tight components and the domain of such colimit is $F=\oplaxncolim{K}$, which is $I$-absolute. Then again by~\cite{mesiti_colimitsintwodimslices} the domain $2$-functor $\dom\:\laxslice{\L}{E}\to\L$ preserves the latter colimit $p\st(z)$, since $p\st$ is an $\F$-functor. We obtain that $\dom\c {p}\st\c \Lambda'$ is a universal cartesian-marked oplax cocone that presents
	$$\dG{z}=\oplaxncolim{\left(\dom\c \h{p}\st \c K'\right)}.$$
	
	Explicitly, given $(D,X')\al{(f,\h\nu)}(C,X)$ in $\Groth{W}$, we have that $\dom\p{p\st\p{\Lambda'_{(C,X)}}}$ is the unique morphism into the pullback $\dG{z}$ induced by
	\begin{cd}[3.35][7.3]
		Q^{(C,X)}\arrow[rd,dashed,"{}",shorten <=-0.75ex] \arrow[dd,"{\G{p}\left(K'(C,X)\right)}"']\arrow[rrd,bend left=25,"{p\st {K'(C,X)}}"{inner sep=0ex}]\&[-5.5ex]\\[-3.5ex]
		\& \dG{z}\arrow[r,"{p\st z}"] \& E \arrow[dd,"{p}"]\\
		K(C,X)\arrow[rd,"{\Lambda_{(C,X)}}"'{inner sep=0.1ex}]\arrow[rrd,bend left=25,"{K'(C,X)}"{inner sep=0ex}] \\[-3.5ex]
		\& F\arrow[uu,leftarrow,"{\G{p}(z)}"{pos=0.82,inner sep=0.2ex},crossing over] \arrow[r,"{z}"']\& B
	\end{cd}
	Then $\dom\p{{p}\st\p{\Lambda'_{f,\nu}}}$ is the unique lifting along the discrete opfibration $\G{p,F}(z)$ of the $2$-cell $\Lambda_{f,\nu}\ast {\G{p,K(C,X)}\left(K'(C,X)\right)}$ to $\dom\p{p\st\p{\Lambda'_{(C,X)}}}$.
	
	Notice that in particular this construction can be applied to $z=\id{F}$, exhibiting the domain of any discrete opfibration over $F$ as a cartesian-marked oplax conical colimit (starting from $K$ and $\Lambda$). Remember that we chose pullbacks in $\L$ such that the change of base of an identity is always an identity. Notice also that, given $z\:F\to B$,
	$$\lambda^z=(z\c -)\c \lambda^{\id{}}$$
	Whence we can express the $\F$-functor $K'_z$ constructed from $z$ in terms of the one constructed from $\id{}$:
	$$K'_z=(z\c -)\c K'_{\id{}}$$
	Then the components and the structure 2-cells of $\Lambda'_z$ and $\Lambda'_{\id{}}$ have the same domains, which determine them.
\end{cons}

	We now present the second of our three theorems of reduction to dense generators. This reduces the study of the fullness of the functors $\G{\tau,F}$, provided that we already know their faithfulness. Indeed such a theorem builds over \thex\ref{faith}.

\begin{teor}\label{fullness}
	Let $I\:\Y\to \L$ be a fully faithful dense generator of $\L$. If for every $Y\in \Y$
	$$\G{\tau,I(Y)}\:\HomC{\L}{I(Y)}{\O}\to \Fib[b][n][I(Y)]$$
	is fully faithful, then for every $F\in \L$ also
	$$\G{\tau,F}\:\HomC{\L}{F}{\O}\to \Fib[b][n][F]$$
	is full, and hence fully faithful by \thex\ref{faith}, so that $\tau$ is a 2-classifier in $\L$.
\end{teor}
\begin{proof}
	Let $F\in \L$; we prove that $\G{\tau,F}$ is full. Consider then two morphisms $z,z'\: F\to \O$ in \L and a morphism $h\: \G{\tau,F}(z)\to\G{\tau,F}(z')$ in $\Fib[b][n][F]$. We search for a $2$-cell $\alpha\:z\aR{}z'\:F\to \O$ such that $\G{\tau,F}(\alpha)=h$. The idea is to write $F$ as a colimit and define $\alpha$ by giving its ``components", which we can produce by the fullness of the functors $\G{\tau,I(Y)}$ with the $Y$'s in \Y that generate $F$.
	
	By \thex\ref{densityintermsofcolim}, \thex\ref{redoplaxnormalconical} and \prox\ref{propstillpreserved}, there exists a $2$-diagram $K\: \Groth{W}\to\L$ which factors through $\Y$ and a universal cartesian-marked oplax cocone 
	$$\Lambda\: \Delta 1 \aoplaxn{} \HomC{\L}{K(-)}{F}$$
	that exhibits $F$ as the $I$-absolute cartesian-marked oplax conical colimit of $K$.
	$$F=\oplaxncolim{K}$$
	
	Then, in order to produce $\alpha$, it suffices to give a modification
	\tcstwodim[7.5]{\Delta 1}{\HomC{\L}{K(-)}{\O}}{(z\c-)\c \Lambda}{(z'\c-)\c \Lambda}{\h\theta}{pos=0.51}{pos=0.51}{,outer sep=0.2ex}
	Given $(C,X)\in \Groth{W}$, we will have that
	\begin{eqD*}
		\left(\tcwl*{F}{\O}{z}{z'}{\alpha}{K(C,X)}{\Lambda_{(C,X)}}\right) =\h \theta_{(C,X)} 
	\end{eqD*}
	So, we will need to have that
	$$\G{\tau,K(C,X)}(\theta_{(C,X)})=\G{\tau,K(C,X)}(\alpha\h \Lambda_{(C,X)})$$
	and, by \prox\ref{corollpseudonatG} and the request that $\G{\tau,F}(\alpha)=h$, the right hand side of such equation is equal to the composite
	$$\G{\tau,K(C,X)}(z\c\Lambda_{(C,X)})\iso {\Lambda_{(C,X)}}\st \left(\G{\tau,F}(z)\right) \ar{ {\Lambda_{(C,X)}}\st \left(h\right) }  {\Lambda_{(C,X)}}\st \left(\G{\tau,F}(z')\right)\iso \G{\tau,K(C,X)}(z'\c\Lambda_{(C,X)})$$
	Call $h^{C,X}$ such composite morphism in $\Fib[b][n][K(C,X)]$. Since $K(C,X)\in \Y$, the functor $\G{\tau,K(C,X)}$ is fully faithful by assumption, and hence there exists a unique $2$-cell
	\tc[9][30][pos=0.45][pos=0.45][,shift right=0.8ex]{K(C,X)}{\O}{z\c\Lambda_{(C,X)}}{z'\c\Lambda_{(C,X)}}{\gamma^{C,X}}
	in \L such that $\G{\tau,K(C,X)}(\gamma^{C,X})=h^{C,X}$. We define the component $\theta_{(C,X)}$ of $\theta$ on $(C,X)$ to be such $2$-cell $\gamma^{C,X}$. The faithfulness of the functors $\G{\tau,K(C,X)}$ guarantees that $\theta$ becomes a modification. Indeed, given a morphism $(f,\nu)\:(D,X')\to(C,X)$ in $\Groth{W}$, we need to prove that
	\modifopls[3.35]{1}{\HomC{\L}{K(C,X)}{\O}}{1}{\HomC{\L}{K(D,X')}{\O}}
	{z'\c\Lambda_{(C,X)}}{z\c\Lambda_{(C,X)}}{z'\c\Lambda_{(D,X')}}{z\c\Lambda_{(D,X')}}
	{}{-\c K(f,\h\nu)}
	{\theta_{(C,X)}}{\theta_{(D,X')}}{z'\Lambda_{(f,\nu)}}{z\Lambda_{(f,\h\nu)}}{pos=0.53}{pos=0.53}{pos=0.614,inner sep=0.25ex}{pos=0.53}{outer sep = -0.25ex}{outer sep =-0.25ex}{description,pos=0.44}{description,pos=0.44}{,shorten >= 1.45em, shorten <= 2em}{,shorten >= 1.8em, shorten <= 1.45em}{6.9}{9.2}
	But since $\G{\tau,K(D,X')}$ is faithful, it suffices to prove that
	$$\G{\tau,K(D,X')}\p{z'\Lambda_{(f,\h\nu)}\c \theta^{D,X'}}=\G{\tau,K(D,X')}\p{\theta_{(C,X)}K(f,\h\nu)\c z\h\Lambda_{(f,\h\nu)}}$$
	and hence that
	$$\G{\tau,K(D,X')}\p{z'\Lambda_{(f,\h\nu)}}\c\G{\tau,K(D,X')}\p{ \theta^{D,X'}}=\G{\tau,K(D,X')}\p{\theta_{(C,X)}K(f,\h\nu)}\c\G{\tau,K(D,X')}\p{ z\h\Lambda_{(f,\h\nu)}}$$
	At this point, it is straightforward to see that such equality is given by the naturality square of $\Fib[b][n][\Lambda_{(f,\h\nu)}]$ (see \prox\ref{DOpFib-}) obtained considering the morphism $h\: \G{\tau,F}(z)\to\G{\tau,\O}(z')$ in $\Fib[b][n][F]$. For this, one needs the pseudo-naturality of $\G{\tau,-}$ (\prox\ref{pseudonatG}), the equation $\G{\tau,K(C,X)}(\theta_{(C,X)})=k^{C,X}$ and the analogous one for $(D,X')$, together with both the squares of \prox\ref{corollpseudonatG}. The square on the right of \prox\ref{corollpseudonatG} allows us to calculate\v $\G{\tau,K(D,X')}\p{\theta_{(C,X)}K(f,\nu)}$, whereas the square on the left allows us to compare the components of $\Fib[b][n][\Lambda_{(f,\h\nu)}]$ on $\G{\tau,F(z)}$ and $\G{\tau,F(z')}$ with $\G{\tau,K(D,X')}\p{ z\h\Lambda_{(f,\h\nu)}}$ and the analogous one with $z'$ instead of $z$.\v
	
	Thus we conclude that $\theta$ is a modification, and by the universal property of the colimit $F$ we obtain an induced $2$-cell \tc*{F}{\O}{z}{z'}{\alpha} in \L. It remains to prove that $\G{\tau,F}(\alpha)=h$.
	
	The idea is to conclude such equality by using the uniqueness part of the universal property of a colimit. As $\G{\tau,F}(\alpha)$ and $h$ are morphisms $\G{\tau,F}(z)\to\G{\tau,F}(z')$ in $\Fib[b][n][F]$, both are totally determined by a morphism in $\L$ from $\dG{z}$ to $\dG{z'}$; we call respectively $\G{\tau,F}(\alpha)$ and $h$ such morphisms in \L. By \conx\ref{lemma}, applied to $z=\id{F}$,
	$$\dG{z}=\oplaxncolim{(\dom\c {\G{\tau,F}(z)}\st \c K')},$$
	presented by the universal cartesian-marked oplax cocone $\dom\c {\G{\tau,F}(z)}\st\c \Lambda'$, with $\Lambda'$ and $K'$ constructed from $\Lambda$ and $K$ as in \conx\ref{lemma}.
	
	Then, in order to conclude that $\G{\tau,F}(\alpha)$ and $h$ are equal, it suffices to show that
	\begin{equation}\label{eqgalphah}
		\left(\G{\tau,F}(\alpha)\c-\right)\c\dom\c {\G{\tau,F}(z)}\st\c \Lambda'=\left(h\c-\right)\c\dom\c {\G{\tau,F}(z)}\st\c \Lambda'
	\end{equation}
	as cartesian-marked oplax natural transformations. And the last part of \conx\ref{lemma} gives us an explicit calculation of $\dom\c {\G{\tau,F}(z)}\st\c \Lambda'$ in terms of $\Lambda$. Precisely, given\v an arbitrary morphism $(D,X')\al{(f,\h\nu)}(C,X)$ in $\Groth{W}$, 
	$$\dom\left({\G{\tau,F}(z)}\st\left(\Lambda'_{(C,X)}\right)\right)={\G{\tau,F}(z)}\st \left(\Lambda_{(C,X)}\right)$$
	and $\dom\p{{\G{\tau,F}(z)}\st\p{\Lambda'_{f,\nu}}}$ is the unique lifting of the $2$-cell $\Lambda_{f,\nu}\ast \G{\G{\tau,F}(z),K(C,X)}\left(\Lambda_{(C,X)}\right)$ to ${\G{\tau,F}(z)}\st \left(\Lambda_{(C,X)}\right)$ along $\G{\tau,F}(z)$.
	We now notice that the following two squares are commutative:
	\begin{eqD*}
		\sq*[n][4.75][11.5]{\dG[\G{\tau,F}(z),K(C,X)]{\Lambda_{(C,X)}}}{\dG{z}}{\dG[\G{\tau,F}(z'),K(C,X)]{\Lambda_{(C,X)}}}{\dG{z'}}{{\G{\tau,F}(z)}\st \left(\Lambda_{(C,X)}\right)}{{\Lambda_{(C,X)}}\st \left(\G{\tau,F}(\alpha)\right)}{\G{\tau,F}(\alpha)}{{\G{\tau,F}(z')}\st \left(\Lambda_{(C,X)}\right)}
		\qquad
		\sq*[n][5][11.5]{\dG[\G{\tau,F}(z),K(C,X)]{\Lambda_{(C,X)}}}{\dG{z}}{\dG[\G{\tau,F}(z'),K(C,X)]{\Lambda_{(C,X)}}}{\dG{z'}}{{\G{\tau,F}(z)}\st \left(\Lambda_{(C,X)}\right)}{{\Lambda_{(C,X)}}\st (h)}{h}{{\G{\tau,F}(z')}\st \left(\Lambda_{(C,X)}\right)}
	\end{eqD*}
	Then, to prove that equation~\refs{eqgalphah} holds on component $(C,X)$, it suffices to show that
	$${\Lambda_{(C,X)}}\st\left(\G{\tau,F}(\alpha)\right)={\Lambda_{(C,X)}}\st\left(h\right)$$
	as morphisms in \Fib[b][n][K(C,X)].
	Using \prox\ref{corollpseudonatG}, we see that
	${\Lambda_{(C,X)}}\st\left(\G{\tau,F}(\alpha)\right)$ is equal to
	$${\Lambda_{(C,X)}}\st\left({\G{\tau,F}(z)}\right)\iso \G{\tau,F}(z\c\Lambda_{(C,X)})\ar{\G{\tau,F}\left(\alpha\h\Lambda_{(C,X)}\right)}\G{\tau,F}(z'\c\Lambda_{(C,X)})\iso{\Lambda_{(C,X)}}\st \left({\G{\tau,F}(z')}\right)$$
	and thus is equal to ${\Lambda_{(C,X)}}\st\left(h\right)$ since $\alpha\h\Lambda_{(C,X)}=\theta_{(C,X)}$, by construction of $\theta_{(C,X)}$.\v
	
	\noindent It only remains to prove that equation~\refs{eqgalphah} holds on morphism component $(f,\nu)$. But this is straightforward using the uniqueness of the liftings through a discrete opfibration, that equation~\refs{eqgalphah} holds on object components and the fact that both $\G{\tau,F}(\alpha)$ and $h$ are morphisms of discrete opfibrations.
\end{proof}

\begin{rem}\label{remfullness}
	As anticipated in \remx\ref{remisoclasses}, the possibility of defining the functors $\G{\tau,F}$ to have as codomain a category of isomorphism classes of discrete opfibrations does not work well with the reduction of the study of $2$-classifiers to generators. \thex\ref{fullness} is however the only delicate point one encounters. We can define a category $\DOpFiboi{F}$ which has as objects isomorphism classes of discrete opfibrations and as morphisms collections of morphisms compatible with every possible change of representative for the domain or the codomain. Notice that we then have a full functor
	$$\Fib[b][n][F]\afull{q} \DOpFiboi{F}$$
	The problem with \thex\ref{fullness} is that, from the assumption that for every $Y\in \Y$ the composite
	$$\HomC{\L}{Y}{\O}\ar{\G{\tau,Y}}\Fib[b][n][Y]\afull{q} \DOpFiboi{Y}$$
	is fully faithful, we cannot deduce the analogue of this for every $F\in \L$. Indeed, we can no longer have
	$$\G{\tau,K(C,X)}(\theta_{(C,X)})=h^{C,X}$$
	in $\Fib[b][n][F]$, but only in $\DOpFiboi{F}$. Of course we then only need $\G{\tau,F}(\alpha)=h$ in $\DOpFiboi{F}$, but it seems that there is no way to find the isomorphisms that regulate $\G{\tau,F}(\alpha)=h$ starting from the ones that regulate $\G{\tau,K(C,X)}(\theta_{(C,X)})=h^{C,X}$ for every $(C,X)$. One strategy could be to induce them using the universal property of the colimit, but we cannot guarantee that the isomorphisms that regulate all the $\G{\tau,K(C,X)}(\theta_{(C,X)})=h^{C,X}$ form a cocone.
\end{rem}

We aim at the third of our three theorems of reduction to dense generators. \thex\ref{faith} and \thex\ref{fullness} allow to check the conditions for $\tau\:\O\b\to \O$ to be a $2$-classifier just on a dense generator. We now want to similarly reduce to dense generators the study of what $\tau$ classifies. The following construction will be very important for this. 

Recall that we denote as $\Fib?[b][n][F]$ the full subcategory of $\Fib[b][n][F]$ on the discrete opfibrations that satisfy a fixed pullback-stable property $\opn{P}$ (in our examples, $\opn{P}$ will be the property of having small fibres).

\begin{cons}\label{consesssurj}
	Let $I\:\Y\to \L$ be a fully faithful dense generator of $\L$. Assume that $\tau$ satisfies a pullback-stable property $\opn{P}$ and that for every $Y\in \Y$
	$$\G{\tau,I(Y)}\:\HomC{\L}{I(Y)}{\O}\to \Fib?[b][n][I(Y)]$$
	is an equivalence of categories. Let then $\phi\:G\to F$ be a discrete opfibration in $\L$ that satisfies the property $\opn{P}$. We would like to construct a characteristic morphism $z\:F\to \O$ for $\phi$. That is, a $z$ such that $\G{\tau,F}(z)$ is isomorphic to $\phi$, so that $\phi$ gets classified by $\tau$.
	
	There exist a $2$-diagram $K\: \Groth{W}\to\L$ which factors through $\Y$ and a universal cartesian-marked oplax cocone $\Lambda$ that exhibits $F$ as the $I$-absolute
	$$F=\oplaxncolim{K}$$
	We would like to induce $z$ via the universal property of the colimit $F$. The idea is condensed in the following diagram:
	\begin{cd}[5.2][5.2]
		H^{(C,X)} \PB{rd} \arrow[r,"{}"]\arrow[d,"{\G{\phi}\left(\Lambda_{(C,X)}\right)}"'] \& G \arrow[d,"{\phi}"]\&[1ex] \O\b \arrow[d,"{\tau}"]\\
		K(C,X)\arrow[rr,bend right=25,"{\hphantom{CC}\G{\tau}^{-1}\left(\G{\phi}\left(\Lambda_{(C,X)}\right)\right)}"'{pos=0.428,inner sep=0.15ex},shorten <=-0.4ex] \arrow[r,"{\Lambda_{(C,X)}}"] \& F\arrow[r,dashed,"{z}"] \& \O
	\end{cd}
	
	For every $(C,X)$ in $\Groth{W}$, the change of base $\G{\phi,K(C,X)}(\Lambda_{C,X})$ of $\phi$ along $\Lambda_{C,X}$ satisfies the property $P$ and is thus in the essential image of $\G{\tau,K(C,X)}$. We can then consider the oplax natural transformation $\chi$ given by the composite
	$$\Delta 1 \aoplaxn{\Lambda} \HomC{\L}{K(-)}{F}\apseudo{\G{\phi,K(-)}}\Fib?[b][n][K(-)]\apseudo{{\G{\tau,K(-)}}^{-1}} \HomC{\L}{K(-)}{\O},$$
	where every ${\G{\tau,K(C,X)}}^{-1}$ is a right adjoint quasi-inverse of ${\G{\tau,K(C,X)}}$ giving an adjoint equivalence. The action of such ${\G{\tau,K(C,X)}}^{-1}$ on morphisms $h\:\psi\to \psi'$ is exhibited by the naturality squares of the counit $\eps$
	\begin{cd}[5.5][6]
		{{\G{\tau,K(C,X)}}\left({\G{\tau,K(C,X)}}^{-1}(\psi)\right)}\arrow[d,"{{\G{\tau,K(C,X)}}\left({\G{\tau,K(C,X)}}^{-1}(h)\right)}"']\arrow[r,iso,"{\eps_{\psi}}"'{inner sep=1ex}] \& {\psi} \arrow[d,"{h}"] \\
		{{\G{\tau,K(C,X)}}\left({\G{\tau,K(C,X)}}^{-1}(\psi')\right)} \arrow[r,iso,"{\eps_{\psi'}}"'{inner sep=1ex}]\&{\psi'}
	\end{cd}
	We are also using that for every morphism $(f,\nu)\:(D,X')\al{}(C,X)$ in $\Groth{W}$ the functor $\Fib[b][n][K(f,\nu)]=K(f,\nu)\st$ restricts to a functor between the essential image of ${\G{\tau,K(D,X')}}$ and the essential image of ${\G{\tau,K(C,X)}}$. Moreover, using \prox\ref{pseudonatG}, we have that $\G{\tau,K(-)}^{-1}$ extends to a pseudo-natural transformation. Its structure $2$-cell on any $(f,\nu)\:(D,X')\al{}(C,X)$ in $\Groth{W}$ is given by the pasting
	\begin{cd}[4][3.5]
		\Fib?[b][n][K^{D,X'}] \arrow[r,iso,shift right=2.65ex,"{\eps^{-1}}"'{inner sep=0.7ex},start anchor={[xshift=5.5ex]}]\arrow[rd,equal,"{}"] \arrow[r,"{{\G{\tau}}^{-1}}"] \& \HomC{\L}{K^{D,X'}}{\O} \arrow[d,"{{\G{\tau}}}"]\arrow[r,"{-\c K(f,\nu)}"] \& \HomC{\L}{K^{C,X}}{\O} \arrow[d,"{{\G{\tau}}}"] \arrow[rd,equal,"{}"]\\
		\& \Fib?[b][n][K^{D,X'}] \arrow[ru,Rightarrow,"{{\left(\G{\tau,K(f,\nu)}\right)}^{-1}}"'{inner sep=0.18ex},shorten <=3ex,shorten >=3ex] \arrow[r,"{K(f,\nu)\st}"'] \& \Fib?[b][n][K^{C,X}] \arrow[r,iso,shift left=2ex,"{\eta^{-1}}"{inner sep=0.7ex},start anchor={[xshift=-5ex]}] \arrow[r,"{{\G{\tau}}^{-1}}"'] \& \HomC{\L}{K^{C,X}}{\O} 
	\end{cd}
	where $\eta$ is the unit of the adjoint equivalence, and we denoted $K(C,X)$ as $K^{C,X}$. 
	
	The composite $\chi$ above is readily seen to be a sigma natural transformation (of Descotte, Dubuc and Szyld's~\cite{descottedubucszyld_sigmalimandflatpseudofun}, w.r.t.\ the cartesian marking). That is, $\chi$ is an oplax natural transformation with the structure 2-cells $\chi_{f,\id{}}$ being isomorphisms for every morphism of type $(f,\id{})$ in ${\left(\Groth{W}\right)}\op$. If we were in a bicategorical context, this would then be enough to induce a morphism $z\:F\to \O$, as we will explore in detail in future work. In our strict $2$-categorical context, we further need to be able to ``normalize" such sigma natural transformation, ensuring that the structure 2-cells on the morphisms $(f,\id{})$ are identities. So that the morphisms $\G{\tau}^{-1}\left(\G{\phi}\left(\Lambda_{(C,X)}\right)\right)$ yield a cartesian-marked oplax cocone. Essentially, this means that we can choose good quasi-inverses of the $\G{\tau,K(C,X)}$'s. Such an extra hypothesis is satisfied by $2$-dimensional presheaves (i.e.\ prestacks), see \thex\ref{teor2classinprestacks}. By \thex\ref{teorfactorization}, it is then satisfied by nice sub-2-categories of prestacks, such as the 2-category of stacks (\thex\ref{teor2classinstacks}).
\end{cons}

	We now present the third of our three theorems of reduction to dense generators. Building over \thex\ref{faith} and \thex\ref{fullness}, we reduce to dense generators the study of what a 2-classifier $\tau\:\O\b\to \O$ classifies. The proof is constructive, based on \conx\ref{consesssurj}, so that we also give a concrete recipe for the characteristic morphisms.

\begin{teor}\label{esssurj}
	Let $I\:\Y\to \L$ be a fully faithful dense generator of $\L$. Assume that $\tau$ satisfies a pullback-stable property $\opn{P}$ and that for every $Y\in \Y$
	$$\G{\tau,I(Y)}\:\HomC{\L}{I(Y)}{\O}\to \Fib?[b][n][I(Y)]$$
	is an equivalence of categories. Assume further that, for every discrete opfibration $\phi\:G\to F$ in $\L$ that satisfies the property $\opn{P}$, the sigma natural transformation $\chi$ produced in \conx\ref{consesssurj} \pteor{starting from $\phi$ and some $K$ and $\Lambda$} is isomorphic to a cartesian-marked oplax natural transformation $\aleph$ \pteor{i.e.\ can be normalized}.
	
	Then for every $F\in \L$
	$$\G{\tau,F}\:\HomC{\L}{F}{\O}\to \Fib?[b][n][F]$$
	is essentially surjective, and hence an equivalence of categories by \thex\ref{fullness}.
\end{teor}
\begin{proof}
	Let $\phi\:G\to F$ be a discrete opfibration in $\L$ that satisfies the property $\opn{P}$. Consider the associated $\chi$ and $\aleph$ as in the statement. Let $z\:F\to \O$ be the unique morphism induced by $\aleph$ via the universal property of the colimit $F=\oplaxncolim{K}$. We prove that $z$ is a characteristic morphism for $\phi$ with respect to $\tau$.
	
	Consider the pullback
	\sq[p][6.4][6.4]{V}{\O\b}{F}{\O}{\t{z}}{\G{\tau,F}(z)}{\tau}{z}
	We want to prove that there is an isomorphism $j\:G\iso V$ such that $\G{\tau,F}(z)\c j=\phi$. Applying \conx\ref{lemma} to $\phi$ and $\id{F}$, we construct $K'=K'_{\id{}}$ and $\Lambda'=\Lambda'_{\id{}}$ that exhibits
	$$G=\oplaxncolim{\left(\dom\c \h{\phi}\st \c K'\right)}.$$
	Applying again \conx\ref{lemma} to $\tau$ and $z\:F\to \O$ we obtain
	$$V=\oplaxncolim{\left(\dom\c \h\tau\st \c (z\c -)\c K'\right)}.$$
	
	We show that 
	$$\dom\c \h\tau\st \c (z\c -)\c K'\iso \dom\c\h {\phi}\st \c K'.$$
	Notice that $(z\c -)\c K'$ is the $2$-functor $\Groth{W}\to \laxslice{\L}{\O}$ associated to the oplax natural transformation $\aleph$, as described in \conx\ref{lemma}. Since $\aleph\iso \chi$, we obtain that $(z\c -)\c K'\iso U$ where $U$ is the $2$-functor $\Groth{W}\to \laxslice{\L}{\O}$ associated to the sigma natural transformation $\chi$. Moreover the general component $t_{C,X}$ on $(C,X)\in \Groth{W}$ of such isomorphism has first component equal to the identity. This produces an isomorphism
	$$\dom\c \h\tau\st \c (z\c -)\c K'\iso \dom\c \h\tau\st \c U$$
	whose general component on $(C,X)$ is over $K(C,X)$. Indeed by \remx\ref{remcomparedompstdomgp}
	$$\dom\left(\tau\st \left(t_{(C,X)}\right)\right)=\dom \left(\G{\tau,K(C,X)}\left(\pr{2}\left(t_{C,X}\right)\right)\right)$$
	and is thus an isomorphism that makes the following triangle commute:
	\begin{cd}[3.4][-3.2]
		{\dom \left(\G{\tau,K(C,X)}\left(\aleph_{(C,X)}\right)\right)} \arrow[rr,aiso,"{}"']\arrow[rd,"{\G{\tau,K(C,X)}\left(\aleph_{(C,X)}\right)}"',inner sep=-0.3ex]\&\& {\dom \left(\G{\tau,K(C,X)}\left(\chi_{(C,X)}\right)\right)} \arrow[ld,"{\G{\tau,K(C,X)}\left(\chi_{(C,X)}\right)}",inner sep=-0.3ex]\\
		\& {K(C,X)}
	\end{cd}
	In this way, we have handled the isomorphisms given by the normalization process. We now show that
	$$\dom\c \h\tau\st\c U\iso \dom\c\h {\phi}\st \c K'.$$
	Given $(C,X)\in \Groth{W}$, by \remx\ref{remcomparedompstdomgp} and by construction of $\chi$ we have $$\dom\left(\tau\st\left(\chi_{(C,X)}\right)\right)=\dom \left(\G{\tau,K(C,X)}\left(\chi_{(C,X)}\right)\right)\iso\dom \left(\G{\phi,K(C,X)}\left(\Lambda_{(C,X)}\right)\right)=\dom \left(\phi\st \left(\Lambda_{(C,X)}\right)\right).$$
	So the counits of the adjoint equivalences $\G{\tau,K(C,X)}\dashv {\G{\tau,K(C,X)}}^{-1}$ give isomorphisms
	$$\xi_{(C,X)}\:Q^{(C,X)}=\left(\dom\c \h\tau\st\c U\right)(C,X)\aisoo\left(\dom\c\h {\phi}\st \c K'\right)(C,X)=H^{(C,X)}$$
	over $K(C,X)$, where the map to $K(C,X)$ from the right hand side is $\G{\phi,K(C,X)}\left(\Lambda_{(C,X)}\right)$. We prove that these isomorphisms form a $2$-natural transformation $\xi$.
	Take then $(f,\nu)\:(C,X)\to (D,X')$ in $\Groth{W}$. By \remx\ref{remcomparedompstdomgp}, we can express the action of $\dom\c \tau\st$ on the morphism
	$$U(f,\nu)=\h[3]\tr*[2.8][4][0][0][l][-0.235][\chi_{f,\nu}][1]{K(C,X)}{K(D,X')}{\O}{K(f,\nu)}{\chi_{(C,X)}}{\chi_{(D,X')}}$$
	in terms of $\dom \left(\G{\tau,K(C,X)}(\chi_{f,\nu})\right)$. Analogously, we express $\dom\left(\phi\st \left(K(f,\nu),\Lambda_{f,\nu}\right)\right)$ in terms of $\dom \left(\G{\phi,K(C,X)}\left(\Lambda_{f,\nu}\right)\right)$. Consider the composite pullbacks
	\begin{eqD*}
		\pbsqdrpN{S^H}{H^{(D,X')}}{G}{K(C,X)}{K(D,X')}{F}{}{}{}{\G{\phi}\left(\Lambda_{(D,X')}\right)}{\phi}{K(f,\nu)}{\Lambda_{(D,X')}}\quad\h[4] \pbsqdrpN{S^Q}{Q^{(D,X')}}{\O\b}{K(C,X)}{K(D,X')}{\O}{}{}{}{\G{\tau}\left(\chi_{(D,X')}\right)}{\tau}{K(f,\nu)}{\chi_{(D,X')}}
	\end{eqD*}
	and call $R^H$ and $R^Q$\v respectively the pullbacks of $\phi$ and of $\tau$ along the composites. By construction of $\chi$\v and by the action of ${\G{\tau,K(C,X)}}^{-1}$ on morphisms, we calculate that $\G{\tau,K(C,X)}(\chi_{f,\nu})$ is precisely the composite
	$$Q^{(C,X)}\aisoo[\xi_{(C,X)}] H^{(C,X)}\ar{\G{\phi,K(C,X)}\left({\Lambda_{f,\nu}}\right)} R^H\iso S^H\aisoo[K(f,\nu)\st \h\xi_{(D,X')}^{-1}] S^Q\iso R^Q$$
	Notice that we also need one triangular equality of the adjoint equivalence $\G{\tau,K(C,X)}$ to handle the $\eta^{-1}$ part of $\G{\tau,K(f,\nu)}^{-1}$. In order to obtain $\dom \left(\tau\st\left(K(f,\nu),\chi_{f,\nu}\right)\right)$, by \remx\ref{remcomparedompstdomgp}, we compose $\G{\tau,K(C,X)}(\chi_{f,\nu})$ with
	$$R^Q\iso S^Q\ar{} Q^{(D,X')}$$
	We conclude the naturality of $\xi$ by definition of $K(f,\nu)\st \h\xi_{(D,X')}^{-1}$. To prove that $\xi$ is $2$-natural, consider a $2$-cell $\delta\:(f,\nu)\aR{}(f',\nu')\:(C,X)\to (D,X')$ in $\Groth{W}$. Both
	$$\xi_{(D,X')}\ast \left(\dom\c \h\tau\st\c U\right)(\delta)$$
	$$\left(\dom\c\h {\phi}\st \c K'\right)(\delta)\ast \xi_{C,X}$$
	give the unique lifting of $K(\delta)\ast \G{\tau,K(C,X)}\left(\chi_{(C,X)}\right)$ to $\xi_{(D,X')}\c \left(\dom\c \h\tau\st\c U\right)(f,\nu)$ along $\G{\phi,K(D,X')}\left(\Lambda_{(D,X')}\right)$. Thus $\xi$ is $2$-natural. We then obtain a $2$-natural isomorphism
	$$\zeta\:\dom\c \h\tau\st \c (z\c -)\c K'\iso \dom\c\h {\phi}\st \c K'$$
	whose general component on $(C,X)\in \Groth{W}$ is over $K(C,X)$. 
	
	As a consequence, there is an isomorphism $j\:G\iso V$, respecting the universal cartesian-marked oplax cocones $\Theta^G$ and $\Theta^V$ that exhibit the two as colimits:
	\begin{cd}[5.4][5.4]
		\Delta 1 \arrow[r,Rightarrow,"{\Theta^G}","\oplaxn"']\arrow[d,Rightarrow,"{\Theta^V}"',"\oplaxn"]\& {\HomC{\L}{\left(\dom\c \h{\phi}\st \c K'\right)(-)}{G}} \arrow[d,Rightarrow,"{j\c -}"]\\
		{\HomC{\L}{\left(\dom\c \h\tau\st \c (z\c -)\c K'\right)(-)}{V}} \arrow[r,Rightarrow,"{-\c\h \zeta^{-1}_{(-)}\h[2]}"'{inner sep=0.6ex}]\& {\HomC{\L}{\left(\dom\c \h{\phi}\st \c K'\right)(-)}{V}}
	\end{cd}
	We want to show that the following triangle is commutative:
	\begin{eqD}{trianglephi}
		\begin{cd}*[4.5][3.5]
			G \arrow[rr,aiso,"{j}"']\arrow[rd,"{\phi}"']\&\& V \arrow[ld,"{\G{\tau,F}(z)}"{inner sep=0.25ex}]\\
			\& F
		\end{cd}
	\end{eqD}
	Since $G$ is a cartesian-marked oplax conical colimit, it suffices to show that
	$$(\phi\c -)\c \Theta^G=(\G{\tau,F}(z)\c -)\c (j\c -)\c \Theta^G$$
	Whence it suffices to show that
	$$(\phi\c -)\c \Theta^G=\left(\G{\tau,F}(z)\c -\c \h \zeta^{-1}_{(-)}\right)\c \Theta^V$$
	Given $(C,X)\in \Groth{W}$, the two have equal components on $(C,X)$ since by \conx\ref{lemma}
	$$\G{\tau,F}(z)\c \Theta^V_{(C,X)}\c \h \zeta^{-1}_{(C,X)}=\Lambda_{(C,X)}\c \G{\tau,K(C,X)}\left(\aleph_{(C,X)}\right)\c  \h \zeta^{-1}_{(C,X)}=\Lambda_{(C,X)}\c\G{\phi,K(C,X)}\left(\Lambda_{(C,X)}\right),$$
	using that $\zeta_{(C,X)}$ is over $K(C,X)$. Given $(f,\nu)\:(D,X')\al{}(C,X)$ in $\Groth{W}$, also the structure $2$-cells of the two cartesian-marked oplax natural transformations on $(f,\nu)$ are equal, since by \conx\ref{lemma}
	$$\G{\tau,F}(z)\ast \Theta^V_{(f,\nu)}\ast \h \zeta^{-1}_{(C,X)}=\Lambda_{f,\nu}\ast \left(\G{\tau,K(C,X)}\left(\aleph_{(C,X)}\right)\c \h \zeta^{-1}_{(C,X)}\right)=\Lambda_{f,\nu}\ast\G{\phi,K(C,X)}\left(\Lambda_{(C,X)}\right).$$
	Therefore the triangle of equation~\refs{trianglephi} is commutative and $\phi$ is in the essential image of $\G{\tau,F}$.
\end{proof}

\begin{rem}
	In the proof of \thex\ref{esssurj}, we have actually proved the following sharper result, that involves the classification of single discrete opfibrations $\phi$.
	
	We also show that the operation of normalization described in \thex\ref{esssurj} is necessary.
\end{rem}

\begin{coroll}[of the proof of \thex\ref{esssurj}]\label{corollsharperesssurj}
	Let $I\:\Y\to \L$ be a fully faithful dense generator of $\L$. Assume that $\tau\:\O\b\to \O$ is a $2$-classifier in $\L$, and let $\phi\:G\to F$ be an arbitrary discrete opfibration in $\L$. Consider $K$ and $\Lambda$ as in \conx\ref{consesssurj}. The following properties are equivalent:
	\begin{enumT}
		\item $\phi$ is classified by $\tau$, i.e.\ $\phi$ is in the essential image of $\G{\tau,F}$;
		\item for every $(C,X)\in \Groth{W}$ the change of base $\G{\phi,K(C,X)}(\Lambda_{C,X})$ of $\phi$ along $\Lambda_{C,X}$ is in the essential image of $\G{\tau,K(C,X)}$, and the operation of normalization described in \thex\ref{esssurj} starting from $\phi$ is possible.
	\end{enumT}
\end{coroll}
\begin{proof}
	The proof of $(ii)\aR{}(i)$ is exactly as the proof of \thex\ref{esssurj}, using the essential image of $\G{\tau,K(C,X)}$ in place of $\Fib?[b][n][K(C,X)]$.
	
	We prove $(i)\aR{}(ii)$. By assumption, there exists a characteristic morphism $z$ for $\phi$. For every $(C,X)$, we then have that $z\c\Lambda_{(C,X)}$ is a characteristic morphism for $\G{\phi,K(C,X)}(\Lambda_{C,X})$. It remains to prove that the operation of normalization described in \thex\ref{esssurj} starting from $\phi$ is possible. We can choose the quasi-inverse of $\G{\tau,K(C,X)}$ (restricted to its essential image) so that for every $b\:K(C,X)\to F$ in $\L$
	$$\G{\tau,K(C,X)}^{-1}(\G{\phi,K(C,X)}(b))=z\c b$$
	and the component of the counit on $\G{\phi,K(C,X)}(b)$ is given by the pseudofunctoriality of the pullback. Then for every morphism $(f,\id{})\:(D,F(f)(X))\al{} (C,X)$ in $\Groth{W}$
	$$\chi_{(C,X)}=z\c \Lambda_{(C,X)}=z\c \Lambda_{(D,F(f)(X))}\c K(f,\id{})=\chi_{(D,F(f)(X))}\c K(f,\id{}).$$
	In order to prove that $\chi_{f,\id{}}=\id{}$, it suffices to prove that $\G{\tau,K(C,X)}(\chi_{f,\id{}})=\id{}$. It is straightforward to see that this holds, using the recipe described in the proof of \thex\ref{esssurj}. Indeed it is just given by the compatibilities of a pullback along a composite of three morphisms with the composite pullbacks.
\end{proof}

\begin{coroll}\label{corollgoodtwoclas}
	Let $I\:\Y\to \L$ be a fully faithful dense generator of $\L$. Let $\omega\:1\to \O$ be a morphism in $\L$ such that the lax limit of the arrow $\omega$ satisfies a fixed pullback-stable property $\opn{P}$. If for every $Y\in \Y$
	$$\Gg{\omega,I(Y)}\:\HomC{\L}{I(Y)}{\O}\to \Fib?[b][n][I(Y)]$$
	is an equivalence of categories and the operation of normalization described in \thex\ref{esssurj} \pteor{starting from every $\phi$} is possible, then $\omega$ is a good 2-classifier in $\L$ with respect to $\opn{P}$.
\end{coroll}
\begin{proof}
	By \remx\ref{remreplacement}, taking $\tau$ to be the lax limit of the arrow $\omega$, we have that $\Gg{\omega,F}=\G{\tau,F}$ for every $F\in \L$. We conclude by \thex\ref{esssurj}.
\end{proof}

\begin{rem}
	The theorems of reduction of the study of 2-classifiers to dense generators offer great benefits. To have an idea of this, we can look at the following example.
	
	In Sections~\ref{sectiontwoclassintwopresheaves} and~\ref{sectiontwoclassinstacks}, we will then apply the theorems of reduction to dense generators to the cases of 2-presheaves (i.e.\ prestacks) and stacks.
\end{rem}

\begin{exampl}\label{exampleredincat}
	The theorems of reduction to dense generators allow us to deduce all the major properties of the Grothendieck construction (or category of elements) from the trivial observation that everything works well over the singleton category.
	
	Indeed the singleton category $\1$ is a dense generator in $\Cat$. So we can just look at the discrete opfibrations over $\1$. Let $\omega=1\:\1\to \Set$. The lax limit $\tau\:\Set\b\to \Set$ certainly has small fibres. In fact one immediately sees that any comma object from $\omega$, say to $F\:\B\to \Set$, gives a discrete opfibration with small fibres, since the fibre over $B\in \B$ is isomorphic to $\HomC{\Set}{1}{F(B)}$. Let $\opn{P}=\opn{s}$ be the property of having small fibres.
	
	The functor
	$$\Gg{\omega,\1}\:\HomC{\CAT}{\1}{\Set}\to \Fib+[b][n][1]\iso\Set$$ 
	sends a functor $\1\to \Set$ to the set it picks, so it is clearly an equivalence of categories. By the theorems of reduction to dense generators, we deduce that the construction of the category of elements is fully faithful and classifies all discrete opfibrations with small fibres (deducing thus the whole \exax\ref{excat}). Indeed, let $p\:\E\to \B$ be a discrete opfibration with small fibres. The density of $\1$ allows us to express $\B$ as a cartesian-marked oplax conical colimit of the constant at 1 functor $\Delta 1$. By \exax\ref{univoplaxnormalcoconepresheaves} we know that the universal cartesian-marked oplax cocone $\Lambda$ is given by
	$$\forall \fib[o]{B'}{f}{B} \text{ in } \B \quad\quad\quad {\begin{cd}*[4][2]
			\1 \arrow[rr,"B",""'{name=Q}] \arrow[d,equal] \&\& \B\\
			\1 \arrow[rru,bend right,"B'"']\& \phantom{.} \arrow[Rightarrow,from=Q,"{f}"{pos=0.355},shift right=1.5ex,shorten <=0.8ex,shorten >= 2.45ex]
	\end{cd}}$$
	Following \conx\ref{consesssurj} and the proof of \corx\ref{corollgoodtwoclas}, we consider the sigma natural transformation $\chi$ given by the composite
	$$\Delta 1 \aoplaxn{\Lambda} \HomC{\Cat}{\Delta 1(-)}{\B}\apseudo{\G{p,\Delta 1(-)}}\Fib+[b][n][\Delta 1(-)]\apseudo{{\Gg{\omega,\Delta 1(-)}}^{-1}} \HomC{\Cat}{\Delta 1(-)}{\Set}.$$
	But $\G{p,\Delta 1(-)}$ is strict 2-natural (thanks to our choice of pullbacks). Similarly, also $\Gg{\omega,\Delta 1(-)}$ and hence its quasi-inverse are strict 2-natural. So that $\chi$ is already cartesian-marked oplax natural. Explicitly, $\chi$ is given by
	$$\forall \fib[o]{B'}{f}{B} \text{ in } \B \quad\quad\quad {\begin{cd}*[4][2]
			\1 \arrow[rr,"(p)_B",""'{name=Q}] \arrow[d,equal] \&\& \Set\\
			\1 \arrow[rru,bend right,"(p)_{B'}"']\& \phantom{.} \arrow[Rightarrow,from=Q,"{f\stb}"{pos=0.355},shift right=1.5ex,shorten <=0.8ex,shorten >= 2.45ex]
	\end{cd}}$$
	where $(p)_B$ is the fibre of $p$ on $B$, since the pullback of $p$ along each $B\:\1\to\B$ gives precisely the fibre over $B$. This induces the known characteristic morphism $\B\to \Set$ for $p$, collecting together the fibres of $p$. So the concrete recipe for characteristic morphisms described in the proof of \thex\ref{esssurj} recovers the usual recipe for the quasi-inverse of the category of elements construction.
\end{exampl}

	We now present two general results (\prox\ref{propfactorization} and \thex\ref{teorfactorization}) that describe a strategy to restrict a good 2-classifier in $\L$ to a good 2-classifier in a nice sub-2-category $\M$ of $\L$. Such a strategy will involve factorization arguments and our theorems of reduction of the study of 2-classifiers to dense generators. The key idea is to restrict $\O\in \L$ to some $\O_\M\in \M$ so that the characteristic morphisms in $\L$ of discrete opfibrations in $\M$ factor through $\O_\M$. This is the strategy that we will follow in Section~\ref{sectiontwoclassinstacks} to restrict our good 2-classifier in prestacks to one in stacks.
	
	\begin{notation}
		Throughout the rest of this section, we fix $\opn{P}$ an arbitrary pullback-stable property $\opn{P}$ for discrete opfibrations in $\L$. We assume of course that $\opn{P}$ only depends on the isomorphism classes of discrete opfibrations.
	\end{notation}
	
	\begin{defne}\label{defnice}
		A fully faithful 2-functor $i\:\M\ffto \L$ will be called \dfn{nice} if it lifts pullbacks along discrete opfibrations, comma objects and the terminal object \pteor{that is, such limits exist in $\M$ and are calculated in $\L$} and preserves discrete opfibrations.
		
		A sub-2-category $i\:\M\cont \L$ (that is, an injective on objects and fully faithful 2-functor $i\:\M\ffto \L$) will be called \dfn{nice} if $i$ is nice. 
	\end{defne}
	
	\begin{exampl}
		Any reflective sub-2-category is nice. Indeed notice that any right 2-adjoint preserves discrete opfibrations thanks to the natural isomorphism between hom-categories given by the adjunction.
	\end{exampl}
	
	\begin{rem}\label{rempropPinsub2cat}
		Given a nice fully faithful 2-functor $i\:\M\ffto \L$, we will say that a discrete opfibration $\phi$ in $\M$ satisfies $\opn{P}$ if $i(\phi)$ does so.
	\end{rem}
	
	We will need the notions of fully faithful morphism and of chronic morphism in $\L$. 
	
	\begin{defne}\label{defchronic}
		A morphism $l\:F\to B$ in $\L$ is \predfn{fully faithful} if for every $X\in \L$ the functor $l\c -\:\HomC{\L}{X}{F}\to\HomC{\L}{X}{B}$ is fully faithful. $l$ is \predfn{chronic} if every $l\c -$ is injective on objects and fully faithful.
	\end{defne}
	
	We are ready to present our first result of restriction. Rather than restricting a good 2-classifier, we start from a morphism $\omega\:1\to \O$ in $\L$ such that its lax limit $\tau$ is a 2-classifier in $\L$. By \remx\ref{remreplacement}, this condition is weaker than being a good 2-classifier. Indeed it means that for every $F\in \L$
	$$\Gg{\omega,F}\:\HomC{\L}{F}{\O}\to \Fib[b][\L][F]$$
	is fully faithful. Of course, the result can then be applied to a starting good 2-classifier in $\L$ as well.

\begin{prop}\label{propfactorization}
	Let $i\:\M\ffto \L$ be a nice fully faithful 2-functor \pteor{\defx\ref{defnice}}. Let then $\omega\:\1\to \O$ in $\L$ such that its lax limit $\tau$ is a 2-classifier in $\L$. Finally, let $\O_\M\in \M$ such that there exists a fully faithful morphism $\ell\:i(\O_\M)\ffto\O$ in $\M$ and $\omega$ factors through $\ell$; call $\omega_\M\:1\to \O_\M$ the resulting morphism. Then the lax limit $\tau_\M$ of the arrow $\omega_\M$ is a 2-classifier in $\M$.
	
	In addition to this, if $\tau$ satisfies $\opn{P}$ then also $\tau_\M$ satisfies $\opn{P}$.
	
	Moreover, given $\phi$ a discrete opfibration in $\M$, if $i(\phi)$ is classified by $\tau$ via a characteristic morphism $z$ that factors through $\ell$ then $\phi$ is classified by $\tau_\M$.
	\begin{eqD*}{\begin{cd}*[5][4]
				\1 \arrow[d,"{i(\omega_\M)}"']\arrow[rd,"{\omega}"]\\
				i(\O_\M) \arrow[r,hookrightarrow,"{\ell}"']\&\O 
		\end{cd}}\quad\quad\quad
		{\sq*[l][5][5][\opn{comma}][2.7][2.2][0.65][d]{\O_{\M,\bullet}}{1}{\O_\M}{\O_\M}{}{\tau_\M}{\omega_\M}{}}
		\quad\quad\quad
		{\begin{cd}*[3][2.85]
				i(G)\PB{rrd}\arrow[d,"{i(\phi)}"']\arrow[rr,"{}"] \&\& \O\b \arrow[d,"{\tau}"] \\
				i(F) \arrow[rr,"{z}"] \arrow[rd,dashed,"{\exists\h i(z_\M)}"'{inner sep=0.3ex},shorten >=-0.3ex]\&\& \O \\[-2.1ex]
				\& i(\O_\M)\arrow[ru,hookrightarrow,"{\ell}"',shorten <=-0.65ex]
		\end{cd}}
	\end{eqD*}
\end{prop}
\begin{proof}
	We prove that, for every $F'\in \L$, there is an isomorphism
	\begin{eqD}{triangleGiomegaK}
		\begin{cd}*[2.85][1.2]
			\HomC{\L}{F'}{i(\O_\M)} \arrow[rr,"{\Gg{i(\omega_\M),F'}}"] \arrow[rd,"{\ell\c -}"'{pos=0.38}]\arrow[rr,iso,shift right=3ex]
			\&\& \Fib[b][\L][F'] \\
			\& \HomC{\L}{F'}{\O} \arrow[ru,"{\Gg{\omega,F'}}"'{pos=0.6}]
		\end{cd}
	\end{eqD}
	Given $z\:F'\to i({\O_\M})$, consider the comma objects
	\begin{eqD*}
		\sq*[l][7.5][7.4][{\lambda_\M}]{L_\M}{1}{F'}{i(\O_\M)}{}{s_\M}{i(\omega_\M)}{z}\quad\quad\quad\quad
		\begin{cd}*[3.3][3]
			L \arrow[rr,"{}"]\arrow[dd,"{s}"']\&\&[-1.75ex] 1 \arrow[lldd,Rightarrow,"{\lambda}",shorten <=2.7ex, shorten >=2.2ex] \arrow[d,"{i(\omega_\M)}"]\\
			\&\& i(\O_\M)\arrow[d,"{\ell}"] \\[-1.5ex]
			F'\arrow[r,"{z}"'] \& i(\O_\M) \arrow[r,"{\ell}"']\& \O
		\end{cd}
	\end{eqD*}
	in $\L$. It is straightforward to see that, since $\ell$ is a fully faithful morphism, $\ell\ast \lambda_\M$ exhibits the comma object on the right hand side. Then $L_\M\iso L$ over $F'$, and such isomorphism is natural in $z$ by the universal property of the comma object. $\ell\c -$ is fully faithful by definition of fully faithful morphism in $\L$. Hence ${\Gg{i(\omega_\M),F'}}$ is fully faithful. 
	
	Moreover, we obtain that if $\tau$ satisfies $\opn{P}$ then the lax limit $\xi$ of the arrow $i(\omega_\M)$ satisfies $\opn{P}$. Since $i$ lifts comma objects and the terminal objects, the lax limit $\tau_\M$ of the arrow $\omega_\M$ can be calculated in $\L$. So $i(\tau_\M)=\xi$ satisfies $\opn{P}$ as well.
	
	Let now $F\in \M$. Since $i$ preserves discrete opfibrations, $i$ induces a fully faithful functor
	$$i\:\Fib[b][\M][F]\to \Fib[b][\L][i(F)].$$
	Notice then that the following square is commutative:
	\begin{eqD}{squareGomegaK}
		\begin{cd}*[4.5][6.5]
			\HomC{\M}{F}{\O_\M} \arrow[r,"{\Gg{\omega_\M,F}}"]\arrow[d,iso,"{i\h[6]}"']\& \Fib[b][\M][F] \arrow[d,"{i}"]\\
			\HomC{\L}{i(F)}{i(\O_\M)} \arrow[r,"{\Gg{i(\omega_\M),i(F)}}"'] \& \Fib[b][\L][i(F)]
		\end{cd}
	\end{eqD}
	Indeed by assumption $i$ lifts comma objects and the terminal object. Whence $\Gg{\omega_\M,F}$ is fully faithful and $\tau_\M$ is a 2-classifier.
	
	Consider now a discrete opfibration $\phi$ in $\M$. Following the diagrams of equations~\refs{triangleGiomegaK} and~\refs{squareGomegaK}, we obtain that if $i(\phi)$ is classified by $\tau$ via a characteristic morphism $z$ that factors through $\ell$ then $\phi$ is classified by $\tau_\M$. This can also be seen from the diagram on the right in the statement, using the pullbacks lemma, after \remx\ref{remequivconstauK}.
\end{proof}

\begin{rem}\label{remequivconstauK}
	$\tau_\M$ can be equivalently produced, in $\L$, as the pullback of $\tau$ along the fully faithful $\ell\:i(\O_\M)\ffto \t{\O}$. Indeed the lax limit $\tau_\M$ of the arrow $\omega_\M$ corresponds with the lax limit of the arrow $i(\omega_\M)$ in $\L$. By whiskering with $\ell$, we then see that the latter is equivalently given by the comma object from $\omega$ to $\ell$, which is also the pullback along $\ell$ of the lax limit of the arrow $\omega$. However, by producing $\tau_\M$ as the lax limit of the arrow $\omega_\M$ in $\M$, it is guaranteed that $\tau_\M$ is a morphism in $\M$.
\end{rem}

We would like to show that we can check the factorizations of the characteristic morphisms in $\L$ of discrete opfibrations in $\M$ just on a dense generator. The following construction helps with this.

\begin{cons}\label{consdenseoffullsubcat}
	Let $i\:\M\ffto \L$ be a fully faithful 2-functor. Consider then $I\:\Y\to \M$ a fully faithful 2-functor such that $i\c I\:\Y \to \L$ is a dense generator of $\L$. By Kelly's~\cite[Theorem~5.13]{kelly_basicconceptsofenriched}, then $I$ is a fully faithful dense generator of $\M$.
	
	Moreover, let $F\in \M$. We want to exhibit $F$ as a nice colimit of the objects that form the dense generator $I$. By \conx\ref{lemma}, there exist a 2-diagram $J\:\A\to \L$ which factors through $i\c I$ and a weight $W\:\A\op\to \Cat$ such that, calling $K\deq J\c \groth{W}$,
	$$i(F)=\oplaxncolim{K}$$
	and this colimit is $(i\c I)$-absolute.
	Call $\Lambda$ the universal cartesian-marked oplax cocone that presents such colimit. Notice that, as $K$ factors through $i\c I$, it also factors through $i$. Call $K_\M\:\Groth{W}\to \M$ the resulting diagram; so that $i\c K_\M=K$. It is clear that $K_\M$ factors through $I$. Take $\Lambda_\M$ to be the unique cartesian-marked oplax cocone such that $i\c \Lambda_\M=\Lambda$. Then, since a fully faithful 2-functor reflects colimits (see also \prox\ref{propstillpreserved}),
	$$F=\oplaxncolim{K_\M},$$
	exhibited by $\Lambda_\M$. Moreover, this colimit is $I$-absolute, as $\t{I}\iso \t{(i\c I)}\c i$.
\end{cons}

Building over $\prox\ref{propfactorization}$, we now present a general result of restriction of good 2-classifiers in $\L$ to nice sub-2-categories $\M$ of $\L$. We show that the factorization of the characteristic morphisms in $\L$ of discrete opfibrations in $\M$ can be checked just on a dense generator of the kind described in $\conx\ref{consdenseoffullsubcat}$. Then our theorems of reduction of the study of a 2-classifier to dense generators (\corx\ref{corollgoodtwoclas}) guarantee that we find a good 2-classifier in $\M$. For this, we need to ensure that the operation of normalization described in \thex\ref{esssurj} (starting from every $\phi$) is possible. We show that we can just do the normalization process in $\L$, where it is certainly possible since we have a good 2-classifier; see \corx\ref{corollsharperesssurj}.

\begin{teor}\label{teorfactorization}
	Let $i\:\M\cont \L$ be a nice sub-2-category \pteor{\defx\ref{defnice}}. Let $\omega\:\1\to \O$ in $\L$ be a good 2-classifier in $\L$ with respect to $\opn{P}$. Let then $\O_\M\in \M$ such that there exists a chronic arrow \pteor{\defx\ref{defchronic}} $\ell\:i(\O_\M)\ffto\O$ in $\M$ and $\omega$ factors through $\ell$; call $\omega_\M\:1\to \O_\M$ the resulting morphism. Finally, let $I\:\Y\to \M$ be a fully faithful 2-functor such that $i\c I$ is a dense generator of $\L$. Assume that for every $\psi\:H\to I(Y)$ a discrete opfibration in $\M$ that satisfies $\opn{P}$ over $I(Y)$ with $Y\in \Y$, every characteristic morphism of $i(\psi)$ with respect to $\omega$ factors through $\ell$. Then $\omega_\M$ is a good 2-classifier in $\M$ with respect to $\opn{P}$.
\end{teor}
\begin{proof}
	By \conx\ref{consdenseoffullsubcat}, $I\:\Y\to \M$ is a fully faithful dense generator of $\M$. By \prox\ref{propfactorization}, the lax limit of the arrow $\omega_\M$ satisfies $\opn{P}$ and for every $Y\in \Y$
	$$\Gg{\omega_\M,I(Y)}\:\HomC{\M}{I(Y)}{\O_\M}\to \Fib?[b][n][I(Y)]$$
	is an equivalence of categories. Indeed $\Gg{\omega_\M,I(Y)}$ is fully faithful and, given $\psi\:H\to I(Y)$ a discrete opfibration in $\M$ that satisfies $\opn{P}$, any characteristic morphism of $i(\psi)$ in $\L$ factors through $\ell\:i(\O_\M)\ffto \O$. In order to prove that $\omega_\M\:1\to \O_\M$ is a good 2-classifier in $\M$, by \corx\ref{corollgoodtwoclas}, it only remains to prove that the operation of normalization described in the proof of \thex\ref{esssurj} is possible.
	
	So let $\phi\:G\to F$ be a discrete opfibration in $\M$ that satisfies $\opn{P}$. By \conx\ref{consdenseoffullsubcat}, we express
	$$F=\oplaxncolim{K_\M},$$
	exhibited by $\Lambda_\M$. Looking at the proof of \corx\ref{corollgoodtwoclas} (and \conx\ref{consesssurj}), we consider the sigma natural transformation $\chi_\M$ given by the composite
	$$\Delta 1 \aoplaxn{\Lambda_\M} \HomC{\M}{K_\M(-)}{F}\apseudo{\h[-8]\G{\phi,K_\M(-)}\h[-8]}\Fib?[b][\M][K_\M(-)]\apseudo{\h[-8]{\Gg{\omega_\M,K_\M(-)}}^{-1}\h[-8]} \HomC{\M}{K_\M(-)}{\O_\M}.$$
	We can visualize it as follows:
	\begin{cd}[5][5]
		H_\M^{(C,X)} \PB{rd} \arrow[r,"{}"]\arrow[d,"{\G{\phi}\left(\Lambda_{\M,(C,X)}\right)}"'] \& G \arrow[d,"{\phi}"]\&[1ex] 1 \arrow[d,"{\omega_\M}"]\\
		K_\M(C,X)\arrow[rr,bend right=25,"{\hphantom{CC}\Gg{\omega_\M}^{-1}\left(\G{\phi}\left(\Lambda_{\M,(C,X)}\right)\right)}"'{pos=0.428,inner sep=0.15ex},shorten <=-0.4ex] \arrow[r,"{\Lambda_{\M,(C,X)}}"] \& F\arrow[r,dashed,"{z_\M}"] \& \O_\M
	\end{cd}
	For this, we need to choose an adjoint quasi-inverse of $\Gg{\omega_\M,K_\M(C,X)}$ for every $(C,X)\in \Groth{W}$. Let then $\psi\:H\to K_\M(C,X)$ be a discrete opfibration in $\M$ that satisfies $\opn{P}$. By assumption, the ``normalized" characteristic morphism ${\Gg{\omega,i(K_\M(C,X))}}^{-1}(i(\psi))=t$ defined as in the proof of \corx\ref{corollsharperesssurj} starting from $i(\phi)$ and $K$ and $\Lambda$, factors through $\ell\:i(\O_\M)\ffto \O$. We define ${\Gg{\omega_\M,K_\M(C,X)}}^{-1}(\psi)=t_\M$ to be the morphism in $\M$ corresponding to the resulting morphism $i(K_\M(C,X))\to i(\O_\M)$ in $\L$ given by the factorization. So that $\ell\c i(t_\M)=t$. We then extend ${\Gg{\omega_\M,K_\M(C,X)}}^{-1}$ to a right adjoint quasi-inverse of ${\Gg{\omega_\M,K_\M(C,X)}}$, choosing the components of the counit on the objects $\psi$ to be the isomorphism corresponding to the one in $\L$ from the comma of $i(\omega_\M)$ and $i(t_\M)$ to the comma of $\omega=\ell\c i(\omega_\M)$ and $t=\ell \c i(t_\M)$ composed with the counit of ${\Gg{\omega,K(C,X)}}\dashv {\Gg{\omega,K(C,X)}}^{-1}$.
	
	We prove that $\chi_\M$ is cartesian-marked oplax natural. It suffices to prove that $$(\ell\c-)\c i\c \chi_\M=\chi$$
	where $\chi$ is the cartesian-marked (i.e.\ ``normal") oplax natural transformation produced as in the proof of \corx\ref{corollsharperesssurj}, starting from $i(\phi)$ and $K$ and $\Lambda$. Indeed $\ell$ is a chronic arrow and $i$ is injective on objects and fully faithful. So we show that the following diagram of oplax natural transformations is commutative:
	\begin{cd}[4][5.5]
		\Delta 1 \arrow[d,equal,"{}"] \arrow[r,Rightarrow,"{\Lambda_\M}"] \&[-2ex] \HomC{\M}{K_\M(-)}{F}\arrow[d,Rightarrow,"{i}"]\arrow[r,Rightarrow,"{\G{\phi,K_\M(-)}}"] \& \Fib?[b][\M][K_\M(-)] \arrow[d,Rightarrow,"{i}"]\arrow[r,Rightarrow,"{{\Gg{\omega_\M,K_\M(-)}}^{-1}}"] \& \HomC{\M}{K_\M(-)}{\O_\M}\arrow[d,Rightarrow,"{(\ell\c -)\c i}"] \\
		\Delta 1 \arrow[r,Rightarrow,"{\Lambda}"'] \& \HomC{\L}{K(-)}{i(F)} \arrow[r,Rightarrow,"{\G{i(\phi),K(-)}}"'] \& \Fib?[b][\L][K(-)]\arrow[r,Rightarrow,"{{\Gg{\omega,K(-)}}^{-1}}"'] \& \HomC{\L}{K(-)}{\O}.
	\end{cd}
	The square on the left is commutative by construction of $\Lambda_\M$. The square in the middle is commutative because pullbacks along opfibrations in $\M$ are calculated in $\L$, by assumption. We also use that the structure of a discrete opfibration in $\M$ is just given by the structure of the underlying discrete opfibration in $\L$. We prove that the square on the right is commutative as well. Let $(C,X)\in \Groth{W}$. Given $\psi\:H\to K_\M(C,X)$ a discrete opfibration in $\M$ that satisfies $\opn{P}$,
	$$\ell \c i\left({\Gg{\omega_\M,\y{C}}}^{-1}(\psi)\right)={\Gg{\omega,i(K_\M(C,X))}}^{-1}(i(\psi))$$
	by construction of ${\Gg{\omega_\M,K_\M(C,X)}}^{-1}$. Given $\theta\:\psi\to \psi'$ in $\Fib?[b][\M][K_\M(C,X)]$,
	$$\ell\ast i\left({\Gg{\omega_\M,K_\M(C,X)}}^{-1}(\theta)\right) ={\Gg{\omega,i(K_\M(C,X))}}^{-1}(i(\theta))$$
	because they are equal after applying the fully faithful ${\Gg{\omega,i(K_\M(C,X))}}$, by construction of the counit of ${\Gg{\omega_\M,K_\M(C,X)}}\dashv{\Gg{\omega_\M,K_\M(C,X)}}^{-1}$ (together with the proof of \prox\ref{propfactorization}). Finally, let $(f,\nu)\:(D,X')\al{} (C,X)$ in $\Groth{W}$. The two composite oplax natural transformations of the square on the right also have the same structure 2-cells on $(f,\nu)$. Indeed it suffices to show it after applying the fully faithful ${\Gg{\omega,i(K_\M(C,X))}}$. And this is straightforward to prove, following (part of) the recipe for $\G{\tau}(\chi_{f,\nu})$ described in the proof of \thex\ref{esssurj} (together with the construction of the counit of ${\Gg{\omega_\M,K_\M(C,X)}}\dashv{\Gg{\omega_\M,K_\M(C,X)}}^{-1}$ and the proof of \prox\ref{propfactorization}). We conclude that $\chi_\M$ is cartesian-marked oplax natural.
\end{proof}

\begin{rem}
	\thex\ref{teorfactorization} offers another strategy to produce a good 2-classifier in a 2-category via a dense generator. Indeed, this is what we will do in Section~\ref{sectiontwoclassinstacks} to produce our good 2-classifier in stacks. An advantage of this strategy is that we do not have to do the normalization process described in \thex\ref{esssurj}. By Kelly's~\cite[Proposition~5.16]{kelly_basicconceptsofenriched}, any 2-category $\M$ equipped with a fully faithful dense generator $I\:\Y\to \M$ is equivalent to a full sub-2-category of $\m{\Y\op}{\Cat}$ containing the representables. So, after Section~\ref{sectiontwoclassintwopresheaves}, the strategy described in \thex\ref{teor2classinstacks} can be very helpful.
	
	Notice that the proof of \thex\ref{teorfactorization} shows that we just need to be able to factorize the ``normalized" characteristic morphisms produced as in the proof of \corx\ref{corollsharperesssurj} (starting from every $\phi$).
\end{rem}

\section{A 2-classifier in prestacks}\label{sectiontwoclassintwopresheaves}

In this section, we apply our theorems of reduction of the study of 2-classifiers to dense generators to the case of prestacks. Our theorems offer great benefits here. Indeed they allow us to just consider the classification over representables, which is essentially given by the Yoneda lemma (see \prox\ref{propbicatclasprestacks}). We show that the normalization process required by \thex\ref{esssurj} is possible in prestacks (\thex\ref{teor2classinprestacks} and \remx\ref{remnormalizationinprestacks}). Whence, by \thex\ref{teorfactorization}, it is also possible in any nice sub-2-category of prestacks (such as the 2-category of stacks, see \thex\ref{teor2classinstacks}). 

We thus produce a good 2-classifier in prestacks, in \thex\ref{teor2classinprestacks} (see also \defx\ref{defgood2clas}). Our result is in line with Hofmann and Streicher's~\cite{hofmannstreicher_liftinggrothuniverses} and with the recent Awodey's~\cite{awodey_onhofmannstreicheruniverses}, see \remx\ref{remhofmannstricherawodey}. We conclude the section extracting from the constructive proof of \thex\ref{esssurj} a concrete recipe for the characteristic morphisms in prestacks (\remx\ref{remconcreterecipeforclasmorph}).

In Section~\ref{sectiontwoclassinstacks}, we will restrict the good 2-classifier in prestacks to a good 2-classifier in stacks (\thex\ref{teor2classinstacks}, using \thex\ref{teorfactorization}).

Throughout the rest of this paper, we consider the 2-category $\L=\m{\C\op}{\Cat}$ of 2-presheaves on a small category $\C$ (that is, prestacks on $\C$). Notice that this 2-category is complete and cocomplete, since $\Cat$ is so. Recall from \prox\ref{propdiscopfibinpresheaves} the characterization of discrete opfibrations in \m{\C\op}{\Cat} and from \defx\ref{defhavingsmallfibres} the definition of discrete opfibration in \m{\C\op}{\Cat} with small fibres.

\begin{notation}
	Given $p\:\E\to \B$ a discrete opfibration in $\Cat$ with small fibres, we denote as $(p)_B$ the fibre of $p$ on $B\in \B$.
\end{notation}

\begin{cons}\label{constwoclassinprestacks}
	We search for a good 2-classifier $\omega\:1\to \O$ in $\m{\C\op}{\Cat}$. Looking at the archetypal example of $\Cat$, we expect such a good 2-classifier to classify all discrete opfibrations in $\m{\C\op}{\Cat}$ with small fibres. By \exax\ref{exarepresinprestacks}, representables form a dense generator
	$$I=\yy\:\C\to \m{\C\op}{\Cat}.$$
	Then, by our theorems of reduction to dense generators (see \corx\ref{corollgoodtwoclas}), we will be able to look just at the functors
	$$\Gg{\omega,\y{C}}\:\HomC{\m{\C\op}{\Cat}}{\y{C}}{\O}\to \Fib+[b][n][\y{C}]$$
	with $C\in \C$. Notice that the left hand side is isomorphic to $\O(C)$, by the Yoneda lemma. As we want all the $\Gg{\omega,\y{C}}$'s with $C\in \C$ to be equivalences of categories  (forming then a pseudo-natural equivalence by \prox\ref{pseudonatG}), the assignment of $\O$ is forced up to equivalence to be
	$$C\am{\O} \Fib+[b][n][\y{C}].$$
	This is a nice generalization of what happens in dimension 1. Indeed recall that the subobject classifier in 1-dimensional presheaves sends $C$ to the set of sieves on $C$. And sieves on $C$ are equivalently the subfunctors of $\y{C}$. In line with the philosophy to upgrade subobjects to discrete opfibrations (discussed in~\ref{subsection2class}), discrete opfibrations over $\y{C}$ generalize the concept of sieve to dimension $2$. Notice that $\O$ coincides with the composite
	$$\C\op\arr{\yyop}{\m{\C\op}{\Cat}}\op\ar{\Fib+[b][n][-]}\CATlarge$$
	and is thus a pseudofunctor by \prox\ref{DOpFib-}. We then take as $\omega\:1\to \O$ the pseudonatural transformation with component on $C\in \C$ that picks the identity $\id{\y{C}}$ on $\y{C}$. However, $\O$ is not a strict 2-functor and, a priori, does not land in $\Cat$; so $\O$ cannot be a 2-classifier in $\m{\C\op}{\Cat}$. Thanks to our joint work with Caviglia~\cite{cavigliamesiti_indexedgrothconstr}, we can produce a nice concrete strictification of $\O$. Although it was already known before~\cite{cavigliamesiti_indexedgrothconstr} that any pseudofunctor can be strictified, by the theory developed by Power in~\cite{power_generalcoherenceresult} and later by Lack in~\cite{lack_codescentobjcoherence}, the work of~\cite{cavigliamesiti_indexedgrothconstr} can be applied to produce an explicit and easy to handle strictification of $\O$, which in addition lands in $\Cat$. Moreover, as we will present in Section~\ref{sectiontwoclassinstacks}, such strictification can also be restricted in a natural way to a good 2-classifier in stacks.
\end{cons}
	
\begin{prop}[{\cite[\exax{5.9}, \thex{4.7}, \conx{4.8}]{cavigliamesiti_indexedgrothconstr}}]\label{propstrictification}
	The pseudofunctor $\O\:C\mto \Fib+[b][n][\y{C}]$ is pseudonaturally equivalent to the 2-functor
	\begin{fun}
		\t{\O} & \: & \C\op \hphantom{c}& \too & \hphantom{c}\Cat \\[1.3ex]
		&& C\hphantom{C} & \mto & \m{{\left(\slice{\C}{C}\right)}\op}{\Set}\\[0.3ex]
		&& (C \al{f} D) & \mto & -\c {(f\c =)}\op
	\end{fun}
	The pseudonatural equivalence $j\:\t{\O}\simeq \O$ is given by an indexed version of the Grothendieck construction. Explicitly, a discrete opfibration $\psi\:H\to \y{C}$ with small fibres corresponds with the presheaf
	\begin{fun}
		j_C^{-1}(\psi) & \: & {\left(\slice{\C}{C}\right)}\op & \too & \Set \\[1ex]
		&& (D\ar{f}C) & \mto & {(\psi_D)}_{f}\\[0.6ex]
		&&(f \al{g} f\c g)&\mto& H(g)
	\end{fun}
\end{prop}

\begin{rem}\label{remhofmannstricherawodey}
	Call $j^{-1}$ the quasi-inverse of $j$ described in~\cite[\conx{4.8}]{cavigliamesiti_indexedgrothconstr} and hinted above (that transforms $\psi$ into $j_C^{-1}(\psi)$). The composite $1\ar{\omega}\O\ar{j^{-1}}\t{\O}$ is (isomorphic to) a 2-natural transformation $\t{\omega}$. Explicitly, the component $\t{\omega}_C\:\1\to\t{\O}(C)$ of $\t{\omega}$ on $C\in \C$ picks the constant at $1$ presheaf $\Delta 1\:{\left(\slice{\C}{C}\right)}\op \to \Set$.
	
	We will prove in \thex\ref{teor2classinprestacks} that $\t{\omega}\:1\to\t{\O}$ is a good 2-classifier in $\m{\C\op}{\Cat}$ that classifies all discrete opfibrations with small fibres. This is in line with Hofmann and Streicher's~\cite{hofmannstreicher_liftinggrothuniverses}, where a similar idea is used to construct a universe in $1$-dimensional presheaves for small families, in order to interpret Martin-L\"{o}f type theory in a presheaf topos. See also the recent Awodey's~\cite{awodey_onhofmannstreicheruniverses}, that constructs Hofmann and Streicher's universe in 1-dimensional presheaves in a more conceptual way.
\end{rem}

\begin{rem}
	In~\cite[\exax{5.9}]{cavigliamesiti_indexedgrothconstr} we also suggest a strategy to prove that, for $\C$ a 2-category, what works is $\t{\O}\:C\mto \m{\pi\stb{\left(\oplaxslice{\C}{C}\right)}\op}{\Set}$, where $\pi\stb$ is the left adjoint of the inclusion $\Cat\ito \twoCAT$.
\end{rem}

\begin{prop}\label{propttauhassmallfibres}
	The lax limit $\t{\tau}\:\t{\O}\b\to \t{\O}$ of the arrow $\t{\omega}\:1\to \t{\O}$ is a discrete opfibration in $\m{\C\op}{\Cat}$ with small fibres. 
	
	As a consequence, for every $F\in \m{\C\op}{\Cat}$, the functor 
	$$\Gg{\t{\omega},F}\:\HomC{\m{\C\op}{\Cat}}{F}{\t{\O}}\to \Fib[b][n][F]$$
	lands in $\Fib+[b][n][F]$.
\end{prop}
\begin{proof}
	By \remx\ref{remreplacement} any comma object from $\t{\omega}$ can be expressed as a pullback of $\t{\tau}$, and by \remx\ref{remhavingsmallfibresispbstable} the property of having small fibres is stable under pullback. So we can just look at $\t{\tau}$. Since comma objects in $\m{\C\op}{\Cat}$ are calculated pointwise, for every $C\in \C$ the component $\t{\tau}_C$ of $\t{\tau}$ on $C$ is given by the comma object in $\Cat$ from $\t{\omega}_C=\Delta 1$ to $\id{\t{\O}(C)}$. Given $Z\in \t{\O}(C)=\m{{\left(\slice{\C}{C}\right)}\op}{\Set}$,
	$${\left(\t{\tau}_C\right)}_{Z}\iso \HomC{\t{\O}(C)}{\Delta 1}{Z}=\HomC{\m{{\left(\slice{\C}{C}\right)}\op}{\Set}}{\Delta 1}{Z}\iso Z(\id{C})$$
	and thus $\t{\tau}_C$ has small fibres. Indeed a natural transformation from $\Delta 1$ to $Z$ is the same thing as an element in $Z(\id{C})$, by the naturality condition. This is similar to the proof of the Yoneda lemma; see also \remx\ref{remsimilartoyonedalemma}. 
\end{proof}

\begin{rem}
	$\t{\O}\b$ is a pointed version of $\t{\O}$. The ``points" of $Z\in \t{\O}(C)$ are the elements of $Z(\id{C})$. 
\end{rem}

It will be useful to consider first the bicategorical classification process produced by the pseudofunctor $\O$.

\begin{rem}\label{remlimitsinpseudofunctors}
	We need to consider the 2-category (actually $\CATlarge$-enriched category) $\Psm{\C\op}{\CATlarge}$ of pseudofunctors from $\C\op$ to $\CATlarge$, pseudonatural transformations and modifications. We can of course extend the definition of discrete opfibration in a 2-category to one in any $\CATlarge$-enriched category, considering $\CATlarge$ in the place of $\Cat$.
	
	By Bird, Kelly, Power and Street's~\cite[Remark~7.4]{birdkellypowerstreet_flexiblelimits}, $\Psm{\C\op}{\CATlarge}$ has all bilimits and all flexible limits, calculated pointwise. In particular, it has all comma objects, the terminal object and all pullbacks along discrete opfibrations, calculated pointwise. Indeed, for the latter, recall that in $\CATlarge$ pullbacks along discrete opfibrations exhibit bi-iso-comma objects (the idea is similar to that of the diagram of equation~\refs{1}). So, given $p\:E\to B$ and $z\:F\to B$ in $\Psm{\C\op}{\CATlarge}$ with $p$ a discrete opfibration, we can construct (pointwise) a bi-iso-comma object $G$ of $p$ and $z$ whose universal square is filled with an identity. This is obtained by choosing the pullbacks as representatives for the bi-iso-comma objects in $\CATlarge$ on every component. It follows that $G$ is also a pullback in $\Psm{\C\op}{\CATlarge}$, using that discrete opfibrations lift identities to identities.
\end{rem}

\begin{rem}\label{remsimilartoyonedalemma}
	As hinted in the proof of \prox\ref{propttauhassmallfibres}, the Yoneda lemma is the reason why that proposition holds. Consider the pseudofunctor $\O$ and the lax limit $\tau$ of the arrow $\omega\:1\to \O$ in $\Psm{\C\op}{\CATlarge}$. For every $C\in \C$ and every discrete opfibration $\psi\:H\to \y{C}$ in $\m{\C\op}{\Cat}$ with small fibres, by the Yoneda lemma
	$${\left(\tau_C\right)}_{\psi}\iso \HomC{\O(C)}{\id{\y{C}}}{\psi}\iso{\left(\psi_C\right)}_{\id{C}}.$$
	Thus $\tau_C$ has small fibres.
\end{rem}

\begin{prop}
	For every $F\in \m{\C\op}{\Cat}$, taking comma objects from\linebreak[4] $\omega\:1\to \O$ extends to a functor
	$$\Gg{\omega,F}\:\HomC{\Psm{\C\op}{\CATlarge}}{F}{\O}\to \Fib+[b][\m{\C\op}{\Cat}][F]$$
\end{prop}
\begin{proof}
	Taking comma objects from the morphism $\omega\:1\to \O$ in $\Psm{\C\op}{\CATlarge}$ certainly extends to a functor
	$$\Gg{\omega,F}\:\HomC{\Psm{\C\op}{\CATlarge}}{F}{\O}\to \Fib[b][\Psm{\C\op}{\CATlarge}][F]$$
	by \remx\ref{remreplacement}. But, given $z\:F\to \O$, the comma object
	\sq[l][6.8][6.8][\h[-3]\v[2]\opn{comma}]{L}{1}{F}{\O}{}{s}{\omega}{z}
	in $\Psm{\C\op}{\CATlarge}$ is calculated pointwise. Since $1$ and $F$ are both strict 2-functors, the universal property of the comma object induces a strict 2-functor $L$ and a strict 2-natural transformation $s$. The functor $\Gg{\omega,F}$ also sends modifications between $F$ and $\O$ to strict 2-natural transformations over $F$. Moreover, every component $s_C$ of $s$ on $C\in \C$ needs to be a discrete opfibration in $\CATlarge$ with small fibres, by \remx\ref{remsimilartoyonedalemma} and the fact that the property of having small fibres is pullback-stable. Since $F(C)\in \Cat$, it follows that $s_C$ is a discrete opfibration in $\Cat$ with small fibres. And then $s$ is a discrete opfibration in $\m{\C\op}{\Cat}$ with small fibres, by \prox\ref{propdiscopfibinpresheaves} (and \defx\ref{defhavingsmallfibres}).
\end{proof}

\begin{rem}
	The following proposition shows how the bicategorical classification process in prestacks is essentially given by the Yoneda lemma.
\end{rem}

\begin{prop}\label{propbicatclasprestacks}
	For every $C\in \C$, the functor
	$$\Gg{\omega,\y{C}}\:\HomC{\Psm{\C\op}{\CATlarge}}{\y{C}}{\O}\to \Fib+[b][\m{\C\op}{\Cat}][\y{C}]=\O(C)$$
	is isomorphic to the Yoneda lemma's equivalence of categories. 
	
	Thus $\Gg{\omega,\y{C}}$ is an equivalence of categories.
\end{prop}
\begin{proof}
	Given $z\:\y{C}\to \O$, call $\ov{z}\:G\to \y{C}$ the corresponding element in $\O(C)$ via the Yoneda lemma and $s\:L\to \y{C}$ the morphism on the left of the comma object
	\sq[l][6.65][6.65][\h[-3]\v[2]\opn{comma}]{L}{1}{\y{C}}{\O}{}{s}{\omega}{z}
	in $\Psm{\C\op}{\CATlarge}$. We show that there is a 2-natural isomorphism $L\iso G$ over $\y{C}$. Given $D\in \C$, we have that $L(D)\iso G(D)$ over $\HomC{\C}{D}{C}$ because of the following bijection between the fibres, which is natural in $f\:D\to C$
	$${(s_D)}_f\iso\HomC{\O(D)}{\id{\y{D}}}{z_D(f)}\iso \HomC{\O(D)}{\id{\y{D}}}{\y{f}\st \ov{z}}\iso {((\y{f}\st \ov{z})_D)}_{\id{D}}\iso {(\ov{z}_D)}_f$$
	The first isomorphism is given by the explicit construction of comma objects in $\CATlarge$. It is natural by construction of the structure of discrete opfibration on $s$ induced by the comma object (see \remx\ref{remreplacement}). The second natural isomorphism is given by pseudonaturality of $z$. The third one is given by the Yoneda lemma and trivially natural. The fourth one is given by the explicit construction of pullbacks in $\Cat$ and is natural by construction of $\G{\ov{z},\y{D}}$ on morphisms. It is straightforward to show that the isomorphism $L(D)\iso G(D)$ is 2-natural over $\y{C}$ and to conclude the proof.
\end{proof}

We are ready to prove that, at least over representables, $\t{\omega}\:1\to \t{\O}$ satisfies the conditions of a good 2-classifier in prestacks that classifies all discrete opfibrations with small fibres.

\begin{prop}\label{propclassprestacksonrepresentables}
	For every $C\in \C$, the functor
	$$\Gg{\t{\omega},\y{C}}\:\HomC{\m{\C\op}{\Cat}}{\y{C}}{\t{\O}}\to \Fib+[b][n][\y{C}]$$
	is an equivalence of categories.
\end{prop}
\begin{proof}
	We prove that there is an isomorphism
	\begin{eqD}{isoGomegatildeomega}
		\begin{cd}*[2][-2]
			\HomC{\m{\C\op}{\Cat}}{\y{C}}{\t{\O}} \arrow[rr,"{\Gg{\t{\omega},\y{C}}}"] \arrow[rd,"{j\c -}"'{pos=0.38}]\arrow[rr,iso,shift right=3ex]
			\&\& \Fib+[b][\m{\C\op}{\Cat}][\y{C}] \\
			\& \HomC{\Psm{\C\op}{\CATlarge}}{\y{C}}{\O} \arrow[ru,"{\Gg{\omega,\y{C}}}"'{pos=0.6}]\v[-0.3]
		\end{cd}
	\end{eqD}
	Given $z\:\y{C}\to \t{\O}$, consider the comma objects
	\begin{eqD*}
		\sq*[l][7.4][7.4][\t{\lambda}]{\t{L}}{1}{\y{C}}{\t{\O}}{}{\t{s}}{\t{\omega}}{z}\quad\quad\quad\quad
		\begin{cd}*[3.2][3]
			L \arrow[rr,"{}"]\arrow[dd,"{s}"']\&\&[-1.5ex] 1 \arrow[lldd,Rightarrow,"{\lambda}",shorten <=2.7ex, shorten >=2.2ex] \arrow[d,"{\t{\omega}}"]\\
			\&\& \t{\O}\arrow[d,"{j}"{pos=0.47}] \\[-1.5ex]
			\y{C}\arrow[r,"{z}"'] \& \t{\O} \arrow[r,"{j\h[2]}"']\& \O
		\end{cd}
	\end{eqD*}
	respectively in $\m{\C\op}{\Cat}$ and in $\Psm{\C\op}{\CATlarge}$. Notice that the left hand side comma is also a comma object in $\Psm{\C\op}{\CATlarge}$ because commas are calculated pointwise in both 2-categories. We have that $\t{L}(D)\iso L(D)$ over $\HomC{\C}{D}{C}$ for every $D\in \C$, since $j_D$ is an equivalence of categories. Indeed $(j\ast \t{\lambda})_D$ exhibits the comma object on the right hand side in component $D$. It is straightforward to show that the isomorphism $\t{L}(D)\iso L(D)$ is 2-natural in $D\in \C$ over $\y{C}$ and to conclude the isomorphism of equation~\refs{isoGomegatildeomega}. Notice that $j\c\t{\omega}$ is isomorphic to $\omega$, whence $\Gg{\omega,\y{C}}$ is isomorphic to $\Gg{j\c\t{\omega},\y{C}}$.
	
	By \prox\ref{propbicatclasprestacks}, the functor $\Gg{\omega,\y{C}}$ is an equivalence of categories. By the Yoneda lemma, also $j\c -$ is an equivalence of categories. Indeed the following square is commutative:
	\begin{cd}[3.9][4.3]
		{\HomC{\m{\C\op}{\Cat}}{\y{C}}{\t{\O}}} \arrow[d,iso,"{}"] \arrow[r,"{j\c -}"] \& {\HomC{\Psm{\C\op}{\CATlarge}}{\y{C}}{\O}}\arrow[d,simeq,"{}"] \\
		{\t{\O}(C)} \arrow[r,"{j_C}"']\& {\O(C)}
	\end{cd}
	And thus $j\c -$ is an equivalence of categories by the two out of three property. Therefore also $\Gg{\t{\omega},\y{C}}$ is an equivalence of categories.
\end{proof}

	We now apply the theorems of reduction of 2-classifiers to dense generators to prove that $\t{\omega}\:1\to \t{\O}$ is a good 2-classifier in prestacks that classifies all discrete opfibrations with small fibres. The partial result that (the lax limit of the arrow) $\t{\omega}$ gives a 2-classifier can be obtained from Weber's~\cite[Example~4.7]{weber_yonfromtwotop}. However, Weber's paper does not address the problem of which discrete opfibrations get classified.
	
\begin{teor}\label{teor2classinprestacks}
	The 2-natural transformation $\t{\omega}$ from $1$ to
	\begin{fun}
		\t{\O} & \: & \C\op \hphantom{c}& \too & \hphantom{c}\Cat \\[1.3ex]
		&& C\hphantom{C} & \mto & \m{{\left(\slice{\C}{C}\right)}\op}{\Set}\\[0.3ex]
		&& (C \al{f} D) & \mto & -\c {(f\c =)}\op
	\end{fun}
	that picks the constant at 1 presheaf on every component is a good $2$-classifier in $\m{\C\op}{\Cat}$ that classifies all discrete opfibrations with small fibres.
\end{teor}
\begin{proof}
	Consider the fully faithful dense generator $\yy\:\C\to \m{\C\op}{\Cat}$ formed by representables. By \prox\ref{propttauhassmallfibres}, the lax limit of the arrow $\t{\omega}$ has small fibres. By \prox\ref{propclassprestacksonrepresentables}, we have that for every $C\in \C$ the functor
	$$\Gg{\t{\omega},\y{C}}\:\HomC{\m{\C\op}{\Cat}}{\y{C}}{\t{\O}}\to \Fib+[b][n][\y{C}]$$
	is an equivalence of categories. In order to prove that $\t{\omega}\:1\to \t{\O}$ is a good 2-classifier in $\m{\C\op}{\Cat}$ with respect to the property of having small fibres, by \corx\ref{corollgoodtwoclas}, it only remains to prove that the operation of normalization described in \thex\ref{esssurj} is possible.
	
	So let $\phi\:G\to F$ be a discrete opfibration in $\m{\C\op}{\Cat}$ with small fibres. Using the dense generator $\yy\:\C\to \m{\C\op}{\Cat}$, we express $F$ as a cartesian-marked oplax colimit of representables. By \exax\ref{univoplaxnormalcoconepresheaves},
	$$F\iso \wcolim{F}{\yy}\iso \oplaxncolim{\left(\yy\c\h{\groth{F}}\right)},$$
	whence $K=\yy\c\h{\groth{F}}$, with the universal cartesian-marked oplax cocone $\Lambda$ given by
	$$\forall \fib[o][2.4]{(D,X')}{(f,\nu)}{(C,X)} \text{ in } {\Grothdiag{F}} \quad\quad\quad{
		\begin{cd}*[3.8][1.8]
			\y{C} \arrow[rr,"{\ceil{X}}",""'{name=Q}] \arrow[d,"{\y{f}}"'] \&\& F\\
			\y{D} \arrow[rru,bend right,"{\ceil{X'}}"']\& \phantom{.} \arrow[Rightarrow,from=Q,"{\ceil{\nu}}"{pos=0.36},shift right=1.5ex,shorten <=1ex,shorten >= 2.7ex]
	\end{cd}}$$	
	Looking at the proof of \corx\ref{corollgoodtwoclas} (and \conx\ref{consesssurj}), we consider the sigma natural transformation $\chi$ given by the composite
	$$\Delta 1 \aoplaxn{\Lambda} \HomC{\m{\C\op}{\Cat}}{K(-)}{F}\apseudo{\h[-8]\G{\phi,K(-)}\h[-8]}\Fib+[b][n][K(-)]\apseudo{\h[-8]{\Gg{\t{\omega},K(-)}}^{-1}\h[-8]} \HomC{\m{\C\op}{\Cat}}{K(-)}{\t{\O}}.$$
	We can visualize it as follows:
	\begin{cd}[5][5]
		H^{(C,X)} \PB{rd} \arrow[r,"{}"]\arrow[d,"{\G{\phi}\left(\Lambda_{(C,X)}\right)}"'] \& G \arrow[d,"{\phi}"]\&[1ex] 1 \arrow[d,"{\t{\omega}}"]\\
		\y{C}\arrow[rr,bend right=25,"{\hphantom{CC}\Gg{\t{\omega}}^{-1}\left(\G{\phi}\left(\Lambda_{(C,X)}\right)\right)}"'{pos=0.428,inner sep=0.15ex},shorten <=-0.4ex] \arrow[r,"{\Lambda_{(C,X)}}"] \& F\arrow[r,dashed,"{z}"] \& \t{\O}
	\end{cd}
	For this, we need to choose an adjoint quasi-inverse of $\Gg{\t{\omega},K(C,X)}=\Gg{\t{\omega},\y{C}}$ for every $(C,X)\in \Groth{F}$. By \prox\ref{propclassprestacksonrepresentables}, we can construct such a quasi-inverse by taking quasi-inverses of $j\c -$ and $\Gg{\omega,\y{C}}$. Both the latter are given by the Yoneda lemma, respectively by the proof of \prox\ref{propclassprestacksonrepresentables} and by \prox\ref{propbicatclasprestacks}. So given $\psi\:H\to \y{C}$, we can take $\Gg{\t{\omega},\y{C}}^{-1}(\psi)$ to be the morphism $\y{C}\to \t{\O}$ which corresponds to $j_C^{-1}(\psi)$ (see \prox\ref{propstrictification}). With this choice, $\chi_{(C,X)}$ would be the morphism $\y{C}\to \t{\O}$ which corresponds to
	\begin{fun}
		j_C^{-1}\left(\G{\phi}\left(\Lambda_{(C,X)}\right)\right) & \: & {\left(\slice{\C}{C}\right)}\op & \too & \Set \\[1ex]
		&& (D\ar{f}C) & \mto & {\left(\G{\phi}\left(\Lambda_{(C,X)}\right)_D\right)}_{f}\\[0.6ex]
		&&(f \al{g} f\c g)&\mto& H^{(C,X)}(g)
	\end{fun}
	Using the explicit construction of pullbacks in $\Cat$ to calculate ${\left(\G{\phi}\left(\Lambda_{(C,X)}\right)_D\right)}_{f}$, we do not obtain a cartesian-marked oplax natural transformation $\chi$. This is due to the unnecessary keeping track of the morphisms $f\:D\to C$ other than the objects of $G(D)$.
	
	Instead, we choose ${\Gg{\t{\omega},\y{C}}}^{-1}$ on the objects $\G{\phi}\left(\Lambda_{(C,X)}\right)$ so that $\chi_{(C,X)}$ is the morphism $\y{C}\to \t{\O}$ which corresponds to
	\begin{fun}
		\ceil{\chi_{(C,X)}} & \: & {\left(\slice{\C}{C}\right)}\op & \too & \Set \\[1ex]
		&& (D\ar{f}C) & \mto & {\left(\phi_D\right)}_{F(f)(X)}\\[0.6ex]
		&&(f \al{g} f\c g)&\mto& G(g)
	\end{fun}
	This is the operation of normalization that we need. Notice that $${\left(\phi_D\right)}_{F(f)(X)}={\left(\phi_D\right)}_{\Lambda_{(C,X)}(f)}\iso {\left(\G{\phi}(\Lambda_{(C,X)})_D\right)}_{f}$$
	and that such isomorphism is natural in $f\in {\left(\slice{\C}{C}\right)}\op$. So that, thanks to the argument above, $\Gg{\t{\omega},\y{C}}(\chi_{(C,X)})\:Q^{(C,X)}\to \y{C}$ is indeed isomorphic to $\G{\phi}\left(\Lambda_{(C,X)}\right)\:H^{(C,X)}\to \y{C}$. We then extend ${\Gg{\t{\omega},\y{C}}}^{-1}$ to a right adjoint quasi-inverse of ${\Gg{\t{\omega},\y{C}}}$, choosing the components of the counit on the objects $\G{\phi}\left(\Lambda_{(C,X)}\right)$ to be the just obtained isomorphisms. 
	
	We prove that $\chi$ is cartesian-marked oplax natural. Given a morphism $(f,\id{})\:(D,X')\al{} (C,F(f)(X))$ in ${\left(\Groth{F}\right)}\op$, it is straightforward to show that
	$$\chi_{(D,X')}\c \y{f}=\chi_{(C,F(f)(X))}$$
	using that $F$ is a strict 2-functor. We still need to show that $\chi_{f,\id{}}=\id{}$. For this, it is straightforward to prove that 
	$$\Gg{\t{\omega},\y{C}}(\chi_{f,\id{}})=\id{},$$
	following the recipe given in the proof of \thex\ref{esssurj}. So we conclude using the fully faithfulness of $\Gg{\t{\omega},\y{C}}$.
\end{proof}

\begin{rem}\label{remnormalizationinprestacks}
	Looking at the proof of \thex\ref{teor2classinprestacks}, we see that the idea of the operation of normalization in prestacks is the following. Rather than considering the ``local fibres" of the $\G{\phi}(\Lambda_{(C,X)})$'s, that are not compatible with each other, we express all of them in terms of the ``global fibres" of $\phi$.
\end{rem}

\begin{rem}\label{remconcreterecipeforclasmorph}
	The proof of \thex\ref{teor2classinprestacks} also gives us a recipe for the characteristic morphism $z\:F\to \t{\O}$ of a discrete opfibration $\phi\:G\to F$ in $\m{\C\op}{\Cat}$ with small fibres.
	\sq[l][6][6][\h[-3]\v[2]\opn{comma}]{G}{1}{F}{\t{\O}}{}{\phi}{\t{\omega}}{z}
	We obtain that $z$ is the 2-natural transformation whose component on $C\in \C$ is the functor $z_C$ that sends $X\in F(C)$ to
	\begin{fun}
		z_C(X) & \: & {\left(\slice{\C}{C}\right)}\op & \too & \Set \\[1ex]
		&& (D\ar{f}C) & \mto & {\left(\phi_D\right)}_{F(f)(X)}\\[0.6ex]
		&&(f \al{g} f\c g)&\mto& G(g)
	\end{fun}
	Given $\nu\:X\to X'$ in $F(C)$, we have that $z_C(\nu)$ is the natural transformation whose component on $f\:D\to C$ is the function
	$${F(f)(\nu)}\stb\:{\left(\phi_D\right)}_{F(f)(X)}\to {\left(\phi_D\right)}_{F(f)(X')}$$
	that calculates the codomain of the liftings along $\phi_D$ of $F(f)(\nu)$.
	
	It is interesting to compare our result with what happens in dimension 1. The characteristic morphism for a subobject $G\ito F$ in $1$-dimensional presheaves has component on $C$ that sends $X\in F(C)$ to
	$$\Union{D\in \C}{\set{D\ar{f}C}{F(f)(X)\in G(D)}}.$$
	While in dimension 1 the fibre on $F(f)(X)$ can only be either empty or a singleton, in dimension 2 we need to handle the general sets formed by such fibres.	
\end{rem}

\section{A 2-classifier in stacks}\label{sectiontwoclassinstacks}

In this section, we restrict our good 2-classifier in prestacks (\thex\ref{teor2classinprestacks}) to a good 2-classifier in stacks that classifies all discrete opfibrations with small fibres (\thex\ref{teor2classinstacks}). We follow the strategy described in the general \thex\ref{teorfactorization} to restrict a good 2-classifier to a nice sub-2-category. The idea is to select, out of all the presheaves on slices involved in the definition of $\t{\O}$, the sheaves with respect to the Grothendieck topology induced on the slices. This restriction of $\t{\O}$ is tight enough to give a stack $\O_J$, but at the same time loose enough to still host the classification process of prestacks.

Our result solves a problem posed by Hofmann and Streicher in~\cite{hofmannstreicher_liftinggrothuniverses}. Indeed, in a different context, they considered the same natural idea to restrict their analogue of $\t{\O}$ by taking sheaves on slices. However, this did not work for them, as it does not give a sheaf. Our results show that such a restriction yields nonetheless a stack and a good 2-classifier in stacks.

As explained in \remx\ref{remourstacks}, we take strictly functorial stacks with the respect to a subcanonical topology $J$. Throughout this section, we consider the full sub-2-category $\St[J]{\C}$ of $\m{\C\op}{\Cat}$ on stacks. Call $i\:\St[J]{\C}\ffto \m{\C\op}{\Cat}$ the injective on objects and fully faithful 2-functor of inclusion.

We want to show that $i$ satisfies all the assumptions of \thex\ref{teorfactorization}.

The following proposition does not seem to appear in the literature.

\begin{prop}\label{proplimitsinstacks}
	The 2-category $\St+[J]{\C}$ of pseudofunctorial stacks has all bilimits and all flexible limits, calculated in $\Psm{\C\op}{\Cat}$ and hence pointwise.
	
	$\St[J]{\C}$ has all flexible limits \pteor{thus all comma objects and the terminal object} and all pullbacks along discrete opfibrations, calculated in $\m{\C\op}{\Cat}$ and hence pointwise.
\end{prop}
\begin{proof}
	By \thex\ref{stacksarebireflective}, $\St+[J]{\C}$ is a bireflective sub-2-category of $\Psm{\C\op}{\Cat}$. So by \remx\ref{remlimitsinpseudofunctors} it has all bilimits, calculated in $\Psm{\C\op}{\Cat}$. Consider then a flexible weight $W$ and a 2-diagram $F$ in $\St+[J]{\C}$. Then the flexible limit of $F$ weighted by $W$ exists in $\Psm{\C\op}{\Cat}$, by \remx\ref{remlimitsinpseudofunctors}. In particular, by flexibility of $W$, it satisfies the universal property of a bilimit of a 2-diagram that factors through $\St+[J]{\C}$. And it is then a pseudofunctorial stack. It follows that it is the flexible limit of $F$ weighted by $W$ in $\St+[J]{\C}$, since fully faithful 2-functors reflect 2-limits.
	
	Consider now a 2-diagram $F$ in $\St[J]{\C}$ and a flexible weight $W$. The flexible limit of $F$ weighted by $W$ exists in $\m{\C\op}{\Cat}$, calculated pointwise. Then it is also the flexible limit in $\Psm{\C\op}{\Cat}$, as the latter is as well calculated pointwise. So it is a stack, as the 2-diagram in $\Psm{\C\op}{\Cat}$ factors through $\St+[J]{\C}$. Whence we have produced the flexible limit of $F$ weighted by $W$ in $\St[J]{\C}$.
	
	Finally, consider $p\: E\to B$ and $z\: F\to B$ in $\St[J]{\C}$, with $p$ a discrete opfibration. The pullback of $p$ and $z$ exists in $\m{\C\op}{\Cat}$, calculated pointwise. Then it is also the bi-iso-comma object of $p$ and $z$ in $\Psm{\C\op}{\Cat}$, by \remx\ref{remlimitsinpseudofunctors}. We conclude that it is a stack and hence the pullback of $p$ and $z$ in $\St[J]{\C}$, by the argument above.
\end{proof}

\begin{prop}\label{propdiscopfibinstacks}
	A morphism in $\St[J]{\C}$ is a discrete opfibration if and only if its underlying morphism in $\m{\C\op}{\Cat}$ is so. In particular, $i\:\St[J]{\C}\ffto \m{\C\op}{\Cat}$ preserves discrete opfibrations.
\end{prop}
\begin{proof}
	The ``if" part is clear by the fact that $\St[J]{\C}\ffto \m{\C\op}{\Cat}$ is fully faithful. The ``only if" part follows from \prox\ref{propdiscopfibinpresheaves}, since $J$ is subcanonical.
\end{proof}

\begin{rem}\label{remstacksformnicesub}
	Putting together \prox\ref{proplimitsinstacks} and \prox\ref{propdiscopfibinstacks}, we have thus proved that $i\:\St[J]{\C}\cont \m{\C\op}{\Cat}$ is a nice sub-2-category (\defx\ref{defnice}).
\end{rem}

\begin{defne}
	We say that a discrete opfibration $\phi$ in $\St[J]{\C}$ \dfn{has small fibres} if $i(\phi)$ has small fibres. Notice that this is in line with \remx\ref{rempropPinsub2cat}.
\end{defne}

\begin{prop}\label{propffinprestacks}
	Let $l\:F\to B$ be a morphism in $\m{\C\op}{\Cat}$ \pteor{that is, a 2-natural transformation $l$}. $l$ is a fully faithful morphism if and only if for every $C\in \C$ the component $l_C$ of $l$ on $C$ is a fully faithful functor.
	
	$l$ is chronic \pteor{\defx\ref{defchronic}} if and only if for every $C\in \C$ the component $l_C$ of $l$ on $C$ is an injective on objects and fully faithful functor.
\end{prop}
\begin{proof}
	The proof is straightforward.
\end{proof}

\begin{cons}
	We want to produce the object $\O_\M$ of \thex\ref{teorfactorization} in our case with $i\:\St[J]{\C}\cont \m{\C\op}{\Cat}$. That is, a stack, which we will call $\O_J$, that is a nice restriction of the good 2-classifier $\t{\O}$ in prestacks.
	Recall that, in dimension 1, the subobject classifier in sheaves is given by taking closed sieves. We produce a 2-categorical notion of closed sieve. 
	
	We have already said in \conx\ref{constwoclassinprestacks} that discrete opfibrations over representables generalize the concept of sieve to dimension 2; we call them 2-sieves. Thanks to \prox\ref{propstrictification} (indexed Grothendieck construction, explored in our joint work with Caviglia~\cite{cavigliamesiti_indexedgrothconstr}), we can equivalently consider presheaves on slice categories. We now need to generalize closedness of a sieve to dimension 2. The indexed Grothendieck construction can be restricted to a bijection between 1-dimensional sieves on $C\in \C$ and presheaves ${\left(\slice{\C}{C}\right)}\op\to \2$. It can be shown that closed 1-sieves correspond with sheaves ${\left(\slice{\C}{C}\right)}\op\to \2$, and that the closure of a sieve corresponds to the sheafification of the corresponding presheaves. So we define \dfn{closed 2-sieves} to be the sheaves ${\left(\slice{\C}{C}\right)}\op\to \Set$ (with respect to the Grothendieck topology induced by $J$ on the slices, that we call again $J$). And we use them to restrict our good 2-classifier $\omega\:1\to\t{\O}$ in prestacks to a good 2-classifier $\omega_J\:1\to\O_J$ in stacks (\thex\ref{teor2classinstacks}).
\end{cons}

\begin{rem}\label{remclosed2sieves}
	The ``maximal" 2-sieve $\id{\y{C}}$, associated with $\t{\omega}_C=\Delta1\:{\left(\slice{\C}{C}\right)}\op \to \Set$ (see \remx\ref{remhofmannstricherawodey}), is a closed 2-sieve.
	
	Closed 2-sieves are stable under pullbacks. Indeed if $F\:{\slice{\C}{C}}\op\to \Set$ is a sheaf and $f\:D\to C$ is a morphism in $\C$ then also $F\c (f\c =)\:{\slice{\C}{D}}\op\to \Set$ is a sheaf.
\end{rem}

\begin{prop}\label{propOJisastack}
	The 2-functor
	\begin{fun}
		{\O_J} & \: & \C\op \hphantom{c}& \too & \hphantom{c}\Cat \\[1.3ex]
		&& C\hphantom{C} & \mto & \Sh[J]{\slice{\C}{C}}\\[0.3ex]
		&& (C\al{f} D) & \mto & -\c {(f\c =)}\op\v[0.5]
	\end{fun}
	is a stack with respect to the Grothendieck topology $J$.\v
	
	Moreover, the inclusions $\Sh[J]{\slice{\C}{C}}\ffto \m{{\left(\slice{\C}{C}\right)}\op}{\Set}$ form a chronic arrow \linebreak[4]${\ell\:i(\O_J)\ffto\t{\O}}$ in $\m{\C\op}{\Cat}$. And $\t{\omega}\:1\to \t{\O}$ factors through $\ell$; call $\omega_J\:1\to \O_J$ the resulting morphism.
\end{prop}
\begin{proof}
	The second part of the statement is clear after \prox\ref{propffinprestacks} and \remx\ref{remclosed2sieves}. 
	
	We prove that $\O_J$ is a stack (recall from \defx\ref{defstack} the definition). So let $C\in \C$ and $S\in J(C)$ a covering sieve on $C$. 
	
	We first prove the uniqueness of gluings of morphisms. Let $M,N\in \Sh[J]{\slice{\C}{C}}$ and let $\alpha,\beta\:M\aR{}N$ two natural transformations such that $f\st \alpha=f\st \beta$ for every $(D\ar{f}C)\in S$. We show that $\alpha=\beta$. Given $(D\ar{f}C)\in S$,
	$$\alpha_f=(f\st \alpha)_{\id{D}}=(f\st \beta)_{\id{D}}=\beta_f.$$
	Let now $g\:E\to \C$ in $\C$ and consider $g\st S\in J(g)=J(E)$. Since $M$ is a sheaf,
	$$M(g)\iso \Match[M]{g\st S}$$
	where the right hand side denotes the set of matching families for $M$ with respect to the covering sieve $g\st S$. And this holds analogously for $N$. We produce a commutative square
	\begin{cd}[4][4.5]
		M(g) \arrow[d,iso,"{}"]\arrow[r,"{\alpha_g}"]\& N(g)\arrow[d,iso,"{}"] \\
		\Match[M]{g\st S}\arrow[r,"{[\alpha_g]}"'] \& \Match[N]{g\st S}
	\end{cd}
	and analogously for $\beta$. We define $[\alpha_g]$ to send a matching family $(h\in g\st S)\am{m} (x_h\in M(g\c h))$ to the matching family $(h\in g\st S)\mto (\alpha_{g\c h}(x_h)\in N(g\c h))$. The latter is indeed a matching family by naturality of $\alpha$. The square above commutes since, for every $m\in \Match[M]{g\st S}$, calling $X$ the amalgamation of $m$, we have that $\alpha_g(X)$ is an amalgamation of $[\alpha_g](m)$. Notice now that $[\alpha_g]=[\beta_g]$, as for every $h\in g\st S$ we have that $g\c h\in S$ and hence $\alpha_{g\c h}=\beta_{g\c h}$. Thus $\alpha_g=\beta_g$.
	
	We prove that we have the gluings of morphisms. Let $M,N\in \Sh[J]{\slice{\C}{C}}$ and consider a matching family
	$$(D\ar{f} C)\in S \h[3]\mto \h[3] (\alpha_f\:f\st M\aR{}f\st N) \text{ in } \Sh[J]{\slice{\C}{D}}$$
	So that for every $D'\ar{l}D\ar{f}C$ with $f\in S$ it holds that $l\st \alpha_f=\alpha_{f\c l}$. We produce a natural transformation $\lambda\:M\aR{}N$ such that $f\st\lambda=\alpha_f$ for every $(D\ar{f}C)\in S$. Given $(D\ar{f}C)\in S$, we would like to define $\lambda_f\deq {(\alpha_f)}_{\id{D}}$. Let $g\:E\to C$ in $\C$. We define $\lambda_g$ to be the composite
	$$M(g)\iso \Match[M]{g\st S} \ar{[\lambda_g]}\Match[N]{g\st S}\iso N(g)$$
	where $[\lambda_g]$ sends a matching family $(h\in g\st S)\am{m} (x_h\in M(g\c h))$ to the matching family $(h\in g\st S)\mto ((\alpha_{g\c h})_{\id{}}(x_h)\in N(g\c h))$. The latter is indeed a matching family by naturality of $\alpha_{g\c h}$. It is then straightforward to prove that $\lambda$ is natural. Given $(D\ar{f}C)\in S$, we have that $\lambda_f={(\alpha_f)}_{\id{D}}$, since $\id{D}\in f\st S$ and hence the amalgamation of any matching family on $f\st S$ is just the datum on $\id{D}$. So $f\st \lambda=\alpha_f$. Indeed for every $l\:D'\to D$ in $\C$
	$$(f\st \lambda)_l=\lambda_{f\c l}={(\alpha_{f\c l})}_{\id{E}}={(l\st \alpha_f)}_{\id{E}}={(\alpha_f)}_l.$$
	
	It remains to prove that we have the gluing of objects. So consider a descent datum
	$$(D\ar{f} C)\in S \h[3]\mto \h[3] M_f\in\Sh[J]{\slice{\C}{D}}$$
	with $\phi^{f,h}\:h\st M_f\iso M_{f\c h}$ such that the cocycle condition
	\begin{cd}[5][5]
		k\st h\st M_f \arrow[d,"", iso] \arrow[r,"k\st \phi^{f,h}"] \&   k\st M_{f\c h} \arrow[d,"\phi^{f\c h,k}"] \\
		(h\c k)\st M_f \arrow[r,"\phi^{f,h\c k}"']\& {M_{f\c h\c k}.}
	\end{cd}
	holds for every $D''\ar{k} D' \ar{h}  D\ar{f} C$ with $f\in S$. We produce $M\in \Sh[J]{\slice{\C}{C}}$ and for every $(D\ar{f}C)\in S$ isomorphisms $\psi^f\:f\st M\iso M_f$ such that
	\begin{cd}[5][5]
		h\st f\st M \arrow[d,"",iso] \arrow[r,"h\st \psi^{f}"] \&   {h\st M_f} \arrow[d,"\phi^{f,h}"] \\
		(f\c h)\st M\arrow[r,"\psi^{f\c h}"']\& {M_{f\c h}.}
	\end{cd}
	for every $D'\ar{h} D \ar{f} C$ with $f\in S$. We construct the presheaf \v[1]
	\begin{fun}
		Z & \: & {\left(\slice{\C}{C}\right)}\op & \too & \Set \\[1.3ex]
		&&\parbox{2.05cm}{\trslice*+[3.7][3.7]{D}{C}{D'}{f\in S}{h}{f'}}&\mtoo&\h[-3]\fib[n][2.7]{M_f(\id{D})}{Z(\underline{h})}{M_{f'}(\id{D'})}\\[1.3ex]
		&&\p{D\ar{g\notin S}C}&\mtoo& \hphantom{C}\emp
	\end{fun}
	where $Z(\underline{h})$ is the composite
	$$M_f(\id{D})\ar{M_f(\underline{h})} M_f(h)=(h\st M_f)(\id{D'})\aisoo[\phi^{f,h}_{\id{D'}}] M_{f\c h}(\id{D'}).\v[-1]$$
	$Z$ is indeed a functor, by the cocycle condition. Moreover, it is straightforward to show that $f\st Z\iso M_f$ for every $(D\ar{f}C)\in S$. However, $Z$ is not a sheaf. So we define $M\deq Z^{++}$, where $Z^+$ is the plus construction of $Z$ and hence $Z^{++}$ is the sheafification of $Z$. It is straightforward to check that $(f\st Z)^+\iso f\st (Z^+)$, by the explicit plus construction. Thus, using that $f\st Z\iso M_f$, we define $\psi^f$ to be the composite
	$$f\st (Z^{++})\iso (f\st Z^+)^+\iso (f\st Z)^{++}\iso (M_f)^{++}\iso M_f,$$
	where the last isomorphism is given by the fact that $M_f$ is a sheaf. It is then straightforward to show that the isomorphisms $\psi^f$ satisfy the required condition.
\end{proof}

\begin{rem}\label{remrepresentablesinstacks}
	Representables form a fully faithful dense generator $\yy\:\C\to \St[J]{\C}$ of the kind described in \conx\ref{consdenseoffullsubcat}. We want to apply \thex\ref{teorfactorization} on such a dense generator. So we need to factorize the characteristic morphisms in $\m{\C\op}{\Cat}$ of discrete opfibrations with small fibres in $\St[J]{\C}$ over representables.
\end{rem}

\begin{prop}\label{propfactorizationofcharmorinstacks}
	For every $\psi\:H\to \y{C}$ a discrete opfibration in $\St[J]{\C}$ with small fibres, with $C\in \C$, every characteristic morphism of $i(\psi)$ with respect to $\t{\omega}$ factors through $\ell\:i(\O_J)\ffto \t{\O}$.
\end{prop}
\begin{proof}
	It suffices to prove that the characteristic morphism $z$ for $i(\phi)$ with respect to $\t{\omega}$ produced in \remx\ref{remconcreterecipeforclasmorph} (and \thex\ref{teor2classinprestacks}) factors through $\ell$. Indeed such factorization only depends on the isomorphism class of $z$, by the Yoneda lemma, as any presheaf isomorphic to a sheaf is a sheaf. By \remx\ref{remconcreterecipeforclasmorph}, $z$ is the 2-natural transformation $\y{C}\to \t{\O}$ that corresponds with the functor
	\begin{fun}
		z_C(\id{C}) & \: & {\left(\slice{\C}{C}\right)}\op & \too & \Set \\[1ex]
		&& (D\ar{f}C) & \mto & {\left(\phi_D\right)}_{f}\\[0.6ex]
		&&(f \al{g} f\c g)&\mto& H(g)\v
	\end{fun}
	So it suffices to prove that such functor is a sheaf. Let $f\:D\to C$ in \C and $R$ a covering sieve on $f\:D\to C$, i.e.\ $R\in J(D)$. Consider then a matching family
	$$\trslicematchN{D'}{D}{C}{g}{}{f}{4}{4} \!\in R \quad \amm{m} \quad X_g\in {(\phi_{D'})}_{f\c g}$$
	on $R$ for $z_C(\id{C})$. We need to show that there is a unique $X\in {(\phi_D)}_f$ such that $H(g)(X)=X_g$ for every $\p{D'\ar{g}D}\in R$. Notice that $m$ is also a matching family on $R\in J(D)$ for $H$, as ${(\phi_{D'})}_{f\c g}\subseteq H(D')$ and the action of $z_C(\id{C})$ on morphisms is given by the action of $H$. Since $\y{C}\:\C\op\to \Set$, also $H\:\C\op\to \Set$. As $H$ is a stack, it then needs to be a sheaf. So there is a unique $X\in H(D)$ such that $H(g)(X)=X_g$ for every $\p{D'\ar{g}D}\in R$. It remains to prove that $\phi_D(X)=f$. Since $\y{C}$ is separated, it suffices to prove that $\phi_D(X)\c g=f\c g$ for every $\p{D'\ar{g}D}\in R$. But by naturality of $\phi$
	$$\phi_D(X)\c g=\phi_{D'}(H(g)(X))=\phi_{D'}(X_g)=f\c g.$$
	We thus conclude that $z_C(\id{C})$ is a sheaf.
\end{proof}

We can now apply \thex\ref{teorfactorization} (based on our theorems of reduction of the study of a 2-classifier to dense generators) to guarantee that we have produced a good 2-classifier in stacks that classifies all discrete opfibrations with small fibres. The following theorem is original.

\begin{teor}\label{teor2classinstacks}
	The 2-natural transformation $\omega_J$ from $1$ to
	\begin{fun}
		{\O_J} & \: & \C\op \hphantom{c}& \too & \hphantom{c}\Cat \\[1.3ex]
		&& C\hphantom{C} & \mto & \Sh[J]{\slice{\C}{C}}\\[0.3ex]
		&& (C\al{f} D) & \mto & -\c {(f\c =)}\op\v[0.5]
	\end{fun}
	that picks the constant at 1 sheaf on every component is a good $2$-classifier in $\St[J]{\C}$ that classifies all discrete opfibrations with small fibres.
\end{teor}
\begin{proof}
	By \thex\ref{teorfactorization}, the restriction $\omega_J$ of the good 2-classifier $\t{\omega}$ in $\m{\C\op}{\Cat}$ along $i\:\St[J]{\C}\ffto \m{\C\op}{\Cat}$ is a good 2-classifier in $\St[J]{\C}$ with respect to the property of having small fibres. We can apply \thex\ref{teorfactorization} thanks to \remx\ref{remstacksformnicesub}, \thex\ref{teor2classinprestacks}, \prox\ref{propOJisastack}, \remx\ref{remrepresentablesinstacks} and \prox\ref{propfactorizationofcharmorinstacks}.
\end{proof}

\begin{rem}
	We can extract from \thex\ref{teorfactorization} and \remx\ref{remconcreterecipeforclasmorph} a recipe for the characteristic morphism $z_J\:F\to \O_J$ of a discrete opfibration $\phi\:G\to F$ in $\St[J]{\C}$ with small fibres.
	\sq[l][6][6][\h[-3]\v[2]\opn{comma}]{G}{1}{F}{\O_J}{}{\phi}{\omega_J}{z_J}
	Recall that we denote as $i$ the inclusion $\St[J]{\C}\cont \m{\C\op}{\Cat}$. We obtain that $z_J$ corresponds to the 2-natural transformation $i(z_J)\:i(F)\to i(\O_J)$ whose component on $C\in \C$ is the functor $i(z_J)_C$ that sends $X\in i(F)(C)$ to the sheaf
	\begin{fun}
		i(z_J)_C(X) & \: & {\left(\slice{\C}{C}\right)}\op & \too & \Set \\[1ex]
		&& (D\ar{f}C) & \mto & {\left(i(\phi)_D\right)}_{i(F)(f)(X)}\\[0.6ex]
		&&(f \al{g} f\c g)&\mto& i(G)(g)\\[0.275ex]
	\end{fun}
\end{rem}

\subsection*{Acknowledgements}

I sincerely thank my PhD supervisor Nicola Gambino for the numerous helpful discussions on the subject of this paper and useful comments on earlier versions of this work. I would also like to thank the anonymous referees for their useful suggestions. Part of this research has been conducted while visiting the University of Manchester.

\bibliographystyle{abbrv}
\bibliography{Bibliography}

\end{document}